\documentclass[12pt]{article}
\usepackage{amsthm, amsmath, amssymb, enumerate}
\usepackage{fullpage}
\usepackage{stfloats}
\usepackage[colorlinks=true]{hyperref}
\usepackage{tikz}
\usetikzlibrary{positioning}

\usepackage[all]{xy}
\usepackage{esint}

\begin{document}
\title{Modular Heights of Unitary Shimura Varieties III: Proof of the Main Theorem}
\author{Ziqi Guo}
\maketitle

\theoremstyle{plain}
\newtheorem{thm}{Theorem}[section]
\newtheorem{theorem}[thm]{Theorem}
\newtheorem{cor}[thm]{Corollary}
\newtheorem{corollary}[thm]{Corollary}
\newtheorem{lem}[thm]{Lemma}
\newtheorem{lemma}[thm]{Lemma}
\newtheorem{pro}[thm]{Proposition}
\newtheorem{proposition}[thm]{Proposition}
\newtheorem{prop}[thm]{Proposition}
\newtheorem{definition}[thm]{Definition}
\newtheorem{assumption}[thm]{Assumption}
\newtheorem{conjecture}[thm]{Conjecture}
\def\avint{\mathop{\,\rlap{-}\!\!\int}\nolimits}

\theoremstyle{remark}
\newtheorem{remark}[thm]{Remark}
\newtheorem{example}[thm]{Example}
\newtheorem{remarks}[thm]{Remarks}
\newtheorem{problem}[thm]{Problem}
\newtheorem{exercise}[thm]{Exercise}
\newtheorem{situation}[thm]{Situation}
\newtheorem{acknowledgment}[thm]{Acknowledgment}

\numberwithin{equation}{subsection}

\newcommand{\ZZ}{\mathbb{Z}}
\newcommand{\CC}{\mathbb{C}}
\newcommand{\QQ}{\mathbb{Q}}
\newcommand{\RR}{\mathbb{R}}
\newcommand{\HH}{\mathcal{H}}     % upper half plane

\newcommand{\ad}{\mathrm{ad}}            % _ad admissible class
\newcommand{\NT}{\mathrm{NT}}
\newcommand{\nonsplit}{\mathrm{nonsplit}}
\newcommand{\Pet}{\mathrm{Pet}}
\newcommand{\Fal}{\mathrm{Fal}}
\newcommand{\Af}{\mathbb{A}_f}

\newcommand{\cs}{{\mathrm{cs}}}

\newcommand{\XU}{X_U}
\newcommand{\Fn}{F_v}
\newcommand{\LU}{L_U}
\newcommand{\LL}{\overline{\mathcal{L}}}
\newcommand{\OF}{\mathcal{O}_F}
\renewcommand{\OE}{\mathcal{O}_E}
\newcommand{\XXU}{\mathcal{X}_U}
\newcommand{\OA}{\underline{\Omega}_\mathcal{A}}
\newcommand{\OU}{\Omega_{\mathcal{X}_U/\mathbb{Z}[\frac{1}{n}]}}
\newcommand{\WA}{\underline{\omega}_\mathcal{A}}
\newcommand{\WU}{\omega_{\mathcal{X}_U/\mathbb{Z}[\frac{1}{n}]}}
\newcommand{\HHom}{\mathcal{H}\mathrm{om}}

\newcommand{\pair}[1]{\langle {#1} \rangle}
\newcommand{\wpair}[1]{\left\{{#1}\right\}}
\newcommand{\wh}{\widehat}
\newcommand{\wt}{\widetilde}

\newcommand\Spf{\mathrm{Spf}}

\newcommand{\lra}{{\longrightarrow}}

\newcommand{\matrixx}[4]
{\left( \begin{array}{cc}
  #1 &  #2  \\
  #3 &  #4  \\
 \end{array}\right)}        % 2*2 matrix

\newcommand{\BA}{{\mathbb {A}}}
\newcommand{\BB}{{\mathbb {B}}}
\newcommand{\BC}{{\mathbb {C}}}
\newcommand{\BD}{{\mathbb {D}}}
\newcommand{\BE}{{\mathbb {E}}}
\newcommand{\BF}{{\mathbb {F}}}
\newcommand{\BG}{{\mathbb {G}}}
\newcommand{\BH}{{\mathbb {H}}}
\newcommand{\BI}{{\mathbb {I}}}
\newcommand{\BJ}{{\mathbb {J}}}
\newcommand{\BK}{{\mathbb {K}}}
\newcommand{\BL}{{\mathbb {L}}}
\newcommand{\BM}{{\mathbb {M}}}
\newcommand{\BN}{{\mathbb {N}}}
\newcommand{\BO}{{\mathbb {O}}}
\newcommand{\BP}{{\mathbb {P}}}
\newcommand{\BQ}{{\mathbb {Q}}}
\newcommand{\BR}{{\mathbb {R}}}
\newcommand{\BS}{{\mathbb {S}}}
\newcommand{\BT}{{\mathbb {T}}}
\newcommand{\BU}{{\mathbb {U}}}
\newcommand{\BV}{{\mathbb {V}}}
\newcommand{\BW}{{\mathbb {W}}}
\newcommand{\BX}{{\mathbb {X}}}
\newcommand{\BY}{{\mathbb {Y}}}
\newcommand{\BZ}{{\mathbb {Z}}}

\newcommand{\CA}{{\mathcal {A}}}
\newcommand{\CB}{{\mathcal {B}}}
\newcommand{\CD}{{\mathcal{D}}}
\newcommand{\CE}{{\mathcal {E}}}
\newcommand{\CF}{{\mathcal {F}}}
\newcommand{\CG}{{\mathcal {G}}}
\newcommand{\CH}{{\mathcal {H}}}
\newcommand{\CI}{{\mathcal {I}}}
\newcommand{\CJ}{{\mathcal {J}}}
\newcommand{\CK}{{\mathcal {K}}}
\newcommand{\CL}{{\mathcal {L}}}
\newcommand{\CM}{{\mathcal {M}}}
\newcommand{\CN}{{\mathcal {N}}}
\newcommand{\CO}{{\mathcal {O}}}
\newcommand{\CP}{{\mathcal {P}}}
\newcommand{\CQ}{{\mathcal {Q}}}
\newcommand{\CR }{{\mathcal {R}}}
\newcommand{\CS}{{\mathcal {S}}}
\newcommand{\CT}{{\mathcal {T}}}
\newcommand{\CU}{{\mathcal {U}}}
\newcommand{\CV}{{\mathcal {V}}}
\newcommand{\CW}{{\mathcal {W}}}
\newcommand{\CX}{{\mathcal {X}}}
\newcommand{\CY}{{\mathcal {Y}}}
\newcommand{\CZ}{{\mathcal {Z}}}

\newcommand{\ab}{{\mathrm{ab}}}
\newcommand{\Ad}{{\mathrm{Ad}}}
\newcommand{\an}{{\mathrm{an}}}
\newcommand{\Aut}{{\mathrm{Aut}}}

\newcommand{\Br}{{\mathrm{Br}}}
\newcommand{\bs}{\backslash}
\newcommand{\bbs}{\|\cdot\|}

\newcommand{\Ch}{{\mathrm{Ch}}}
\newcommand{\cod}{{\mathrm{cod}}}
\newcommand{\cont}{{\mathrm{cont}}}
\newcommand{\cl}{{\mathrm{cl}}}
\newcommand{\criso}{{\mathrm{criso}}}
\newcommand{\de}{{\mathrm{d}}}
\newcommand{\dR}{{\mathrm{dR}}}
\newcommand{\df}{\mathrm{det}^*}
\newcommand{\disc}{{\mathrm{disc}}}
\newcommand{\Div}{{\mathrm{Div}}}
\renewcommand{\div}{{\mathrm{div}}}
\newcommand{\Dh}{\widehat{\mathcal{D}}}
\newcommand{\Ei}{\mathrm{Ei}}
\newcommand{\Eis}{{\mathrm{Eis}}}
\newcommand{\End}{{\mathrm{End}}}
\newcommand{\Fil}{\mathrm{Fil}}
\newcommand{\Frob}{{\mathrm{Frob}}}

\newcommand{\Gal}{{\mathrm{Gal}}}
\newcommand{\GL}{{\mathrm{GL}}}
\newcommand{\GO}{{\mathrm{GO}}}
\newcommand{\GSO}{{\mathrm{GSO}}}
\newcommand{\GSp}{{\mathrm{GSp}}}
\newcommand{\GSpin}{{\mathrm{GSpin}}}
\newcommand{\GU}{{\mathrm{GU}}}
\newcommand{\BGU}{{\mathbb{GU}}}

\newcommand{\Has}{\mathrm{hasse}}
\newcommand{\Hom}{{\mathrm{Hom}}}
\newcommand{\Hol}{{\mathrm{Hol}}}
\newcommand{\HC}{{\mathrm{HC}}}
\newcommand{\id}{\mathrm{id}}
\newcommand{\Img}{{\mathrm{Im}}}
\newcommand{\Ind}{{\mathrm{Ind}}}
\newcommand{\ine}{\mathrm{ine}}
\newcommand{\inv}{{\mathrm{inv}}}
\newcommand{\Isom}{{\mathrm{Isom}}}

\newcommand{\Jac}{{\mathrm{Jac}}}
\newcommand{\JL}{{\mathrm{JL}}}

\newcommand{\Ker}{{\mathrm{Ker}}}
\newcommand{\KS}{{\mathrm{KS}}}

\newcommand{\Lie}{{\mathrm{Lie}}}
\renewcommand{\mod}{\mathrm{mod}}
\newcommand{\mm}{\mathfrak{m}}
\newcommand{\new}{{\mathrm{new}}}
\newcommand{\Nm}{\mathrm{Nm}}
\newcommand{\NS}{{\mathrm{NS}}}

\newcommand{\ord}{{\mathrm{ord}}}
\newcommand{\ol}{\overline}
\newcommand{\otf}{\otimes^*}
\newcommand{\rank}{{\mathrm{rank}}}

\newcommand{\PGL}{{\mathrm{PGL}}}
\newcommand{\PSL}{{\mathrm{PSL}}}
\newcommand{\Pic}{\mathrm{Pic}}
\newcommand{\Prep}{\mathrm{Prep}}
\newcommand{\Proj}{\mathrm{Proj}}
\renewcommand{\Pr}{\mathcal{P}r}
\newcommand{\Picc}{\mathcal{P}ic}

\newcommand{\ram}{\mathrm{ram}}
\renewcommand{\Re}{{\mathrm{Re}}}
\newcommand{\Res}{{\mathrm{Res}}}
\newcommand{\red}{{\mathrm{red}}}
\newcommand{\reg}{{\mathrm{reg}}}
\newcommand{\sm}{{\mathrm{sm}}}
\newcommand{\sing}{{\mathrm{sing}}}
\newcommand{\SL}{\mathrm{SL}}
\newcommand{\SLL}{\widetilde{\mathrm{SL}}}
\newcommand{\SO}{\mathrm{SO}}
\newcommand{\Sp}{\mathrm{Sp}}
\newcommand{\spl}{\mathrm{spl}}
\newcommand{\Sym}{{\mathrm{Sym}}}
\newcommand{\Spec}{\mathrm{Spec}}
\renewcommand{\ss}{\mathrm{ss}}
\newcommand{\tor}{{\mathrm{tor}}}
\newcommand{\tr}{{\mathrm{tr}}}

\newcommand{\ur}{{\mathrm{ur}}}
\newcommand{\U}{\mathrm{U}}
\newcommand{\UU}{\mathrm{U}(1,1)}
\newcommand{\vol}{{\mathrm{vol}}}

\newcommand{\ds}{\displaystyle}

\begin{abstract}
This is the third and the last of a series of three papers, in which we prove a formula expressing the modular height of a unitary Shimura variety over a CM number field in terms of the logarithmic derivative of the Hecke L-function associated with the CM extension. The main idea of our proof is to compare the holomorphic projection of the derivative of a certain mixed Eisenstein-theta series and the arithmetic degree of a generating series of divisors on unitary Shimura varieties.

In this paper, we compute the arithmetic degree of the arithmetic generating series of divisors on unitary Shimura varieties, and then, combining with the results from the first two papers in this series, derive the modular height formula for unitary Shimura varieties.

\end{abstract}

\tableofcontents

\section{Introduction}\label{introduction}
The goal of this series of three papers (\cite{Guo1}, \cite{Guo2} and the current one) is to prove a formula expressing the modular height of a unitary Shimura variety over a CM field in terms of the logarithmic derivative of the Hecke L-function associated with the CM extension. Our work can be viewed as an extension of X. Yuan's work \cite{Yuan1}, which is based on the work of Yuan--Zhang--Zhang \cite{YZZ2} on the Gross--Zagier formula, and the work of Yuan--Zhang \cite{YZ1} on the averaged Colmez conjecture. All these works are in turn inspired by the pioneering work of Gross--Zagier \cite{GZ} and some philosophies of Kudla's program \cite{Kud1,Kud2,Kud3,Kud4}. These works all aim to calculate the arithmetic invariants of Shimura varieties using special values of L-functions.

In our work, we will focus on the generating series of divisors on unitary Shimura varieties and their arithmetic versions, comparing them with the derivative of mixed theta-Eisenstein series. Through a series of specific and intricate computations, we will provide a precise formula for the modular height.

This is the third and the last paper in this series. Its goal is to carry out all the computations on the ``arithmetic side" and, together with the results from the first two papers, to prove the modular height formula. More precisely, we compute the arithmetic degree part of the arithmetic generating series of divisors within the height series defined in \cite{Guo2}, namely its arithmetic intersection number with the arithmetic 1-cycle defined by the self-intersection of the Hodge bundle. Then, we consider the difference between the derivative series defined in \cite{Guo1} and the height series; this is a cuspidal form, which we call the difference series. Finally, by combining with the computational results from \cite{Guo1,Guo2}, we obtain from the constant term of the difference series an inductive formula involving the modular height, thereby proving our main theorem.

To save space and avoid repetition, most of the notation and terminology used in this series of papers can be found in \cite[Sec 2.1, 4.1]{Guo1}, and will not be defined in this paper and \cite{Guo2}.

\subsection{Modular height of the unitary Shimura variety}

Throughout this paper, we always fix a totally real field $F$ of degree $d=[F:\QQ]$ with a fixed infinite place $\iota$. Let $E/F$ be a fixed CM extension, i.e., a totally imaginary quadratic extension of $F$.  We fix an extension of $\iota$ to an embedding $E\to\CC$, which we also denote by $\iota$.

Let $\BV$ be a totally positive definite \textit{incoherent Hermitian space} over $\BA_E$ of dimension $n+1$ with $n>0$, i.e., there does not exist any Hermitian space $V$ over $E$ such that $V\otimes_E \BA_E=\BV$. We also define $V$ to be the \textit{nearby coherent Hermitian space} with respect to $\iota$, i.e., $V$ has signature $(n,1)$ at $\iota$ and $(n+1,0)$ at all other archimedean places.

Let $G=\Res_{F/\QQ}\U(V)$ be a reductive group over $\QQ$, where $\U(V)$ is the unitary group of the Hermitian space $V$. The Hermitian symmetric domain $D$ is defined as follows:
\begin{equation*}
    D=\{z\in\BP(V_{\iota,\CC})\big|q(z)<0\}.
\end{equation*}
Here $q(z)<0$ means that every nonzero vector spanning the line $z$ has negative Hermitian norm. The domain $D$ is connected and carries a $\U(V_\CC)$-invariant complex structure.

Choose a certain Hermitian lattice $\Lambda\subset V$, which is self-dual at each finite place $v$ of $F$ unramified in $E/F$, and is $\varpi_{E_v}$-modular or almost $\varpi_{E_v}$-modular at each ramified place. These definitions are introduced in \cite[Section 2.1, Def 2.2]{Guo1}. Let $U$ be an open compact subgroup of $G(\wh{\QQ})$. We require $U$ to be \textit{maximal} at each finite place $v$, i.e., $U_v\subset\U(V_v)$ is the stabilizer of $\Lambda(\mathcal{O}_{E_v})$. Then we can define a \textit{unitary Shimura variety} $X_U$ over the reflex field $E$, whose $(\CC,\iota)$-points are given by
\begin{equation*}
    X_U(\CC)=G(\QQ)\backslash D\times G(\wh{\QQ})/U.
\end{equation*}
It is smooth over $E$ of dimension $n$.

Under the above complex uniformization, let $L_D$ be the tautological line bundle on $D$, i.e., for each $z\in D$, the fiber $L_{D,z}$ is isomorphic to $\CC z$. This line bundle carries a natural Hermitian metric $h_{L_D}$ such that
\begin{equation*}
    h_{L_D}(s_z)=-q(s_z),
\end{equation*}
for $s_z\in L_{D,z}\cong \CC z$, $z\in D$, which is equivariant under the action of $\U(V_\CC)$. Let $\LU$ be the descent of $L_D\times 1_{G(\wh{\QQ})/U}$, this is an ample line bundle on $\XU$ with $\QQ$-coefficients. We sometimes call $L_U$ the \textit{Hodge bundle} of $X_U$.

For the arithmetic theory we use the normalization of \cite[Section~2]{Guo2}. After the usual base change, the arithmetic intersections are computed on the PEL-type RSZ models and normalized by the auxiliary degrees appearing there. We continue to denote the resulting arithmetic model and Hodge bundle by $\mathcal X_U$ and $\LL_U$. At the relevant finite places this agrees with the corresponding relative RSZ/Rapoport--Zink model.

The \textit{modular height} of $X_U$ with respect to $\LL_U$ is defined to be
\begin{equation*}
    h_{\LL_U}(X_U)=\frac{\wh{\deg}(\LL_U^{n+1})}{\deg(L_U^n)}.
\end{equation*}
Here $\deg(L_U^n)=\deg_L(X_U)$ is the self-intersection number of $L_U$ over the generic fiber $X_U$, which is referred to in some references as the ``geometric volume", and the numerator is the arithmetic self-intersection number on the arithmetic variety $\mathcal{X}_U$ in the setting of Arakelov geometry.

By the projection formula, $h_{\LL_U}(X_U)$ is independent of $U$. In fact, for any split place $v$ of $E/F$, we can allow $U_v$ to be the principal congruence subgroup of the maximal open compact subgroup. However, for other local subgroups $U_v$, $\LL_U$ is less canonical at places $v$ such that $U_v$ is not maximal, we will assume that $U$ is maximal at every nonsplit $v$ in the main theorem and afterwards.

Now, we further require $F\ne \QQ$ primarily to ensure properness. It should be noted that the vast majority of computations in this paper hold true for $F=\QQ$. The goal of this paper is to prove the following formula.
\begin{theorem}\label{Modular height}
    Assume the standing lattice and Schwartz data of \cite[Definition~2.2 and Section~4.1]{Guo1} and the arithmetic local hypotheses of \cite[Assumption~2.4]{Guo2}.  Denote by $\eta$ the quadratic character associated with $E/F$, and let $L_f(\cdot,\eta)$ denote the product of all finite local Hecke $L$-factors, consistently with \cite[Section~2.1]{Guo1}. We have
    \begin{equation*}
    \begin{aligned}
    h_{\LL}(X)=&2\sum_{m=0}^n\frac{L'_f(m+1,\eta^{m+1})}{L_f(m+1,\eta^{m+1})}
    +2(n+1)\log|d_F|+\left(n+2+\left\lfloor\frac n2\right\rfloor\right)\log d_{E/F}\\
    &+[F:\QQ]\left((n+1)\sum_{m=1}^n\frac1m-n-2(n+1)(\gamma+\log(2\pi))\right).
    \end{aligned}
    \end{equation*}
    Here $\gamma$ is the Euler constant, $d_F$ is the discriminant of $F/\QQ$ and $d_{E/F}$ is the norm of the relative discriminant of $E/F$.
\end{theorem}

This formula has a particularly simple form after completing the finite
$L$-functions. Put
\[
    \Gamma_{\CC}(s)=2(2\pi)^{-s}\Gamma(s)
\]
and define
\begin{equation*}
\begin{aligned}
   \Lambda_n^{\mathrm{ht}}(s)={}&|d_F|^{2(n+1)s}d_{E/F}^{(n+2+\lfloor n/2\rfloor)s}\prod_{m=0}^nL_f(m+s,\eta^{m+1})^2\\
   &\times\left[\Gamma_{\CC}(s)^{n+2}\prod_{m=1}^n\Gamma_{\CC}(m+s)\right]^{[F:\QQ]}.
\end{aligned}
\end{equation*}
Then Theorem~\ref{Modular height} is equivalently
\begin{equation*}
    h_{\LL}(X)=\frac{(\Lambda_n^{\mathrm{ht}})'(1)}{\Lambda_n^{\mathrm{ht}}(1)}.
\end{equation*}
Thus the finite $L$-factors, the archimedean factors and the
discriminant terms are combined into a single logarithmic derivative.

Computing the modular height of Shimura varieties is a very important problem in the theory of special value formulas, because modular height formulas always contain the crucial information of the derivative of an L-function at a special value, and there have been many similar works prior to this. First of all, in 2001, Bost (unpublished) and K\"uhn (\cite[Thm 6.1]{Kuh}) proved a modular height formula for modular curves. Later in 2006, Kudla--Rapoport--Yang \cite[Theorem 1.0.5]{KRY} proved a formula for the modular heights of quaternionic Shimura curves over $\QQ$. After that, Bruinier--Burgos Gil--Kuhn \cite{BBK} proved a modular height formula for Hilbert modular surfaces in 2007, and H\"ormann \cite{Ho} proved a modular height formula up to $\log\QQ_{>0}$ for Shimura varieties of orthogonal types over $\QQ$ in 2014. The most closely related conclusion is the following work by Bruinier--Howard \cite{BH} in 2023. When $F=\QQ$, they computed the modular height of the Shimura varieties of unitary similitudes with a different arithmetic line bundle. Furthermore, what needs to be emphasized is that our choices of level group and integral model are \textit{different} from that of \cite{BH}. Roughly speaking, compared with \cite{BH}, the Hermitian lattices at ramified primes in our case are not self-dual, which leads to some differences. Locally, the relative RSZ/Rapoport--Zink models attached to our $\varpi_{E_v}$-modular or almost $\varpi_{E_v}$-modular lattice data exhibit the relevant exotic smoothness.  Accordingly, all global arithmetic statements in the present paper are formulated using the normalized RSZ objects above.

The works introduced above all consider Shimura varieties over $\QQ$ or over an imaginary quadratic field. A natural question is whether they can be generalized to Shimura varieties over more general fields. In fact, this generalization is highly nontrivial, because many techniques that work over $\QQ$ are no longer effective in the general setting; see the explanation in \cite[Section~1.1]{Guo1}. In 2023, X. Yuan \cite{Yuan1} computed the modular height of quaternionic Shimura curves over totally real fields, which is based on the work of Yuan--Zhang--Zhang \cite{YZZ2} on the Gross--Zagier formula, and the work of Yuan--Zhang \cite{YZ1} on the averaged Colmez conjecture. Our proof can be viewed as a higher-dimensional generalization of their proof. In particular, when $n=1$ and the standard self-dual inert lattice datum of the present paper is retained, the set $\Sigma_f$ in the more general curve theorem of \cite[Theorem~1.2]{Guo2} is empty.  Our formula is exactly the specialization after the comparison of arithmetic Hodge bundles in \cite[Section~3]{Guo2}.

Finally, for general Shimura varieties, we can always conjecture that the formula for the modular height of its integral model should, up to some constant differences, align with certain logarithmic derivatives of L-functions.

We should remind the readers that the main terms of our main formula (and also those related formulas mentioned above) are sums of logarithmic derivatives of the Hecke L-functions associated with the quadratic character $\eta$. This is intuitive, as the use of Tamagawa numbers can demonstrate that the geometric self-intersection number $\deg(L_U)$ is essentially the product of special values of these L-functions, i.e.,
\begin{equation*}
    \deg(L_U)=c\cdot\prod_{m=0}^n L_f(m+1,\eta^{m+1}),
\end{equation*}
where $c$ is a constant unrelated to the L-function. See \cite[Theorem A]{BH} for a precise formula in their setting. Similar computations of self-intersection numbers of line bundles on Shimura varieties can be traced back to \cite{Vi}, where the author gives a formula for quaternionic Shimura curve. Then similar to \cite{Yuan1} and \cite{BH}, when we shift to the arithmetic case and consider the modular height, the special values of L-functions transform into corresponding logarithmic derivatives. The relationship between formulas for geometric intersection numbers and arithmetic intersection numbers is similar to the relation between the Gross--Zagier formula and the Waldspurger formula  (as fully explored in \cite{YZZ2}), and is also similar to the relation between the averaged Colmez conjecture and the class number formula (as treated in \cite{YZ1}).

Meanwhile, because the right-hand side of our final formula is completely transparent, this means that compared to \cite{YZZ2,YZ1}, we need to calculate each term accurately, leading to higher computational complexity. In fact, even compared to \cite{Yuan1} and \cite{BH}, we need to engage in more discussion and computations. In contrast to the former, as we are considering higher-dimensional cases, we will encounter some entirely new challenges in both the computation of Whittaker functions and arithmetic intersection numbers, necessitating innovative approaches for resolution. In comparison to the latter, as we mentioned earlier, many techniques used in their work for classical modular forms do not hold in our scenario. Therefore, we also need to invest more effort.

\subsection{Theorem of the arithmetic degree of generating series}\label{generating series introduction}
In our proof of the modular height formula, the generating series formed by the special divisors on the unitary Shimura variety and its arithmetic version are very crucial. It is the core component of the height series on the arithmetic side. In fact, a common idea shared between our proof and the proof of the arithmetic volume formula in \cite{BH} is the use of generating series of divisors to do induction on the dimension of Shimura varieties. From a more general perspective, generating series are important tools for studying special cycles on Shimura varieties.

Regarding the generating series of special divisors on unitary Shimura varieties, we have given a detailed introduction in \cite[Sec 4.2]{Guo2}. Here we give a brief review.  Fix the standard archimedean Gaussians and let $\Phi_f\in\mathcal S(V_f)^U$, with $\Phi=\Phi_\infty\otimes\Phi_f$. The generating series on $X_U$ is
\begin{equation*}
    Z_\Phi(\tau)=\Phi_f(0)[L^\vee]+\sum_{t\in F_+}Z_t(\Phi_f)q^t.
\end{equation*}
Here $q=(e^{2\pi i \tau_1},\cdots,e^{2\pi i \tau_d})$ with $d=[F:\QQ]$ and $\tau=(\tau_k)^d_{k=1}\in\mathcal{H}^d$, and $Z_t$ is the weighted special divisor
\begin{equation*}
    Z_t(\Phi_f)=\sum_{x\in U\backslash V_f,\langle x,x\rangle=t}\Phi_f(x)Z(x)_U
\end{equation*}
with $Z(x)_U$ the Kudla special divisor on $X_U$ associated with $x$. This is a Hilbert modular form of weight $\mathfrak{m}$. Here $\mathfrak{m}=(\mathfrak{m}_v)_{v|\infty}$, where $\mathfrak{m}_v$ is a pair of integers that is defined in \cite[Sec 3.2]{Guo1}.  We can also write this generating series as
\begin{equation*}
    Z(g,\Phi)=r(g)\Phi(0)[L^\vee]+\sum_{t\in F^+}\sum_{y\in U\backslash V_f,\langle y,y\rangle=t}r(g)\Phi(y)Z(y)_U
\end{equation*}
with $g\in\UU(\BA_F)$ and $r(g)$ the Weil representation.

Furthermore, in \cite[Sec 4.3]{Guo2}, using the $\LL$-admissible extension, we also define the arithmetic generating series of arithmetic special divisors on the integral model $\mathcal{X}_U$:
\begin{equation*}
    \wh{\mathcal{Z}}_\Phi(\tau):=\Phi_f(0)[\LL^\vee]+\sum_{t\in F_+}\wh{\mathcal{Z}}_t(\Phi_f)q^t,
\end{equation*}
which can also be written as $\wh{\mathcal{Z}}(g,\Phi)$.

The core significance of the arithmetic generating series is that by taking arithmetic intersection numbers, one obtains a genuine automorphic form. Moreover, these automorphic forms are often closely related to derivatives of Eisenstein series at special points. This line of work is collectively referred to as the ``\textit{arithmetic Siegel--Weil formula}". In this paper, we also prove the following theorem, which perfectly illustrates this philosophy.
\begin{theorem}\label{main theorem of arithmetic Siegel-Weil II}
    Assume the arithmetic hypotheses of Theorem~\ref{Modular height} and use the standard lattice Schwartz function of \cite[Section~4.1]{Guo1}. Let $\wh{\mathcal Z}(g,\Phi)$ be the normalized RSZ arithmetic class associated with the $\LL$-admissible extension of the generating series of divisors $Z(g,\Phi)$, with $g\in\UU(\BA_F)$.  Let $\wh{\mathcal{Z}}_*(g,\Phi)$ be its non-constant part and put $\hat{\xi}=(\LL^n)/\deg_L(X)$.  Let $\Pr' \mathcal{J}'(0,g,\Phi)$ be the quasi-holomorphic projection of the derivative of the Eisenstein series defined in \cite[Theorem 3.4]{Guo1}.

    The comparison is coefficientwise at positive Fourier indices, initially for $g=(g_\infty,1)$ and on the finite compact neighbourhoods used in Section~\ref{Comparison subsection}. Moreover,
    \begin{equation*}
       \Pr' \mathcal{J}'(0,g,\Phi)+\wh{\mathcal{Z}}_*(g,\Phi)\cdot \hat{\xi}
    \end{equation*}
    is a finite sum of non-degenerate pseudo-theta series and non-singular pseudo-Eisenstein series, with explicit nonzero Fourier coefficients.
\end{theorem}
We refer to \cite[Section 2.3]{Guo1} for the complete definition of the pseudo-theta series and pseudo-Eisenstein series, and Section \ref{Comparison subsection} for explicit expressions for those pseudo-theta series and pseudo-Eisenstein series. The local terms appearing in the theorem are computed in Section~\ref{Explicit local terms}. Their normalized heights are compared in Subsection~\ref{ramified matching endpoint subsection}, with proofs in Appendix~\ref{appendix constant comparisons}. These local identities give the coefficients of the pseudo-series in the theorem.

Now we explain the significance behind this conclusion. The arithmetic intersection $\wh{\mathcal{Z}}_*(g,\Phi)\cdot \hat{\xi}$ can be regarded as the arithmetic degree of the arithmetic generating series. While for the first term, note that according to our discussion in the first paper of this series \cite[Section 3.2]{Guo1}, the holomorphic projection of the derivative of mixed theta-Eisenstein series $I(s,g,\Phi)$ can be decomposed into two parts using the quasi-holomorphic projection, i.e., $\Pr' I'(0,g,\Phi)$ and $\Pr' \mathcal{J}'(0,g,\Phi)$. According to \cite[Theorem 1.3]{Guo2}, we have already known that $\Pr' I'(0,g,\Phi)$ is in perfect agreement with the arithmetic intersection number of the arithmetic generating series of divisors with a CM cycle. Thus, together with the above conclusion, we have actually given an explicit comparison between the derivative series defined in \cite{Guo1} and the height series defined in \cite{Guo2}.

The arithmetic Siegel--Weil formula has become a highly active topic in the study of special value formulas in recent years, and it has deep applications. For instance, C. Li and W. Zhang \cite[Theorem 1.3.1]{LZ} prove an identity between the arithmetic degree of Kudla--Rapoport cycles of full rank and the derivative of nonsingular Fourier coefficients of the incoherent Eisenstein series, and Ryan Chen \cite{Chen} (and its corresponding series of articles) extends the identity to corank 1. In fact, using their notation, our Theorem \ref{main theorem of arithmetic Siegel-Weil II} can be described as the arithmetic Siegel--Weil formula of rank 1.

From a more general perspective, the properties of derivatives of L-functions can, in turn, lead to deductions about certain properties of the Chow groups of Shimura varieties. For instance, in \cite{LL1,LL2}, C. Li and Y. Liu prove that when the derivative of a specific L-function is nonzero at a special value, certain Chow cycle of the unitary Shimura variety is non-trivial. This type of work ultimately provides strong evidence for the Birch and Swinnerton--Dyer conjecture, and even the more general Beilinson--Bloch conjecture.

\subsection{Idea of proof}
Now we introduce the organization of this paper and the ideas behind the proof of the main Theorem \ref{Modular height}.

\subsubsection*{Derivative series and height series}
We briefly review the main contents of the first two papers in this series. In \cite[Sec 1.2, Sec 3]{Guo1}, we introduce the derivative series $\Pr I'(0,g,\Phi)$, which is essentially the holomorphic projection of the derivative of a mixed theta-Eisenstein series. More precisely,
\begin{equation*}
    \Pr I'(0,g,\Phi)=\Pr'I'(0,g,\Phi)-\Pr'\mathcal J'(0,g,\Phi).
\end{equation*}
The main theorem of the first paper, \cite[Theorem~1.2]{Guo1}, gives an explicit formula for the derivative series.

In \cite[Sec 1.3, 4]{Guo2}, we introduce the height series
\begin{equation*}
    \big(\wh{\mathcal{Z}}_{*}(g,\Phi)-\frac{\deg_L(Z_*(g,\Phi))}{\deg_L(X)}\LL\big)\cdot(\mathcal{P}-\hat{\xi}),
\end{equation*}
which is essentially the arithmetic intersection number of the arithmetic generating series of divisors. We see that the height series can also be divided into two main parts: the arithmetic intersection number of the arithmetic generating series with the CM cycle $\mathcal{P}$, and its arithmetic intersection number with the Hodge cycle $\hat{\xi}$. The main comparison theorem of the second paper, \cite[Theorem~1.3]{Guo2}, provides an explicit comparison between $\Pr' I'(0,g,\Phi)$ in the derivative series and $\wh{\mathcal{Z}}_*(g,\Phi)\cdot \mathcal{P}$ in the height series.

\subsubsection*{An analogue of the geometric Siegel--Weil formula}
In Section \ref{Arithmetic degree of generating series} and \ref{Explicit local terms} of this paper, we will study and explicitly compute $\wh{\mathcal{Z}}_*(g,\Phi)\cdot \hat{\xi}$ in the height series. The specific approach is as follows. Note that there is a ``geometric Siegel--Weil formula" of generating series on unitary Shimura varieties:
\begin{equation}\label{geometric Siegel--Weil formula}
    \deg_L(Z(g,\Phi)):=Z(g,\Phi)\cdot c_1(L)^{n-1}=-\deg_L(X)E(0,g,\Phi).
\end{equation}
We use the unitary normalization of \cite[(4.2.3)]{Guo2}. Thus, a natural idea is to consider an analogous formula in the arithmetic setting. However, we will find that some extra terms appear; see Theorem \ref{decomposition containing D}. Computing these extra terms is the most crucial part of this paper.

In fact, compared to the computation of this part in \cite[Section 4.7]{Yuan1}, which took only a short subsection, the higher-dimensional case we are facing is much more complicated. We compare the Hodge bundles on the ambient and lower-dimensional Shimura varieties. The local differences are computed directly from the corresponding isogenies and Cartier divisors using the moduli interpretation. In this way, we can express this part of the result in terms of a modular height in one lower dimension, together with some error terms.

\subsubsection*{The comparison}
In Section \ref{Comparison of two series}, we compare the height series
\begin{equation*}
    \big(\wh{\mathcal{Z}}_{*}(g,\Phi)-\frac{\deg_L(Z_*(g,\Phi))}{\deg_L(X)}\LL\big)\cdot(\mathcal{P}-\hat{\xi})
\end{equation*}
with the holomorphic projection of the derivative series $\Pr I'(0,g,\Phi)$. Consider the difference
\begin{equation*}
    \mathcal{D}(g,\Phi)=\Pr I'(0,g,\Phi)+\big(\wh{\mathcal{Z}}_{*}(g,\Phi)-\frac{\deg_L(Z_*(g,\Phi))}{\deg_L(X)}\LL\big)\cdot(\mathcal{P}-\hat{\xi}),
\end{equation*}
this is cuspidal. The computations of \cite[Section~4]{Guo1} and \cite[Section~5]{Guo2}, together with the local calculations of Section~\ref{Explicit local terms}, decompose its nonzero Fourier expansion into non-singular pseudo-theta and pseudo-Eisenstein series. By the key lemma, its nonzero Fourier coefficients can be compared with theta series and Eisenstein series. It is worth noting that in this comparison, some local terms are determined indirectly from lower-dimensional height formulas, rather than by computing the corresponding arithmetic intersections directly.

Finally, taking the constant term of the difference series gives an induction formula for the modular height. Starting from the formula for unitary Shimura curves proved in \cite{Guo2}, we then obtain Theorem~\ref{Modular height} by induction on the dimension.

\subsection*{Acknowledgement}
The author is deeply grateful for the valuable assistance and meticulous guidance provided by Professor Xinyi Yuan. Indeed, it is thanks to his previous work with Shou-Wu Zhang and Wei Zhang that the author has had the privilege to build upon it and make further contributions. He would like to thank his friend Weixiao Lu for much helpful advice. He thanks Ryan Chen, Yinchong Song, Liang Xiao, Yu Luo and Roy Zhao for helpful communication. He also thanks Yifeng Liu and Wei Zhang for some useful suggestions. Finally, the author is grateful to the anonymous referee for many valuable comments and suggestions on this paper.

\section{Arithmetic degree of generating series}\label{Arithmetic degree of generating series}

In this section, we study the arithmetic degree of the arithmetic generating series defined in \cite[Section~4.3]{Guo2}, namely the following arithmetic intersection
\begin{equation}\label{divisor height series}
    { \wh{\mathcal Z}_*(g,\Phi)\cdot\hat\xi=\sum_{\substack{y\in U\backslash V_f\\q(y)\in F_+}}r(g)\Phi(y)\frac{(\LL|_{\mathcal Z(y)})^n}{\deg_L(X)}.}
\end{equation}
in the height series. For convenience, for $q(y)=t$, we write
\begin{equation*}
   r(g)\Phi(y)=W_t^{\mathfrak m}(g_\infty)r(g_f)\Phi_f(y),
\end{equation*}
with the original multiplicities unchanged.

The arrangement of this section is as follows. We first introduce some algebraic results on local Hermitian lattices in Subsection \ref{Algebraic lemmas of local Hermitian lattices}. Then we define the $D$-terms and use the geometric and local Siegel--Weil formulas to write their contribution as a sum of pseudo-Eisenstein series in Subsection \ref{Modular heights of special divisors}. Finally, in Subsection \ref{Adelic Hodge line section}, for each $y$ occurring in the generating series, we compare the restriction of the Hodge bundle to $Z(y)$ with a Hodge bundle coming from a lower-dimensional Shimura variety, which is actually an adelic line bundle. In this way, $D_0(y)$ is written as a sum of local terms together with a lower-dimensional height term. These local terms will be computed explicitly in the next section.

We should remind the reader that we have finished all the computations at archimedean places in \cite[Sec 5.2]{Guo2}. Thus, in this section and the next, the local computations are at non-archimedean places.

\subsection{Algebraic lemmas of local Hermitian lattices}\label{Algebraic lemmas of local Hermitian lattices}
In this subsection, we introduce several algebraic lemmas describing the properties of Hermitian lattices under group actions. More specifically, let $\Lambda\subset \BV^\infty$ be the special Hermitian lattice defined in \cite[Definition 2.2]{Guo1}. At a finite place $v$, write $p_v$  for a chosen uniformizer of $F_v$. For $a\in F_v^\times$, denote by
\begin{equation*}
    \Lambda_{v,a}=\{x\in\Lambda_v|q(x)=a\},
\end{equation*}
where $q$ denotes the Hermitian norm. Our goal is to give an explicit description of  $U_v\backslash\Lambda_{v,a}$ at each finite place $v$. This will be used in our explicit computations later. 

For convenience, a vector $x\in\Lambda_v$ is called \textit{primitive} if it is not divisible in $\Lambda_v$ by any non-unit of $\mathcal O_{E_v}$.

\subsubsection*{Algebraic lemma: split places}
We discuss case by case according to the behavior of $v$ in $E/F$. First, suppose $v$ is split in $E/F$.

\begin{lemma}\label{algebraic lemma}

Suppose $v$ is split in $E/F$, and put $r=v(a)$. The orbit set
$U_v\backslash\Lambda_{v,a}$ depends only on $r$, and
\[
 \big|U_v\backslash\Lambda_{v,a}\big|
 =\begin{cases}(r+1)(r+2)/2,&r\ge0,\\0,&r<0.\end{cases}
\]
In particular, its cardinality is independent of $n\ge1$.

\end{lemma}
\begin{proof}

Integrality of the norm gives the assertion for $r<0$. Assume
$r\ge0$. We give orbit representatives uniformly for $n\ge1$.
 Let $\{e_i\}_{1\le i\le n+1}$ be an orthonormal basis of $\Lambda_v$, i.e.,
    \begin{equation*}
        \langle e_i,e_j\rangle=\delta_{i,j}.
    \end{equation*}
    Then each $x\in \Lambda_v$ can be represented by a coordinate
    \begin{equation*}
        \Big((a_1,b_1),\cdots,(a_{n+1},b_{n+1})\Big)^t,\quad a_i,b_i\in F_v,
    \end{equation*}
    where each $(a_i,b_i)$ represents an element in $E_v\cong F_v\times F_v$. We call $x$ \textit{type} $(s,t)$ if
    \begin{equation}\label{definition of type}
        s=\min\{v(a_i)\}_{1\le i\le n+1},\quad t=\min\{v(b_i)\}_{1\le i\le n+1}.
    \end{equation}
     Each pair $(s,t)$ with $s,t\ge0$ and $s+t\le r$  determines one orbit in $U_v\backslash\Lambda_{v,a}$. On the one hand, if $x_1,x_2\in\Lambda_{v,a}$ are in the same { orbit}, their type must be the same. Indeed, under this orthonormal basis we have an isomorphism
    \begin{equation*}
        U_v\cong\GL_{n+1}(\mathcal{O}_{F_v}),
    \end{equation*}
    such that each element in $U_v$ can be represented by a pair of matrices
    \begin{equation}\label{action of split unitary}
        [X,(X^{-1})^t],\quad X\in\GL_{n+1}(\mathcal{O}_{F_v}),
    \end{equation}
    where $X$ acts on the first coordinate $(a_1,\cdots,a_{n+1})^t$ of $x$ and $(X^{-1})^t$ acts on the second.  The valuations of both coordinate ideals are therefore invariant under $U_v$.

    On the other hand, if $x_1,x_2\in\Lambda_{v,a}$ have the same type, they are in the same { orbit}. In fact, any $x$ with type $(s,t)$ is equivalent to a vector with coordinate
    \begin{equation*}
        \Big((p_v^s,a/p_v^s),(0,p_v^t),(0,0),\cdots,(0,0)\Big)^t.
    \end{equation*}

Move the first split component to $p_v^se_1$, so that $b_1=a/p_v^s$.
Its stabilizer consists of matrices
\[
 X=\begin{pmatrix}1&\lambda\\0&Y\end{pmatrix},\qquad
 b'\longmapsto Y^{-t}(b'-b_1\lambda^t),
 \quad b'=(b_2,\ldots,b_{n+1})^t.
\]
If $r-s>t$, then $p_v^{-t}b'$ is primitive; take $\lambda=0$ and
choose $Y$ sending $b'$ to $p_v^te_1$. If $r-s=t$, take $Y=1$
and $\lambda^t=(b'-p_v^te_1)/b_1$, which is integral.
This finishes the proof.

\end{proof}

Note that according to the above proof, for any $x\in\Lambda_v$, the Hermitian norm $q(x)$ and the type $(s,t)$ determine the orbit of $x$ by $U_v$-action. We also have the following corollary characterizing the orthogonal complement of each $x$.

\begin{corollary}\label{lattice corollary split}
    Denote by $V_x^\perp\subset V$ the orthogonal complement of $x$. Fix an orthonormal basis $\{e_i\}_{1\le i\le n+1}$ of $\Lambda_v$, and assume that $x$ is a linear combination of $e_1$ and $e_2$. Then $\Lambda_v\cap V_{x,v}^\perp$ is generated by $\{x^\perp\}\cup\{e_i\}_{3\le i\le n+1}$, where $x^\perp$ is also a linear combination of $e_1$ and $e_2$ such that $\langle x,x^\perp\rangle=0$ and $q(x)$ is divisible by $q(x^\perp)$. Moreover, under the notation of Lemma \ref{algebraic lemma}, if $x$ is type $(s,t)$, then $v(q(x^\perp))=v(q(x))-s-t$.
\end{corollary}
\begin{proof}
    This proof follows directly from the above lemma. Note that the assumption on $x$ is without loss of generality, because in any $U_v$-orbit of $x$ there exists such a representative.
\end{proof}

\subsubsection*{Algebraic lemma: inert places}
Now we consider the case when $v$ is inert in $E/F$. The following lemma is the counterpart of Lemma \ref{algebraic lemma} for inert places.

\begin{lemma}\label{algebraic lemma inert case}

Suppose $v$ is inert in $E/F$, and put $r=v(a)$. Then
\[
 \big|U_v\backslash\Lambda_{v,a}\big|
 =\begin{cases}1+\lfloor r/2\rfloor,&r\ge0,\\0,&r<0.\end{cases}
\]
The orbit set depends only on $r$, and its cardinality is independent
of $n\ge1$.

\end{lemma}
\begin{proof}

The assertion for $r<0$ follows from integrality. For $r\ge0$,
write $x=p_v^kx_0$ with $x_0$ primitive. Then
$0\le k\le\lfloor r/2\rfloor$ and $v(q(x_0))=s=r-2k$.
We show that primitive vectors of fixed norm form one orbit.

If $a_0=q(x_0)$ is a unit, then $\mathcal O_{E_v}x_0$ is a self-dual line and an orthogonal direct summand of $\Lambda_v$, so
$$
\Lambda_v=\mathcal O_{E_v}x_0\oplus x_0^\perp.
$$
Suppose now that $a_0$ is not a unit. Choose $w\in\Lambda_v$ with $\langle x_0,w\rangle=1$. Using the surjectivity of the trace map and Hensel's lemma, we may replace $w$ by an isotropic vector $w'$ such that $\langle x_0,w'\rangle$ is a unit, and then rescale $w'$ so that $\langle x_0,w'\rangle=1$. Thus $\mathcal O_{E_v}x_0+\mathcal O_{E_v}w'$ is a self-dual plane with Gram matrix
$$
\begin{pmatrix}
a_0&1\\
1&0
\end{pmatrix},
$$
hence an orthogonal direct summand of $\Lambda_v$. Its orthogonal complement is again self-dual.

Finally, every self-dual Hermitian lattice over an unramified extension has an orthonormal basis. Indeed, reduction modulo the uniformizer gives a vector of unit norm; taking its orthogonal complement and arguing by induction gives an orthogonal basis, and surjectivity of the norm map on units allows us to rescale the basis vectors to norm $1$. Hence the orthogonal complements associated with any two primitive vectors of norm $a_0$ are isometric, and therefore an isometry of $\Lambda_v$ carries one vector to the other.

The same rank-two calculation gives an orthogonal complement
$\mathcal O_{E_v}x^\perp\oplus\bigoplus_{i=3}^{n+1}\mathcal O_{E_v}e_i$
with $v(q(x^\perp))=s$ and $q(e_i)=1$. We call $s$ the primitive
type. Hence the types $r,r-2,\ldots,r-2\lfloor r/2\rfloor$ give
exactly the asserted orbits.

\end{proof}

For a vector $y_v$ of primitive type $s$, the same orthogonal decomposition gives
\[
 \Lambda^0_{y,v}
 =\mathcal O_{E_v}p_v^{-\lfloor s/2\rfloor}y_v^\perp
  \oplus\bigoplus_{i=1}^{n-1}\mathcal O_{E_v}e_i
 \subset V_{y,v}^\perp,
\]
where the unit-norm vectors in the complement have been reindexed from $1$ to $n-1$. The first summand has norm valuation $s-2\lfloor s/2\rfloor$. Thus $\Lambda^0_{y,v}$ is self-dual for even $s$ and almost self-dual for odd $s$. Since the norm map $\mathcal O_{E_v}^{\times}\to\mathcal O_{F_v}^{\times}$ is surjective, we may rescale $y_v^\perp$ and assume that $q(y_v^\perp)=p_v^s$.

\subsubsection*{Algebraic lemma: ramified places}
Now we consider the case when $v$ is ramified in $E/F$. Recall that in this case the local Hermitian lattice $\Lambda_v$ is $\varpi_{E_v}$-modular when $2\nmid n$, and is almost $\varpi_{E_v}$-modular when $2| n$. The following lemma is the counterpart of Lemma \ref{algebraic lemma} and \ref{algebraic lemma inert case} at ramified places.

\begin{lemma}\label{algebraic lemma ramified}

Suppose that $v$ is ramified in $E/F$, let $a\in F_v^\times$, and assume that
the indicated norm fiber is nonempty.  By the standing hypotheses of
\cite[Assumption~2.4]{Guo2}, the residue characteristic of $v$ is odd.
\begin{enumerate}
\item If $n$ is odd, then
$$
 \big|U_v\backslash\Lambda_{v,a}\big|=v(a),
 \qquad
 \big|U_v\backslash\Lambda^\vee_{v,a}\big|=v(a)+1.
$$
In particular the orbit set depends only on $v(a)$.

\item If $n$ is even, put
$$
 \epsilon(\Lambda_v):=(-1)^{n/2}d_{\BV}
       \quad\text{in }F_v^\times/\Nm(E_v^\times).
$$
Then
$$
 \big|U_v\backslash\Lambda_{v,a}\big|=
 \begin{cases}
 v(a),&a\notin \epsilon(\Lambda_v)\Nm(E_v^\times),\\
 v(a)+1,&a\in \epsilon(\Lambda_v)\Nm(E_v^\times),
 \end{cases}
$$
and, for the stabilizer of the dual lattice,
$$
 \big|U_v\backslash\Lambda^\vee_{v,a}\big|=
 \begin{cases}
 v(a)+1,&a\notin \epsilon(\Lambda_v)\Nm(E_v^\times),\\
 v(a)+2,&a\in \epsilon(\Lambda_v)\Nm(E_v^\times).
 \end{cases}
$$
Here, as elsewhere in this subsection, the notation $U_v$ for the dual lattice
means its own integral stabilizer.
\end{enumerate}

\end{lemma}

\begin{proof}
Choose $\varpi_{E_v}$ such that $\overline{\varpi}_{E_v}=-\varpi_{E_v}$. By the Jordan decomposition and the classification of modular Hermitian lattices in \cite[Section~4 and Proposition~8.1]{Jac}, if $n$ is odd, then $\Lambda_v$ has an isotropic basis
$$
\{X_1,Y_1,\ldots,X_{\frac{n+1}{2}},Y_{\frac{n+1}{2}}\},
\qquad
\langle X_i,Y_j\rangle=\varpi_{E_v}\delta_{ij}.
$$
If $n$ is even, we may write
$$
\Lambda_v=\mathcal O_{E_v}e\perp
\bigoplus_{i=1}^{n/2}
(\mathcal O_{E_v}X_i\oplus\mathcal O_{E_v}Y_i),
$$
where $q(e)\in\epsilon(\Lambda_v)\Nm(\mathcal O_{E_v}^{\times})$ and
$\langle X_i,Y_j\rangle=\varpi_{E_v}\delta_{ij}$.

We first recall a simple fact. Suppose that $x$ and $x'$ have the same norm. If
$$
\langle\Lambda_v,x-x'\rangle
 \subset
\langle x,x-x'\rangle\mathcal O_{E_v},
$$
then the corresponding symmetry preserves $\Lambda_v$ and sends $x$ to $x'$; see \cite[Lemma~3.2]{BHM}.

We claim that primitive vectors of the same positive norm form one orbit. In the even-dimensional case, a primitive vector of positive norm valuation has its $e$-coefficient in $\varpi_{E_v}\mathcal O_{E_v}$, so its component in the hyperbolic part is primitive. For two primitive vectors $x,x'$ of the same positive norm, choose a vector $z$ in the hyperbolic part such that
$$
\langle x,z\rangle,\ \langle x',z\rangle
\in\varpi_{E_v}\mathcal O_{E_v}^{\times}.
$$
Such a choice is always possible, since after division by $\varpi_{E_v}$ and reduction modulo $\varpi_{E_v}$, the vanishing of either pairing defines a proper hyperplane. Choose $w$ in the hyperbolic part with
$\langle z,w\rangle=\varpi_{E_v}$. For $t\in\mathcal O_{F_v}$,
$$
q(z+\varpi_{E_v}tw)
=q(z)-2\varpi_{E_v}^2t-\varpi_{E_v}^2q(w)t^2.
$$
After dividing by $\varpi_{E_v}^2$, the derivative with respect to $t$ is a unit modulo $p_v$. Hensel's lemma therefore gives $t$ such that
$q(z+\varpi_{E_v}tw)=q(x)$. The two pairings with $x$ and $x'$ are still generators of
$\varpi_{E_v}\mathcal O_{E_v}$. Applying the above fact first to $x$ and then to $x'$ shows that $x$ and $x'$ are equivalent. In particular, a primitive vector of positive norm $b$ may be represented in the first hyperbolic plane by
$$
X_1-\frac{b}{2\varpi_{E_v}}Y_1.
$$

The same argument for the dual hyperbolic lattice gives one primitive orbit for every integral norm, with representative
$$
\varpi_{E_v}^{-1}X_1+\frac b2Y_1.
$$

When $n$ is even, we also need primitive vectors of unit norm. Such a vector has unit $e$-coefficient, so its norm belongs to
$q(e)\Nm(\mathcal O_{E_v}^{\times})$. Conversely, every unit in this class is represented. If $x$ and $x'$ have the same unit norm $b$, then at least one of
$b-\langle x,x'\rangle$ and $b+\langle x,x'\rangle$ is a unit, since the residue characteristic is odd. Applying the same symmetry to $x,x'$ or to $x,-x'$ shows that these vectors also form one orbit.

Now suppose that $n$ is odd. For $x\in\Lambda_{v,a}$ write
$$
x=\varpi_{E_v}^k x_0,
\qquad x_0\ \text{primitive}.
$$
Then
$$
v(q(x_0))=v(a)-k.
$$
A primitive vector in $\Lambda_v$ has positive norm valuation, so
$k=0,\ldots,v(a)-1$. For each $k$ there is exactly one orbit by the preceding argument, while different $k$ give different orbits since divisibility by $\varpi_{E_v}$ is preserved by $U_v$. Hence
$$
\big|U_v\backslash\Lambda_{v,a}\big|=v(a).
$$
Since $\Lambda_v^\vee=\varpi_{E_v}^{-1}\Lambda_v$, multiplication by $\varpi_{E_v}$ reduces the dual problem to the same calculation with norm valuation $v(a)+1$. Thus
$$
\big|U_v\backslash\Lambda^\vee_{v,a}\big|=v(a)+1.
$$

Suppose that $n$ is even. The same argument gives one orbit for each positive primitive norm valuation, hence $v(a)$ orbits in total. There is one additional orbit exactly when the primitive part has unit norm, which occurs precisely when
$$
a\in\epsilon(\Lambda_v)\Nm(E_v^\times).
$$
The argument above shows that this orbit is unique. Therefore
$$
\big|U_v\backslash\Lambda_{v,a}\big|=
\begin{cases}
v(a),&a\notin\epsilon(\Lambda_v)\Nm(E_v^\times),\\
v(a)+1,&a\in\epsilon(\Lambda_v)\Nm(E_v^\times).
\end{cases}
$$

Finally,
$$
\Lambda_v^\vee=
\mathcal O_{E_v}e\perp
\bigoplus_{i=1}^{n/2}
(\varpi_{E_v}^{-1}\mathcal O_{E_v}X_i
 \oplus\varpi_{E_v}^{-1}\mathcal O_{E_v}Y_i).
$$
The dual hyperbolic part gives one orbit for each primitive norm valuation
$0,\ldots,v(a)$, hence $v(a)+1$ orbits. In the matching norm class there is one further orbit coming from the self-dual line $\mathcal O_{E_v}e$. It is different from the valuation-$0$ hyperbolic orbit, since the pairing ideals of primitive representatives are respectively
$\mathcal O_{E_v}$ and $\varpi_{E_v}^{-1}\mathcal O_{E_v}$. Hence
$$
\big|U_v\backslash\Lambda^\vee_{v,a}\big|=
\begin{cases}
v(a)+1,&a\notin\epsilon(\Lambda_v)\Nm(E_v^\times),\\
v(a)+2,&a\in\epsilon(\Lambda_v)\Nm(E_v^\times).
\end{cases}
$$
This proves the lemma.
\end{proof}

\begin{corollary}\label{algebraic corollary ramified}
Let $r=v(a)$ and let $x\in\Lambda_{v,a}$ be of type $s$. Using the basis in the proof of Lemma~\ref{algebraic lemma ramified}, we may choose $e_x\in \mathcal{O}_{E_v}^\times x^\perp$ such that
$$
q(e_x)=(-1)^{r-s}\varpi_{E_v}^{-2r+2s+2}a,
$$
where $x^\perp$ has the same meaning as in Corollary \ref{lattice corollary split}. In particular, $v(q(e_x))=s+1$.

If $n$ is odd, then
$$
\Lambda_v\cap x^\perp
=\mathcal O_{E_v}e_x\perp
\bigoplus_{i=1}^{(n-1)/2}
(\mathcal O_{E_v}X_i\oplus\mathcal O_{E_v}Y_i).
$$

Suppose that $n$ is even. For the type-$s$ orbit in
$\epsilon(\Lambda_v)\Nm(E_v^\times)$, we write $B_s^+$, and for the type-$s$ orbit in the other norm class, we write $B_s^-$. For either orbit,
$$
\Lambda_v\cap x^\perp
=\mathcal O_{E_v}e\perp\mathcal O_{E_v}e_x\perp
\bigoplus_{i=1}^{(n-2)/2}
(\mathcal O_{E_v}X_i\oplus\mathcal O_{E_v}Y_i).
$$
In the class $\epsilon(\Lambda_v)\Nm(E_v^\times)$, there is also an orbit represented by a vector in the self-dual line $\mathcal O_{E_v}e$. Its orthogonal complement is $\varpi_{E_v}$-modular of rank $n$.

\end{corollary}

\begin{proof}
Take the representatives used in the proof of Lemma~\ref{algebraic lemma ramified}. A direct calculation in the first hyperbolic plane gives the vector $e_x$ and its norm. The remaining orthogonal summands are unchanged. The assertion for the orbit from the extra line when $n$ is even is immediate from the decomposition of $\Lambda_v$ in the even-dimensional case. Note that it is not hard to write the corresponding statements for $\Lambda_v^\vee$.
\end{proof}

\subsection{Modular heights of special divisors}\label{Modular heights of special divisors}

In this subsection, we first separate the arithmetic degree of the generating series into the geometric term and the $D$-term. The geometric term is determined by the geometric Siegel--Weil formula \eqref{geometric Siegel--Weil formula}. For the $D$-term, we compare the normalized heights of the special divisors and reduce the difference to local intersection terms. This follows the discussion in \cite[Section~4.7]{Yuan1}; see also \cite[Section~7.3.2]{YZZ2}.

\subsubsection*{The definition of \texorpdfstring{$D$}{D}-terms}

We first fix the special divisor used in the induction. As in \cite[Corollary~4.7]{Guo1}, we fix a special vector in the dual lattice.

\begin{definition}\label{y_0 vector}
Let $y_0\in \Lambda^\vee$ be a vector such that
\begin{enumerate}
\item $q(y_0)$ has valuation $0$ at every finite place.
\item If $n$ is even, we further require that, at every ramified place $v$,
$q(y_0)\in\epsilon(\Lambda_v)\Nm(E_v^\times)$ and that $y_{0,v}$ is represented by a vector in the self-dual line of $\Lambda_v^\vee$.
\end{enumerate}
\end{definition}

When such a vector $y_0$ exists, $Z(y_0)$ has the standard lattice datum in dimension $n-1$ and gives the lower-dimensional term in the induction. More generally, we only require $\deg_L Z(y_0)\ne0$, and keep its normalized height in the formulas below. Appendix~\ref{appendix auxiliary heights} shows that it is enough to assume the standard local conditions outside a finite set. 

Let $y$  be a positive vector with $\deg_L Z(y)\ne0$. Define
\begin{equation}\label{D definition}
 D(y):=\mathcal Z(y)\cdot\hat\xi
 -\frac{\deg_L Z(y)}{\deg_L Z(y_0)}
       \bigl(\mathcal Z(y_0)\cdot\hat\xi\bigr),
\end{equation}
and
\begin{equation}\label{D0 global definition}
 D_0(y):=\frac{D(y)}{\deg_L Z(y)}
 =\frac{\mathcal Z(y)\cdot\hat\xi}{\deg_L Z(y)}
 -\frac{\mathcal Z(y_0)\cdot\hat\xi}{\deg_L Z(y_0)}.
\end{equation}
By the normalization in the definition of an $\LL$-admissible extension, the Green function satisfies
\begin{equation*}
 \int_{X_\sigma(\CC)}g_{D,\sigma}\,c_1(L)^n=0
\end{equation*}
on every connected archimedean component; see \cite[Section~4.1]{Guo2}. Thus the Green correction makes no contribution to this arithmetic degree. 

We have the following decomposition of the arithmetic degree of the generating series.

\begin{theorem}\label{decomposition containing D}
There exists a decomposition
\begin{equation}\label{height decomposition containing D}
 { \wh{\mathcal Z}_*(g,\Phi)\cdot\hat\xi}
 =\sum_{\substack{y\in U\backslash V_f\\q(y)\in F_+}}
 r(g)\Phi(y)
 \left(
 \frac{\deg_L Z(y)}{\deg_L(X)}h_{\LL}(Z(y_0))+D(y)
 \right),
\end{equation}
where
\begin{equation*}
 h_{\LL}(Z(y_0))
 =\frac{\mathcal Z(y_0)\cdot\LL^n}{\deg_L Z(y_0)}.
\end{equation*}
\end{theorem}

\begin{proof}
For every $y$ occurring in the generating series, \eqref{D definition} gives
\begin{equation*}
 \mathcal Z(y)\cdot\hat\xi
 =\frac{\deg_L Z(y)}{\deg_L Z(y_0)}
   \bigl(\mathcal Z(y_0)\cdot\hat\xi\bigr)+D(y).
\end{equation*}
The result follows by summing over $U\backslash V_f$ and using the definition of $h_{\LL}(Z(y_0))$.
\end{proof}

Applying the geometric Siegel--Weil formula \ref{geometric Siegel--Weil formula}, we have
\begin{equation}\label{geometric term in D decomposition}
 \sum_{\substack{y\in U\backslash V_f\\q(y)\in F_+}}
 r(g)\Phi(y)\frac{\deg_L Z(y)}{\deg_L(X)}
 =-E_*(0,g,\Phi).
\end{equation}
When $y_0$ satisfies Definition~\ref{y_0 vector}, the first term on the right hand side of \eqref{height decomposition containing D} is known by induction on $n$. It remains to compute the contribution of $D(y)$.

For later use, put $\U_y=\U(V_y^\perp)$ and
\begin{equation*}
 U_y=U\cap\U_y(\Af).
\end{equation*}
With the Haar measures of \cite[Section~2.1]{Guo1}, the degree formula of homogeneous spaces is
\begin{equation}\label{geometric degree}
 \frac{\deg_L(X)}{\deg_L Z(y)}=\frac{\vol(U_y)\vol(\U_{y,\infty})}{\vol(U)\vol(\U_\infty)};
\end{equation}
see \cite[Proposition~4.2]{YZZ2}. Writing $U=U^vU_v$ and $U_y=U_y^vU_{y,v}$ gives
\begin{equation}\label{degree local factorization}
    \frac{\deg_L Z(y)}{\deg_L X}=\frac{\vol(U^v)}{\vol(U_y^v)}\frac{\vol(U_v)}{\vol(U_{y,v})}\frac{\vol(\U_\infty)}{\vol(\U_{y,\infty})}.
\end{equation}

\subsubsection*{Pseudo-Eisenstein series from the \texorpdfstring{$D$}{D}-term}

We first record the decomposition of the $D$-term that will be proved in Subsection~\ref{Adelic Hodge line section}. Later in Proposition~\ref{D and adelic Hodge line} we will have
\begin{equation}\label{additive property}
     D_0(y)=\sum_{v\nmid\infty}D_{0,v}(y_v)+\frac{c_\sigma}{\deg_L(X)}.
\end{equation}
Here $c_\sigma$ is defined in \eqref{reference height constant}, which will be determined in Section~\ref{ramified matching endpoint subsection}. Note that for a fixed $y$, only finitely many local terms are nonzero.

For a finite place $v$, define
\begin{equation}\label{local F definition}
 \mathcal F_\Phi^{(v)}(g)
 :=-\sum_{\substack{y\in U\backslash V_f\\q(y)\in F_+}}
 r(g)\Phi(y)\,\deg_L Z(y)\,D_{0,v}(y_v),
 \qquad g\in\UU(\BA_F).
\end{equation}
Since $D(y)=\deg_L Z(y)D_0(y)$, equation \eqref{additive property} gives
\begin{equation}\label{D decomposition into F}
\begin{aligned}
 \sum_{\substack{y\in U\backslash V_f\\q(y)\in F_+}}
 r(g)\Phi(y)D(y)
 ={}&-\sum_{v\nmid\infty}\mathcal F_\Phi^{(v)}(g)\\
 &+\sum_{\substack{y\in U\backslash V_f\\q(y)\in F_+}}
 r(g)\Phi(y)\frac{\deg_L Z(y)}{\deg_L(X)}c_\sigma.
\end{aligned}
\end{equation}
Thus Theorem~\ref{decomposition containing D} becomes
\begin{equation}\label{height series with local F}
\begin{aligned}
 { \wh{\mathcal Z}_*(g,\Phi)\cdot\hat\xi}
 ={}&-E_*(0,g,\Phi)h_{\LL}(Z(y_0))
       -\sum_{v\nmid\infty}\mathcal F_\Phi^{(v)}(g)\\
 &+\sum_{\substack{y\in U\backslash V_f\\q(y)\in F_+}}
 r(g)\Phi(y)\frac{\deg_L Z(y)}{\deg_L(X)}c_\sigma.
\end{aligned}
\end{equation}

We now factor the $v$-term. For $a\in F^\times$, put
\[
 V_{f,a}=\{y\in V_f:q(y)=a\},\qquad
 V_{v,a}=\{y_v\in V_v:q(y_v)=a\}.
\]
Recall the local Whittaker functions in \cite[Section~2.2 and Lemma~4.5]{Guo1}. With the same self-dual measures, the local Siegel--Weil formula is
\begin{equation}\label{raw local Siegel Weil factor}
 |a|_v^n\sum_{y_v\in U_v\backslash V_{v,a}}
 r(g_v)\Phi_v(y_v)\frac{\vol(U_v)}{\vol(U_{y,v})}
 =\gamma(\BV_v,q)^{-1}W_{a,v}(0,g_v,\Phi_v).
\end{equation}
Here $\gamma(\BV_v,q)$ is the Weil index of the quadratic norm, with the additive character fixed in \cite{Guo1}. In particular it is $1$ at a split place and $(-1)^{(n+1)v(d_v)}$ at an inert place. At ramified places it is $1$ for odd $n$; for even $n$ it is kept in the formulas.

Since $\Phi=\Phi^v\otimes\Phi_v$, equation \eqref{degree local factorization} gives
\begin{equation}\label{F before SW factorization}
\begin{aligned}
 \mathcal F_\Phi^{(v)}(g)
 =-\deg_L(X)\sum_{a\in F^\times}
 &\frac{\vol(\U_\infty)}{\vol(\U_{y,\infty})}\\
 &\cdot\left(\sum_{y^v\in U^v\backslash V_{f,a}^v}
 r(g^v)\Phi^v(y^v)\frac{\vol(U^v)}{\vol(U_y^v)}\right)\\
 &\cdot\left(\sum_{y_v\in U_v\backslash V_{v,a}}
 r(g_v)\Phi_v(y_v)\frac{\vol(U_v)}{\vol(U_{y,v})}
 D_{0,v}(y_v)\right).
\end{aligned}
\end{equation}
For fixed $a$, the archimedean volume ratio is constant on these orbits. Combining \eqref{raw local Siegel Weil factor} with the global geometric Siegel--Weil formula gives
\begin{equation}\label{away v Siegel Weil factor}
\begin{aligned}
 \sum_{y^v\in U^v\backslash V_{f,a}^v}
 r(g^v)\Phi^v(y^v)\frac{\vol(U^v)}{\vol(U_y^v)}
 ={}&-\gamma(\BV_v,q)|a|_v^n\\
 &\cdot\frac{\vol(\U_{y,\infty})}{\vol(\U_\infty)}
 W_a^v(0,g^v,\Phi^v).
\end{aligned}
\end{equation}
Here we use the conventions of \cite[Section~2.2]{Guo1} and the local Siegel--Weil formula of \cite[Lemma~3.1.1]{Qiu}. Define
\begin{equation}\label{f series}
\begin{aligned}
 f_{\Phi_v,a}(g_v)
 ={}&\gamma(\BV_v,q)\deg_L(X)|a|_v^n\\
 &\cdot\sum_{y_v\in U_v\backslash V_{v,a}}
 r(g_v)\Phi_v(y_v)\frac{\vol(U_v)}{\vol(U_{y,v})}
 D_{0,v}(y_v).
\end{aligned}
\end{equation}
Substitution in \eqref{F before SW factorization} cancels the archimedean volume ratios and gives
\begin{equation}\label{F series}
 \mathcal F_\Phi^{(v)}(g)
 =\sum_{a\in F^\times}
 W_a^v(0,g^v,\Phi^v)f_{\Phi_v,a}(g_v).
\end{equation}
Thus the Weil factor is contained in $f_{\Phi_v,a}$ and does not affect $D_{0,v}$.

\subsection{Adelic Hodge line bundles}\label{Adelic Hodge line section}

We now turn to the computation of the $D$-term. By definition, $D_0(y)$ is the difference between the normalized Hodge intersections of $Z(y)$ and $Z(y_0)$. To compute this difference, we compare the restriction of the Hodge line bundle to $Z(y)$ with the Hodge line bundle on the lower-dimensional Shimura variety defined by $V_y^\perp$. Their difference is local, while the latter is related to the modular height in dimension $n-1$.

We make this comparison using adelic line bundles in the sense of Yuan--Zhang \cite[Sections~2.5 and~4.1]{YZ2}. The Hodge line bundles and their metrics are those fixed in \cite[Sections~2.3 and~4.1]{Guo2}.

\subsubsection*{The adelic Hodge line bundle on \texorpdfstring{$Z(y)$}{Z(y)}}

To define the lower-dimensional Hodge bundle used in this comparison, we first describe the lattice in $V_y^\perp$. At an inert place, this lattice is self-dual precisely when $v(q(y))$ is even. This motivates the following definition.

\begin{definition}\label{split vector definition}
A vector $y\in\Lambda_f^\times$ is called split if $v(q(y))$ is even at every inert place $v$.
\end{definition}

Let $y$ be a positive finite adelic vector. At each finite place, we first reduce to the primitive case and apply the local descriptions in Subsection~\ref{Algebraic lemmas of local Hermitian lattices}. At a split place, keep the adapted basis
\[
 \{y'_v,y_v^\perp,a_ve_1,\ldots,a_ve_{n-1}\},
 \qquad \Nm(a_v)=q(y'_v),\qquad q(e_i)=1,
\]
used in Corollary~\ref{lattice corollary split}.

At an inert place, let $s$ be the primitive type. Put
\begin{equation}\label{lower lattice unramified definition}
 \Lambda^0_{y,v}=
 \begin{cases}
 \displaystyle\mathcal O_{E_v}a_v^{-1}y_v^\perp
       \oplus\bigoplus_{i=1}^{n-1}\mathcal O_{E_v}e_i,
       &v\text{ split},\\[2mm]
 \displaystyle\mathcal O_{E_v}p_v^{-\lfloor s/2\rfloor}y_v^\perp
       \oplus\bigoplus_{i=1}^{n-1}\mathcal O_{E_v}e_i,
       &v\text{ inert}.
 \end{cases}
\end{equation}
The inert lattice is self-dual for even $s$ and almost self-dual for odd $s$, as shown after Lemma~\ref{algebraic lemma inert case}. 

At a ramified place, we use Corollary~\ref{algebraic corollary ramified}. If $n$ is odd, we take the almost $\varpi_{E_v}$-modular lattice in $V_{y,v}^\perp$ given there. The situation is subtle when $n$ is even. If $y_v$ is represented by a vector in the self-dual line, we take its $\varpi_{E_v}$-modular orthogonal complement. For the other orbits, let $e_s$ be the vector in Corollary~\ref{algebraic corollary ramified} and put
$$
w_s=\varpi_{E_v}^{-s-1}e_s.
$$
Then $w_s$ and the unit vector $e$ span a self-dual lattice of rank $2$. In the opposite norm class we keep this lattice together with the remaining $\varpi_{E_v}$-modular summands. In the matching norm class this plane is split; for a hyperbolic basis $u,v$, we take
$$
\mathcal O_{E_v}u\oplus\mathcal O_{E_v}\varpi_{E_v}v
\quad\text{or}\quad
\mathcal O_{E_v}\varpi_{E_v}u\oplus\mathcal O_{E_v}v.
$$
These choices may be made compatibly as $s$ varies. 

As a consequence, for all but finitely many $v$, the resulting lattice is the standard one in dimension $n$. They define
\[
 \Lambda_y^0=\prod_{v\nmid\infty}\Lambda^0_{y,v},\qquad
 U_y^0=\operatorname{Stab}(\Lambda_y^0)
       \subset\U(V_y^\perp)(\Af).
\]
After replacing $Z(y)$ by the finite cover corresponding to $U_y\cap U_y^0$, which we still denote by $Z(y)$, we have a finite morphism
\begin{equation}\label{lower finite morphism}
    f_y:Z(y)\longrightarrow X_{U_y^0}.
\end{equation}
Also note that all comparisons below are made after this finite pull-back.

We now equip the lower-dimensional Hodge bundle with a metric. For the self-dual and $\varpi_{E_v}$-modular lattices chosen above, we use the Hodge bundle on the corresponding integral model as in \cite[Sections~2.3 and~4.1]{Guo2}. At an inert place of odd type, we use the metric coming from a fixed type-$1$ vector. If $n$ is even and $v$ is ramified, for the opposite norm class we similarly use the metric coming from a fixed type-$0$ vector in that class. At the archimedean places, we use the same tautological metric as on $L|_{Z(y)}$.

Indeed, the lattice $\Lambda_y^0$ determines the level $U_y^0$ and hence the corresponding integral model of $X_{U_y^0}$. The Hodge bundle on this integral model therefore gives a metric on its generic fiber. Pulling these metrics back by $f_y$ gives an adelic metric on the Hodge bundle over $Z(y)$.

\begin{definition}\label{adelic lower Hodge line definition}
For $y$, let $\LL_{\mathcal Z(y)}$ be the adelic Hodge bundle on $Z(y)$ defined by the local metrics chosen above. If all the local metrics are induced from the Hodge line bundle on $X_{U_y^0}$, then
\begin{equation}\label{adelic lower Hodge line}
\LL_{\mathcal Z(y)}=f_y^*\LL_{U_y^0}.
\end{equation}
\end{definition}

The pull-back in \eqref{adelic lower Hodge line} is the pull-back of adelic line bundles defined in \cite[Section~2.5.5]{YZ2}. By construction, $\LL_{\mathcal Z(y)}$ has underlying generic line bundle $L|_{Z(y)}$ and the same metric at the archimedean places. The local metrics chosen above are compatible with further finite pull-back; for the two exceptional cases this follows from Lemmas~\ref{inert odd adjacent correspondence} and \ref{ramified marked Hodge comparison}.

When the lower-dimensional lattice is covered by the induction hypothesis, the projection formula gives
\begin{equation}\label{adelic lower projection formula}
    \frac{\LL_{\mathcal Z(y)}^n}{\deg_L Z(y)}=\frac{\LL_{U_y^0}^n}{\deg_L X_{U_y^0}}.
\end{equation}
If $y_0$ satisfies Definition~\ref{y_0 vector}, the induction hypothesis identifies this common value with $h_{\LL}(Z(y_0))$.
For the two exceptional cases, the metrics for different $s$ are obtained by pull-back from the fixed metrics chosen above. Moreover, a finite pull-back multiplies both the arithmetic self-intersection and the geometric degree by the degree of the map. Hence
$$
\frac{\LL_{\mathcal Z(y)}^n}{\deg_L Z(y)}
$$
is independent of $s$ and of the finite level. We may therefore put
\begin{equation}\label{reference height constant}
 c_\sigma
 =\frac{\LL_{\mathcal Z(y)}^n}{\deg_L Z(y)}
   -h_{\LL}(Z(y_0)).
\end{equation}
Equivalently,
\begin{equation*}
 \frac{\LL_{\mathcal Z(y)}^n}{\deg_L Z(y)} =h_{\LL}(Z(y_0))+c_\sigma.
\end{equation*}

\subsubsection*{The relation with the \texorpdfstring{$D$}{D}-term}

We now express the original global $D_0(y)$ using the adelic Hodge bundle defined above. 

\begin{proposition}\label{D and adelic Hodge line}
Let $y$ be an admissible positive vector with $\deg_L Z(y)\ne0$. With the choices above, we have
\begin{equation}\label{D adelic relation}
 \deg_L(X)D_0(y)
 =\frac{(\LL|_{\mathcal Z(y)})^n-\LL_{\mathcal Z(y)}^n}
        {\deg_L Z(y)}+c_\sigma.
\end{equation}
\end{proposition}

\begin{proof}
By the normalization in \cite[Section~4.1]{Guo2}, the vertical part has zero intersection with $\LL^n$, and the Green function has zero integral against $c_1(L)^n$. Thus
\[
 \mathcal Z(y)\cdot\hat\xi
 =\frac{(\LL|_{\mathcal Z(y)})^n}{\deg_L(X)}.
\]
The same identity holds for $y_0$. Substituting these two identities into \eqref{D0 global definition} and using \eqref{reference height constant} gives \eqref{D adelic relation}.
\end{proof}

For a finite place $v$, let $D_{0,v}(y)$ be the $v$-adic contribution of the first term in \eqref{D adelic relation}, divided by $\deg_L(X)$. Explicitly,
\begin{equation}\label{local D0 adelic definition}
\begin{aligned}
 D_{0,v}(y)=\frac{1}{\deg_L(X)\deg_L Z(y)}
 \Bigg(&\bigl(\LL|_{\mathcal Z(y)}-\LL_{\mathcal Z(y)}\bigr)\\
 &\cdot\sum_{j=0}^{n-1}
 (\LL|_{\mathcal Z(y)})^{n-1-j}\LL_{\mathcal Z(y)}^j\Bigg)_v.
\end{aligned}
\end{equation}
The metrics agree at infinity, so summing \eqref{local D0 adelic definition} gives \eqref{additive property}. We will compute this term in the next section.

\section{Explicit local terms}\label{Explicit local terms}

We compute the local terms in \eqref{local D0 adelic definition} and \eqref{f series} at split, inert and ramified places. We first compare the restriction of the Hodge bundle with the adelic line bundle defined in Subsection~\ref{Adelic Hodge line section} explicitly, since their intersection numbers give $D_{0,v}$ by \eqref{local D0 adelic definition}. Then we sum over the orbits in \eqref{f series} and compare the result with the Whittaker derivative. It should be pointed out that the explicit computations at the three types of places are in fact quite different.

Throughout this section we always assume $n\ge2$, since the case of $n=1$ used to start the induction is given by \cite[Theorem~1.2]{Guo2}.

\subsection{Explicit local computations at split places}

Suppose that $v$ is split in $E/F$. We take a vector $y$ in the generating series and write $a=q(y)$ and $r=v(a)$. If $y_v$ has type $(s,t)$ in Lemma~\ref{algebraic lemma}, then
\begin{equation}\label{split depth definition}
 0\le s+t\le r.
\end{equation}
We first compute $D_{0,v}(y_v)$ from \eqref{local D0 adelic definition}.

\begin{proposition}\label{split local D0 formula}
Suppose that $v$ is split in $E/F$. For $y_v$ of type $(s,t)$, we have
\begin{equation}\label{split D0 formula}
 \deg_L(X)D_{0,v}(y_v)=(r-s-t)\log N_v.
\end{equation}
\end{proposition}

\begin{proof}
We express the change of lattice as successive quotients by subgroups of order $N_v$. For one quotient, its derivative gives the difference of the Hodge bundles, and Frobenius gives the required intersections. We then add the contributions of these quotients.

\noindent\textbf{Step 1.} 
In this step we compute the difference of the Hodge bundles when $r-s-t=1$. We first divide the two split components of $y_v$ by their respective common powers of $p_v$. This preserves the orthogonal space and the stabilizer of the vector. With the basis chosen above, the lattice on $V_{y,v}^\perp$ induced from the ambient lattice $\Lambda_v$ and the lattice $\Lambda^0_{y,v}$ at one place of $E$ above $v$ are
\begin{equation}\label{split two lower lattices}
\left\langle t_1,\ldots,t_n\right\rangle,
\qquad
\left\langle p_v^{s+t-r}t_1,t_2,\ldots,t_n\right\rangle,
\end{equation}
respectively. At the other place of $E$ above $v$, the two lattices coincide. Hence the local difference vanishes when $s+t=r$.

For $r-s-t=1$, \cite[Theorem~5.4(2)]{LMZ} identifies the selected split factor with the relative EL moduli problem of a one-dimensional strict $\mathcal O_{F_v}$-divisible group $\mathcal G$ of $\mathcal O_{F_v}$-height $n$. Here strict means that the action on the Lie algebra is the structural action. This identification preserves the Hodge bundle. The change in \eqref{split two lower lattices} is given by the quotient $\mathcal G\to\mathcal G/C$, where
\[
 C=\langle z\rangle\subset\mathcal G[p_v]
\]
is the finite flat Drinfeld subgroup generated by a primitive torsion point $z$. Its divisor is the sum of the multiples of $z$. By the description of Drinfeld level structures in \cite[Sections~1 and~3]{Str}, let $[p_v]{\mathcal G}(z)$ denote the power series for multiplication by $p_v$. The equation
$$
[p_v]{\mathcal G}(z)=0
$$
defines the full $p_v$-torsion, including the zero section $z=0$. Removing this factor, the nonzero $p_v$-torsion points, equivalently the points of exact order $p_v$, are defined by
\begin{equation}\label{split primitive equation}
\frac{[p_v]_{\mathcal G}(z)}{z}=0.
\end{equation}
We denote the Cartier divisor $z=0$ by $\mathcal E$. The constant term of \eqref{split primitive equation} is $p_v$, so $\mathcal E$ has degree one over the special fiber on the common finite level.

Up to a unit, the derivative of the quotient isogeny is
\[
 \prod_{\alpha\in(\mathcal O_{F_v}/p_v)^\times}[\alpha]_{\mathcal G}(z).
\]
Each factor is $z$ times a unit. Thus the derivative vanishes to order $N_v-1$ along $\mathcal E$. Away from $\mathcal E$, the subgroup $C$ is \'etale and the quotient induces an isomorphism on the Hodge bundle. Therefore
\begin{equation}\label{split Hodge divisor depth one}
 \LL|_{\mathcal Z(y)}-\LL_{\mathcal Z(y)}=(N_v-1)\mathcal E.
\end{equation}

\noindent\textbf{Step 2.} Now we use Frobenius to compute the intersection numbers of these two line bundles. On $\mathcal E$, the Cartier divisor of $C|_{\mathcal E}$ is $N_v$ times the zero section. Hence
\[
 C|_{\mathcal E}=\ker(F_{\mathcal G/\mathcal E}),
 \qquad
 (\mathcal G/C)|_{\mathcal E}\simeq\mathcal G^{(N_v)}|_{\mathcal E}.
\]
Throughout this paper, by a CM factor we mean the factor of CM type in the corresponding local decomposition; its moduli are zero-dimensional. In the present split calculation this factor is fixed, so its first Chern class on $\mathcal E$ is zero. It follows that
\begin{equation}\label{split Frobenius Hodge class}
 c_1\bigl(\LL_{\mathcal Z(y)}|_{\mathcal E}\bigr)
 =N_v\,c_1\bigl(\LL|_{\mathcal Z(y)}|_{\mathcal E}\bigr)
 \quad\text{in }\operatorname{CH}^1(\mathcal E)_\QQ.
\end{equation}
The cover obtained by choosing a primitive $p_v$-torsion point has generic degree $N_v^n-1$. Thus the projection formula, summed over its generic components, gives
\begin{equation}\label{split primitive degree}
 \deg_L Z(y)
 =(N_v^n-1)\deg\bigl(\mathcal E\cdot(\LL|_{\mathcal Z(y)})^{n-1}\bigr).
\end{equation}
We now expand the difference of the two $n$-th powers. By \eqref{split Hodge divisor depth one} and \eqref{split Frobenius Hodge class},
\begin{align*}
 &\left((\LL|_{\mathcal Z(y)})^n-\LL_{\mathcal Z(y)}^n\right)_v\\
 ={}&(N_v-1)\sum_{j=0}^{n-1}
 \deg\bigl(\mathcal E\cdot(\LL|_{\mathcal Z(y)})^{n-1-j}
 \cdot\LL_{\mathcal Z(y)}^j\bigr)\log N_v\\
 ={}&(N_v-1)(1+N_v+\cdots+N_v^{n-1})
 \deg\bigl(\mathcal E\cdot(\LL|_{\mathcal Z(y)})^{n-1}\bigr)\log N_v\\
 ={}&\deg_L Z(y)\log N_v.
\end{align*}
This proves \eqref{split D0 formula} for $r-s-t=1$.

\noindent\textbf{Step 3.} Finally, we add the contributions of the successive quotients for general $r-s-t$. For $r-s-t>0$, we choose a primitive $p_v^{r-s-t}$-torsion point $z$. The required quotient factors through
\[
 \mathcal G/\langle p_v^{r-s-t-j+1}z\rangle
 \longrightarrow
 \mathcal G/\langle p_v^{r-s-t-j}z\rangle,
 \qquad 1\le j\le r-s-t.
\]
Each map has a kernel of order $N_v$, as in Step 1. The stabilizer of $z$ on the generic fiber is
\[
 \{g\in\GL_n(\mathcal O_{F_v}):ge_1\equiv e_1\pmod{p_v^{r-s-t}}\},
\]
where $e_1$ is the first basis vector. Its index in $\GL_n(\mathcal O_{F_v})$ is
\begin{equation}\label{split primitive level degree}
 (N_v^n-1)N_v^{n(r-s-t-1)}.
\end{equation}
For the $j$-th quotient, conjugation by
\[
 \operatorname{diag}(p_v^{-(j-1)},1,\ldots,1)
\]
places this stabilizer inside the stabilizer used in Step 1 for the corresponding intermediate lattice. The resulting finite morphism on the generic fiber has degree $N_v^{n(r-s-t-1)}$. Therefore the adelic projection formula in \cite[Proposition~4.1.2]{YZ2} gives the same normalized contribution for each quotient. We add the $r-s-t$ differences and obtain
\begin{equation}\label{split arbitrary depth self intersection}
 \left((\LL|_{\mathcal Z(y)})^n-\LL_{\mathcal Z(y)}^n\right)_v
 =(r-s-t)\deg_L Z(y)\log N_v
\end{equation}
at the chosen place above $v$.

At the conjugate place, the two lattices agree. Over the local rings of codimension one on the normal model, the generic isomorphism and its inverse extend to the $p$-divisible groups by \cite[Section~4.2, Theorem~4 and Corollary~2]{Tate}. Hence the integral Hodge bundles also agree, and this place contributes zero. The sum of the two contributions gives \eqref{split D0 formula}.

If a finite ramified extension of $F_v$ is needed to define the two models simultaneously, the Cartier multiplicity and the local logarithmic factor together multiply the local intersection by the degree of the extension. The projection formula divides by the same degree, so the normalized intersection is unchanged.
\end{proof}

We next substitute this value of $D_{0,v}$ into \eqref{f series} and compute the stabilizer volumes.

\begin{proposition}\label{Explicit f term}
Suppose that $v$ is split in $E/F$ and $a\ne0$. The function $f_{\Phi_v,a}(1)$ is supported on $\mathcal O_{F_v}$. For $r=v(a)\ge0$, we have
\begin{equation}\label{split explicit f formula}
\begin{aligned}
 f_{\Phi_v,a}(1)
 ={}&|d_v|^{n+\frac12}
 \frac{1-N_v^{-(n+1)}}{(1-N_v^{-n})^2}\\
 &\cdot\bigl(r(1-N_v^{-(r+2)n})
 -(r+2)(N_v^{-n}-N_v^{-(r+1)n})\bigr)\log N_v.
\end{aligned}
\end{equation}
Equivalently, the Whittaker formula in \cite[Lemma~4.5]{Guo1} gives
\begin{equation}\label{split f equals S}
\begin{aligned}
 f_{\Phi_v,a}(1)
 ={}&2W'_{a,v}(0,1,\Phi_v)-\log|a|_vW_{a,v}(0,1,\Phi_v)\\
 &+2\left(\frac{L'_v(n+1,\eta_v^{n+1})}{L_v(n+1,\eta_v^{n+1})}
       -\log|d_v|\right)W_{a,v}(0,1,\Phi_v)\\
 &-2|d_v|^{n+\frac12}\frac{1-|d_v|}{N_v-1}\log N_v.
\end{aligned}
\end{equation}
\end{proposition}

\begin{proof}
We compute the stabilizer volume for type $(s,t)$ and sum its contribution. We group the pairs $(s,t)$ by the value of $s+t$, so the final calculation is a finite geometric sum.

First we compute the weight of the orbit of $y_v$. Proposition~\ref{split local D0 formula} gives the factor $(r-s-t)\log N_v/\deg_L(X)$. For $r-s-t>0$, the coordinate calculation for the stabilizer of the lattice in $y_v^\perp$ gives
\[
 |d_v|^{n+\frac12}(1-N_v^{-(n+1)})
 N_v^{(n-1)(r-s-t-1)}\frac{N_v^n-1}{N_v-1}.
\]
The condition that the action on $E_vy_v$ and on $y_v^\perp$ preserve the ambient lattice gives the additional index
\[
 (N_v-1)N_v^{r-s-t-1}.
\]
Therefore
\begin{equation}\label{quotient of volume at split prime}
 \frac{\vol(U_v)}{\vol(U_{y,v})}
 =|d_v|^{n+\frac12}(1-N_v^{-(n+1)})
 (N_v^n-1)N_v^{n(r-s-t-1)}.
\end{equation}

It remains to sum these weights with the computed values of $D_{0,v}$. For each integer $k$ with $1\le k\le r$, there are $r-k+1$ pairs $(s,t)$ satisfying $s+t=r-k$. The pairs with $s+t=r$ contribute zero. Thus \eqref{f series} gives
\begin{align*}
 f_{\Phi_v,a}(1)
 ={}&|d_v|^{n+\frac12}(1-N_v^{-(n+1)})
 N_v^{-nr}(N_v^n-1)\\
 &\cdot\sum_{k=1}^r(r-k+1)kN_v^{n(k-1)}\log N_v.
\end{align*}
We evaluate the sum to obtain \eqref{split explicit f formula}. Then the proof is finished by  \cite[Lemma~4.5]{Guo1}.
\end{proof}

\subsection{Explicit local computations at inert places}
Suppose that $v$ is inert in $E/F$. We compute $D_{0,v}$ and then use it to compare $f_{\Phi_v,a}$ with the Whittaker derivative. We first compare the two Hodge metrics in \eqref{local D0 adelic definition}. For even type and $n\ge3$, Lemma~\ref{inert lower morphism lemma} extends $f_y$ to the normalization of the primitive divisor, so that we can use the pullback of the Hodge bundle from the lower-dimensional integral model. For odd type, we construct maps from type $s$ to type $s-2$ and compute the change of the Hodge bundle in Lemma~\ref{inert odd adjacent correspondence}. These maps reduce the comparison to the fixed metric at type $1$. Together with the intersection relations for the special cycles, these comparisons give the formula for $D_{0,v}$ in Proposition~\ref{inert local D0 formula}; the even case with $n=2$ requires a separate calculation using basic curves. Finally, we sum over the orbits of vectors of a fixed norm to compute $f_{\Phi_v,a}$ and compare it with the Whittaker derivative in Proposition~\ref{Explicit f term inert}.

It should be noted that this subsection is rather lengthy because some essential differences arise between the cases \(n=2\) and \(n\geq 3\), which require us to treat them separately. This is particularly evident in Lemma \ref{inert local D0 formula}.

Let $A_0$ and $A$ be the abelian varieties in the moduli problem of \cite{Guo2}, and let
$$
j:A_0\longrightarrow A
$$
be the special homomorphism corresponding to the primitive vector of type $s$. For the local calculations below, we write
$$
Y=A_0[p^\infty]_v
$$
and use $A$ also for $A[p^\infty]_v$. The polarizations define the adjoint $j^\dagger$. In this subsection, $Z(j)$ denotes the Kudla--Rapoport Cartier divisor on the local integral model. Since
$$
Z(j/p_v)\subset Z(j),
$$
we define the Cartier divisor $\mathcal D_s$ by
$$
Z(j)=Z(j/p_v)+\mathcal D_s.
$$
We call $\mathcal D_s$ the primitive divisor of type $s$.

\begin{lemma}\label{inert lower morphism lemma}
Let $y_v$ be of inert type $s=2\ell$, with $n\ge3$. After passing to a common finite level, the generic morphism $f_y$ in \eqref{lower finite morphism} extends to the normalization of $\mathcal D_s$. The pulled-back Hodge bundle represents the $v$-adic metric of $\LL_{\mathcal Z(y)}$, and the same remains true after further finite pull-back.
\end{lemma}

\begin{proof}
We take the closure of the graph of $f_y$ and normalize it. We first work with the relative moduli problem and show that the projection has finite fibers, using lattice bounds at nonbasic points and the argument with rational curves at basic points. We then pass to the original model, where the projection to the normalization of $\mathcal D_s$ is finite and birational, hence an isomorphism. This gives the required extension of $f_y$.

\textbf{Step 1.} We realize the lower-dimensional object on the generic fiber as a quotient of $A$ by a finite flat subgroup. This allows us to parameterize its possible extensions over the primitive divisor.

Since norms of units are surjective at an inert place, we may multiply $j$ by a unit and assume $j^\dagger j=p_v^{2\ell}$. On the generic fiber, the lattice $\Lambda^0_{y,v}$ in \eqref{lower lattice unramified definition} determines a lower-dimensional object $B$ and an isogeny
\begin{equation*}
u:j(Y)^\perp\longrightarrow B.
\end{equation*}
Let $\operatorname{pr}_{j(Y)^\perp}$ be the orthogonal projection. Then
\begin{equation*}
h=(j^\dagger,p_v^\ell u\operatorname{pr}_{j(Y)^\perp}):
A\longrightarrow Y\times B
\end{equation*}
satisfies
\begin{equation*}
h^\dagger h=[p_v^{2\ell}],\qquad
hj=([p_v^{2\ell}],0).
\end{equation*}

The generic isogeny $h$ has kernel contained in $A[p_v^{2\ell}]$. We characterize this kernel by conditions that continue to make sense integrally. We parameterize the $\mathcal O_{E_v}$-stable maximal isotropic finite flat subgroups
\begin{equation*}
K\subset A[p_v^{2\ell}]
\end{equation*}
such that
\begin{equation*}
K\subset\ker(j^\dagger),\qquad
j(Y[p_v^{2\ell}])\subset K.
\end{equation*}
Put $A'=A/K$. The first condition allows $j^\dagger$ to descend to a morphism
\begin{equation*}
p_Y:A'\longrightarrow Y.
\end{equation*}
The second implies that the composite $Y\to A\to A'$ vanishes on $Y[p_v^{2\ell}]$, and hence there is a morphism
\begin{equation*}
i_Y:Y\longrightarrow A'
\end{equation*}
such that
\begin{equation*}
Y\xrightarrow{[p_v^{2\ell}]}Y\xrightarrow{i_Y}A'
\end{equation*}
is the original composite. The resulting maps $Y\to A/K\to Y$ are adjoint and have composite $1$. Thus
\[
A/K=Y\times B,
\]
where $B$ has the required polarization.

We take the normalized closure of the generic graph in this projective parameter space. The maps $h$ and $j$ determine the quasi-isogeny from $B$ to the fixed divisible group defining the lower local moduli problem, with bounded denominator. On the closure of the generic graph, $B$ has the required strict action and signature: these are closed conditions, and they hold on the schematically dense generic fiber. Conversely, this quasi-isogeny and $j$ determine $h$ and $K$. Thus the map from the graph to the product of the ambient and lower moduli spaces is proper with finite fibers, and hence finite.

\textbf{Step 2.} From now on, the basic locus means the locus in the special fiber where the associated $\mathcal O_{E_v}$-linear $p$-divisible group is basic in the Newton stratification; the nonbasic locus is its complement. Since $Y$ is basic and $A$ is isogenous to $Y\times B$, $B$ is nonbasic whenever $A$ is nonbasic. We show that the graph has finite fibers over the nonbasic locus. Since $h$ is an isogeny and
\begin{equation*}
h^\dagger h=[p_v^{2\ell}],
\end{equation*}
we also have $hh^\dagger=[p_v^{2\ell}]$. Thus $h$ and $h^\dagger$ bound the lattice of the marked lower object $B$ from above and below by fixed powers of $p_v$. The unitary datum here is one of the fully Hodge--Newton decomposable cases in \cite[Theorem~3.5]{GHN}. Hence every nonbasic affine Deligne--Lusztig variety is zero-dimensional by \cite[Theorem~3.3]{GHN}. Its intersection with the bounded region above is therefore finite. Thus only finitely many marked lower objects $B$ occur over a fixed nonbasic point, and the graph has finite fibers there.

\textbf{Step 3.} We now work locally at a basic point of the relative primitive divisor, which is regular by \cite[Theorem~1.2.1]{Zhu}. We show that the projection from the relative graph is an isomorphism. The main point is that the lower object $B$ cannot vary along a rational curve.

By Step~2, every exceptional prime divisor lies over the basic locus. Suppose that such a divisor exists. By Abhyankar's ruledness theorem, the geometric generic fiber over its center is birationally ruled by rational curves; see \cite[Theorem~5.4]{Kol}. Along each such curve, $A$ and $j$ are fixed. Since $A$ is basic and isogenous to $Y\times B$, the lower object $B$ is also basic. The lattice bounds place these marked lower objects in a finite union of compact hyperspecial unitary Deligne--Lusztig components. On each positive-dimensional component, let $\mathcal Q$ be the tautological quotient bundle. Then the calculation in \cite[Proposition~2.6 and Lemma~2.8]{Li} gives
\[
\Omega^1\simeq
\bigwedge^{\operatorname{rk}\mathcal Q-1}\mathcal Q
\otimes(\det\mathcal Q)^{N_v-2}.
\]
Since $\mathcal Q$ is a quotient of a trivial bundle, $\Omega^1$ is globally generated. A nonconstant separable map from $\mathbf P^1$ would therefore give a nonzero map from a globally generated bundle to $\Omega^1_{\mathbf P^1}=\mathcal O(-2)$, which is impossible. Factoring through Frobenius gives the same conclusion for inseparable maps.

Therefore $B$ is constant along each rational curve. This contradicts the finiteness of the map to the product of moduli spaces from Step~1. Hence the graph has no exceptional prime divisor.

It remains to exclude positive-dimensional fibers. Let $\pi$ be the projection from the relative graph to the relative primitive divisor, and let $H$ be a relatively ample Cartier divisor on the graph. Since the relative primitive divisor is regular, $\pi_*H$ is Cartier. Moreover,
\begin{equation*}
H-\pi^*\pi_*H
\end{equation*}
is supported on the exceptional prime divisors of $\pi$, and hence is zero by the preceding argument. Thus $H$ is trivial on every fiber of $\pi$. Since $H$ is relatively ample, every fiber must be zero-dimensional. Therefore $\pi$ is finite. Since it is also birational and the primitive divisor is normal, $\pi$ is an isomorphism.

\textbf{Step 4.} We return to the original integral model. We consider the normalized closure of the generic graph over the normalization of $\mathcal D_s$. We show that its projection has finite fibers everywhere.

At a nonbasic point, the comparison in \cite[Theorem~4.9, (4.16)--(4.17)]{LMZ}, together with the compatibility with duality in \cite[Theorem~4.20]{LMZ}, preserves the morphisms, the polarization, and the bounds on the lower lattice. Thus the argument in Step~2 gives finite fibers at these points.

At a basic point, the associated groups are relatively biformal. The comparison in \cite[Theorem~5.12 and Corollary~5.14]{LMZ} identifies the completed absolute moduli problem with the corresponding base change of the relative one. Moreover, \cite[Proposition~8.13]{LMZ} identifies the special divisors defined by $j$ and $j/p_v$, and hence their difference. Therefore the morphism constructed in Step~3 gives an extension on the completed neighbourhood. We pull this extension back to the normalization of $\mathcal D_s$.

Over this completed neighbourhood, the closure of the generic graph is the graph of the extension just constructed: the lower moduli space is separated, and taking the closure commutes with flat base change. Since normalization is finite, the normalized graph also has finite fibers there. The same conclusion holds at the original basic point, since the completed base change is faithfully flat.

The projection is therefore projective with finite fibers, and hence finite. It is also birational, and its target is normal, so it is an isomorphism. Thus $f_y$ extends to the normalization of $\mathcal D_s$ on the original integral model. The pullback of the lower-dimensional Hodge bundle gives the $v$-adic metric of $\LL_{\mathcal Z(y)}$, and this description is compatible with further finite pullback.
\end{proof}

For the local Hodge calculations, we first describe the subgroup filtration on the $\mu$-ordinary locus, namely the open Newton stratum of the local unitary moduli problem. We then compute the intersections of $\mathcal D_s$ with the cycles $Z(j/p_v^b)$.

The construction in \cite[(5.6), Section~5.3 before Definition~5.11]{LMZ} gives an $\mathcal O_{E_v}$-linear equivalence $\mathcal F$ with strict $\mathcal O_{F_v}$-divisible groups, compatible with base change, and an identification
\begin{equation*}
   J_{\mathfrak A_v}\Lie A\simeq\Lie(\mathcal F A).
\end{equation*}
In particular, it identifies the Lie summands defining the Hodge bundles.

Over a discrete valuation ring, the inverse functor carries the relative ordinary filtration to subgroups of $A$. Indeed, after a faithfully flat base change, the comparison at geometric points preserves saturated submodules and their quotients. At each $p^a$-torsion level, the subgroup inclusion on the special fiber gives a surjection of finite free coordinate algebras. By Nakayama's lemma, this surjection lifts over the discrete valuation ring, and hence gives a closed immersion of the corresponding finite flat group schemes. The induced map on the quotient is an isomorphism on the special fiber and therefore an isomorphism over the ring. Passing through all $a$ and descending gives
\begin{equation}\label{MA}
    M_A\subset A^{\mathrm{conn}}\subset A.
\end{equation}
Here $A^{\mathrm{conn}}$ and $M_A$ correspond to the connected and relative multiplicative subgroups of $\mathcal F A$. The middle quotient corresponds to the positive CM factor, and the last quotient to the \'etale factor. Only $M_A$ contributes to the Lie summand defining $\LL$. The same construction applies to $B$.

Next, we define $T_{b,s}=Z(j/p_v^b)|_{\mathcal D_s}$ whenever $Z(j/p_v^b)$ is nonempty. We first check that these special cycles are flat. Note that every irreducible component of the ambient special fiber contains a basic point. Indeed, after a finite extension of the CM base, CM points and their finite adelic translates meet every geometric connected component of the generic fiber, and their reductions at $v$ are basic. Since the ambient model is proper and smooth, the connected components of the geometric fibers are locally constant over the base. Thus every geometric connected component of the special fiber comes from a unique component of the geometric generic fiber.

For each $b$, the natural map from $Z(j/p_v^b)$ to the ambient model is finite. Indeed, positivity of the Rosati norm gives finite geometric fibers, while homomorphisms extend over discrete valuation rings, so the morphism is proper. Suppose that the special cycle had a vertical irreducible component. Since the morphism is finite, this component would map onto an irreducible component of the ambient special fiber, and hence would pass through a basic point by the preceding paragraph. At such a point, \cite[Theorem~5.12 and Corollary~5.14]{LMZ} identifies the completed lifting problem with the corresponding relative one. The special divisor is flat in the relative model by \cite[Proposition~2.3.3]{Zhu}, giving a contradiction. Hence there are no vertical components. The Cartier branches are Cohen--Macaulay, so there are no embedded associated components either. This proves flatness on the smooth model and on finite covers that are \'etale at $v$.

The Grothendieck--Messing obstruction identifies the conormal bundle of each special divisor with the restriction of the ambient Hodge bundle; see \cite[Propositions~4.1 and~4.3]{How2}. Let $f_0,f_1$ define $Z(j)$ and $Z(j/p_v)$ locally. Then $f_0/f_1$ defines $\mathcal D_s$. On the conormal line bundles, the relation $j=p_v(j/p_v)$ induces multiplication by $p_v$. Consequently
\begin{equation*}
     (f_0/f_1)|_{Z(j/p_v)}=p_v\cdot\text{a unit}.
\end{equation*}
Restricting the above identity to $Z(j/p_v^b)$ shows that $T_{b,s}$ is the Cartier divisor on $Z(j/p_v^b)$ defined by $p_v=0$. For $b\le c$, the same conormal identification gives
\begin{equation*}
   \mathcal O(T_{b,s})|_{T_{c,s}}\simeq(\LL|_{\mathcal Z(y)})^{-1}|_{T_{c,s}}.
\end{equation*}
Taking first Chern classes and using symmetry, we obtain
\begin{equation}\label{inert nested conormal}
T_{b,s}T_{c,s}
=-(\LL|_{\mathcal Z(y)})T_{\max(b,c),s}.
\end{equation}
These Cartier-divisor identities remain valid after pullback to the common refinements used below, and the projection formula shows that the corresponding normalized intersection numbers are unchanged.

For odd type, the metric on $\LL_{\mathcal Z(y)}$ is defined by comparison with the fixed type-$1$ vector. We therefore need to compare two successive odd types. The next lemma constructs the map from type $s$ to type $s-2$ and describes both the change of the Hodge bundle and the pullback of the divisors $T_{b,s-2}$. Iterating these formulas will reduce the computation of $D_{0,v}$ for every odd type to the type-$1$ case.

\begin{lemma}\label{inert odd adjacent correspondence}
Suppose that $v$ is inert, $n\ge2$, and the residue characteristic is odd. For odd $s\ge3$, there is a morphism $\rho_s$ from the normalized local model of type $s$ to that of type $s-2$, of generic degree $N_v^{2n}$. On the normalized local models, with the prescribed generic identification of Hodge bundles, we have
\begin{equation}\label{inert odd adjacent Hodge}
  \LL|_{\mathcal D_s}-\rho_s^*(\LL|_{\mathcal D_{s-2}})=(N_v^2-1)T_{1,s}.
\end{equation}
Whenever $T_{b,s-2}$ is defined, we also have
\begin{equation}\label{inert odd old pullback}
  \rho_s^*T_{b,s-2}=N_v^2T_{b+1,s}.
\end{equation}
Both identities and the metric obtained from the fixed type-$1$ vector are compatible with further pullback to a common finite level.
\end{lemma}

\begin{proof}
We first construct $\rho_s$ and compute the two identities at $\mu$-ordinary vertical primes. We then show that this determines all vertical coefficients for $n\ge3$, and treat $n=2$ by adding a fixed positive CM factor.

\textbf{Step 1.} The first step is to construct the isogeny and compute the degree of $\rho_s$. Write $s=2\ell+1$. When $j$ is a closed immersion, take $C=\ker j^\dagger$ and
\begin{equation*}
  K=C[p_v]+j(Y[p_v^2])\subset A[p_v^2].
\end{equation*}
The summands intersect in $j(Y[p_v])$, so $K$ is finite flat of order $N_v^{2(n+1)}$. The first summand is isotropic for the $p_v^2$-pairing. The two summands are orthogonal because $j^\dagger$ vanishes on $C$, and the second summand is isotropic since $p_v^2\mid j^\dagger j$. Thus $K$ is maximal isotropic. Then the quotient
\begin{equation*}
   \alpha:A\longrightarrow A'=A/K
\end{equation*}
is principally polarized and satisfies $\alpha^\dagger\alpha=[p_v^2].$ Moreover, $j(Y[p_v^2])\subset K$, so $\alpha j$ vanishes on $Y[p_v^2]$. Hence there is a special homomorphism $j':Y\to A'$ such that
\begin{equation*}
  \alpha j=[p_v^2]j'.
\end{equation*}
It follows that
\begin{equation*}
   (j')^\dagger j'=p_v^{-2}j^\dagger j.
\end{equation*}
Hence $j'$ has type $s-2$, and $\alpha/p_v$ gives the generic identification of Hodge bundles.

Identify the rational covariant Dieudonn\'e modules so that $\alpha$ is $p_v$ times the identity. If $M$ is the lattice of $A$ over its coefficient ring $\mathcal O$, then the lattice of $A'$ is
\begin{equation*}
\{m\in M:\langle m,\mathbb D(jY)\rangle\subset p_v\mathcal O\}
+p_v^{-1}\mathbb D(jY).
\end{equation*}
Indeed, the first summand is $p_vM+\mathbb D(C)$, and the formula follows from $K$. It shows that $j'$ is saturated in both components. The map $C\to\ker(j')^\dagger$ is $[p_v]$ followed by the quotient by $j(Y[p_v])$. This positive CM group contributes zero to the Lie summand, so the latter quotient induces an isomorphism there. Thus $\alpha/p_v$ is a unit on this locus.

We also need to check when a primitive special homomorphism is saturated. Suppose that an integral primitive $j$ is not saturated in one of the two components. We may then choose relative crystal generators such that
\begin{equation*}
   jg=x,\qquad jh=p_vz,\qquad Fx=p_vz,\qquad Vx=-p_vz,\qquad Fz=-x,\qquad Vz=x,
\end{equation*}
where $F$ and $V$ denote Frobenius and Verschiebung on the relative covariant Dieudonn'e module. By the signature condition, $x$ spans the relevant image of $V$ modulo $p_v$. The polarization then implies that $\langle x,z\rangle$ is a unit, and hence
\begin{equation*}
    v(j^\dagger j)=1.
\end{equation*}
Therefore, if $j$ is primitive and has type $s\ge3$, its image is saturated in both components.

We now extend the generic construction to the integral models. Take the normalized closure of its graph in the projective parameter space of maximal isotropic finite flat subgroups of $A[p_v^2]$. The same argument as in the proof of Lemma~\ref{inert lower morphism lemma} shows that this graph is the primitive divisor: it is finite over the nonbasic locus, while the argument with the basic Deligne--Lusztig components excludes exceptional divisors. The same absolute--relative comparison then gives a morphism
\begin{equation*}
\rho_s:\mathcal D_s\longrightarrow\mathcal D_{s-2}
\end{equation*}
on the original models.

It remains to compute its generic degree. The stabilizers of the type-$s$ and type-$(s-2)$ vectors have index ratio
\begin{equation*}
   \frac{(N_v^n-(-1)^n)N_v^{n(s-1)}}{(N_v^n-(-1)^n)N_v^{n(s-3)}}=N_v^{2n}.
\end{equation*}
Hence
\begin{equation*}
   \deg\rho_s=N_v^{2n}.
\end{equation*}

\textbf{Step 2.} We compute the two divisors at a vertical prime whose generic point lies in the $\mu$-ordinary locus. At the generic point of this prime, write the projection of $j$ to the positive CM factor as
\begin{equation*}
p_v^iu,
\end{equation*}
where $u$ is primitive and $0\le i\le\ell$. We work over the completed discrete valuation ring at this generic point. Both $Y$ and the positive CM quotient have the same full CM type. By \cite[Theorem~5.4(1)]{LMZ}, these groups correspond to \'etale $p$-divisible $\mathcal O_{E_v}$-modules. Hence the groups and their homomorphisms lift uniquely from the residue field. Thus $u$ and the factorization $p_v^iu$ are defined over this ring. Suppose first that $i>0$. Choose a height-one Lubin--Tate coordinate $z$ transverse to the special cycle, and write $p_v^a$ for the power series giving multiplication by $p_v^a$. The lifting conditions for $p_v^iu$ and $p_v^{i-1}u$ are then
\begin{equation*}
p_v^i=0
\qquad\text{and}\qquad
p_v^{i-1}=0,
\end{equation*}
respectively.
Since $u$ is primitive, the linear obstruction coefficient is a unit. Therefore, we have
\begin{equation*}
   \operatorname{ord}(p_v)=(N_v^2-1)N_v^{2i-2},\qquad \operatorname{ord}(T_{b,s})=N_v^{2(i-b)}\quad (b\le i),
\end{equation*}
while $\operatorname{ord}(T_{b,s})=0$ for $b>i$. 

We now use the filtration in \eqref{MA} to compute the kernel on the generic fiber. Since $j$ has no \'etale component and $i>0$, the projection $p_v^iu$ vanishes on $Y[p_v]$, and hence $j(Y[p_v])\subset M_A[p_v]$. The saturation of $j$ and equality of orders give $j(Y[p_v])=M_A[p_v]$. The polarization identifies $A^{\mathrm{conn}}[p_v]$ with the orthogonal complement of $M_A[p_v]$, so $C[p_v]=A^{\mathrm{conn}}[p_v]$ as well.

Consequently, the generic kernel $K=C[p_v]+j(Y[p_v^2])$ is the inverse image of $M_A[p_v]$ under
\begin{equation*}
    [p_v]:A^{\mathrm{conn}}[p_v^2]\longrightarrow A^{\mathrm{conn}}[p_v].
\end{equation*}
Over the discrete valuation ring, this map is finite flat and surjective, so the same inverse image is finite flat. Since $[p_v]:M_A[p_v^2]\to M_A[p_v]$ is fppf surjective, the inverse image equals $A^{\mathrm{conn}}[p_v]+M_A[p_v^2]$ as an fppf subgroup. It is therefore the finite flat closure of the generic kernel, and gives $K$ over this ring.

The quotient acts by $[p_v^2]$, $[p_v]$, and $1$ on the multiplicative, positive CM, and \'etale factors, respectively. Since the Lie summand defining the Hodge bundle comes from the multiplicative part, $\alpha/p_v$ has order $\operatorname{ord}(p_v)$ on this line. Thus $\alpha/p_v$ has order $\operatorname{ord}(p_v)$ there, and
\begin{equation*}
    \operatorname{ord}(\LL|_{\mathcal D_s}-\rho_s^*\LL|_{\mathcal D_{s-2}})=(N_v^2-1)\operatorname{ord}(T_{1,s}).
\end{equation*}
The projection of $j'$ to the positive CM factor is divisible by $p_v^{i-1}$ but not by $p_v^i$. In the Lubin--Tate coordinate, $\rho_s$ is given by $z\longmapsto [p_v](z).$ Hence the pullback of $T_{b,s-2}$ is $T_{b+1,s}$ with multiplicity $N_v^2$. Therefore
\begin{equation*}
\rho_s^*T_{b,s-2}=N_v^2T_{b+1,s},
\end{equation*}
which proves \eqref{inert odd old pullback}. For $i=0$, Step~1 gives a unit on the Lie summand and all $T_{b,s}$ vanish locally.

\textbf{Step 3.} For $n\ge3$, we show that the preceding calculations determine every vertical coefficient. Recall that the relative Dieudonn\'e module has two components corresponding to the two embeddings of $E_v$, and that Frobenius $F$ interchanges them. Thus $F^2$ preserves each component, which has rank $n+1$. At a geometric point, the valuations of the elementary divisors of $F^2$ are either
\begin{equation*}
(1,\ldots,1)
\end{equation*}
or
\begin{equation*}
(0,1,\ldots,1,2).
\end{equation*}
The first case is basic. Hence at a nonbasic point only the second case occurs. Indeed, the two successive
$F$-maps have complementary minuscule types, total valuation $n+1$, and at most one exponent $0$ and one exponent $2$. For the slopes $\lambda<1$, counted with multiplicity, we have $\sum_{\lambda<1}(1-\lambda)\le1$. If a slope $\lambda=a/b<1$ is written in lowest terms, its multiplicity is divisible by $b$. Hence its contribution to $\sum_{\lambda<1}(1-\lambda)$ counted with multiplicity is a positive integer. By symmetry of the Newton slopes, they are
\begin{equation*}
(1-1/h)^{[h]},\qquad 1^{[n+1-2h]},\qquad (1+1/h)^{[h]}
\end{equation*}
for some $h$. The Hodge and Newton polygons meet at the two break points: the sums of the slopes in the three blocks are
\begin{equation*}
   h-1,\qquad n+1-2h,\qquad h+1,
\end{equation*}
which agree with the corresponding sums of the Hodge slopes. Applying \cite[Theorem~1.1(2)]{Kot} in both components gives an integral
orthogonal decomposition
\begin{equation*}
A[p^\infty]_v=B\times P,
\end{equation*}
where $P$ is principally polarized of positive CM type and $B$ has signature $(2h-1,1)$. The absolute--relative comparison preserves the integral projectors. Since $Y$ has relative slope $1/2$, every special homomorphism from $Y$ lands in $P$. The lattice $\operatorname{Hom}_{\mathcal O_{E_v}}(Y,P)$ is unimodular, and its primitive elements are saturated in both components. For primitive $u:Y\to P$, the lifting obstruction has nonzero differential. Hence $Z(u)$ is smooth by \cite[Theorem~3.17]{LMZ}; see also \cite[Lemma~4.2.1]{Zhu}.

Keep $P,u$ fixed and apply the polarized coefficient-linear deformations of case AU in \cite[Section~4.3]{BiH} to $B$. Their endomorphisms satisfy $N^2=0$ and $N^\dagger=-N$; extending them by zero on $P$ preserves $u$ and the pairing. The Hodge filtration and coefficient action remain fixed, so the Eisenstein conditions are preserved. After removing the fixed Eisenstein factor, the remaining Hodge type is the minimal one required in \cite{BiH}. Therefore, Serre--Tate theory allows us to deform $B$, while keeping $P$ and $u$ fixed, to points of $Z(u)$ in the $\mu$-ordinary locus.

Let $\Gamma$ be a vertical prime of $\mathcal D_s$. The basic locus has dimension at most $\lfloor n/2\rfloor$, while $\dim\Gamma=n-1>\lfloor n/2\rfloor.$ Hence its generic point is nonbasic. If $j/p_v$ is not integral at the generic point of $\Gamma$, then $j$ is primitive there. Since $s\ge3$, the saturation criterion above shows that $j$ is saturated in both components. Thus $j$ is a closed immersion, and this is the case treated in Step~1. Otherwise take the largest $i$ for which $u=j/p_v^i$ is integral. Near a geometric point of a dense open subset of $\Gamma$, this $u$ is primitive and $Z(u)$ is smooth. By the local equation calculation above, the equation of $\mathcal D_s$ restricts to $p_v$ times a unit on $Z(u)$. Thus the special fiber of $Z(u)$ has a unique local component of dimension $n-1$, contained in $\mathcal D_s$ and equal to $\Gamma$ near that point. The deformation above stays in this component and reaches a point in the $\mu$-ordinary locus. Since the $\mu$-ordinary locus is open, the generic point of $\Gamma$ is therefore $\mu$-ordinary. Step~2 gives the required coefficients along $\Gamma$. Since $\Gamma$ was arbitrary, we prove this lemma when $n\ge 3$.

\textbf{Step 4.} It remains to consider the case when $n=2$. In this case, we treat the basic locus by passing to dimension three.

Add a fixed positive CM factor and denote the resulting objects by a superscript $+$. The map
\begin{equation*}
\delta:(A,j)\longmapsto(A\times Y,(j,0))
\end{equation*}
has the integral self-adjoint projector of \cite[(2.11.0.1)--(2.11.0.2)]{LZ}. After any base change, $(j/p_v^b,0)$ is integral if and only if $j/p_v^b$ is integral. Consequently
\begin{equation*}
\delta^*Z^+((j/p_v^b,0))=Z(j/p_v^b),\qquad \delta^*\mathcal D_s^+=\mathcal D_s,\qquad \delta_s^*T_{b,s}^+=T_{b,s},
\end{equation*}
where $\delta_s$ is the restriction of $\delta$ to $\mathcal D_s$. 

On the generic fiber, $C^+=C\times Y$ and $K^+=K\times Y[p_v]$. The second equality holds integrally because both subgroups are finite flat and hence are the schematic closures of their common generic fiber. Identifying $Y/Y[p_v]$ with $Y$ by $[p_v]$, we get
\begin{equation*}
A^{+\prime}=A'\times Y,\qquad
\alpha^+=\alpha\oplus[p_v],\qquad
\rho_s^+\delta_s=\delta_{s-2}\rho_s.
\end{equation*}
At a common level carrying the CM quotients, these identities also hold on the global components. The added factor contributes zero to the Lie summand, and the original fixed CM abelian variety is unchanged. Thus
\begin{equation*}
\delta_s^*(\LL^+|_{\mathcal D_s^+})=\LL|_{\mathcal D_s},\qquad
\alpha^+/p_v=(\alpha/p_v)\oplus1.
\end{equation*}
Then we can pull back the identities proved for $n=3$. Since $\mathcal D_s$ is flat and its generic fiber avoids the vertical divisors, their pullbacks under $\delta_s$ are well defined. Together with the ordinary calculation, this proves both identities for $n=2$. The degree is still computed in the original dimension, so $\deg\rho_s=N_v^4$. This finishes $n=2$ case.

Finally, note that the maps and Hodge bundle comparisons above are compatible with further finite level pullback, so their composition down to type $1$ gives the metric in Definition~\ref{adelic lower Hodge line definition}.
\end{proof}

We now use these comparisons to compute $D_{0,v}$ at an inert place.
For odd $s$, Lemma~\ref{inert odd adjacent correspondence} reduces the calculation to type $1$; for even $s$, we compare with the Hodge bundle on the lower-dimensional integral model. The idea is similar to Proposition \ref{split local D0 formula} for split places.

\begin{proposition}\label{inert local D0 formula}
Suppose that $v$ is inert in $E/F$ and $y_v$ has type $s$. With the metrics of Definition~\ref{adelic lower Hodge line definition},
\begin{equation}\label{inert D0 formula}
\deg_L(X)D_{0,v}(y_v)=
\begin{cases}
\displaystyle\left(s+\frac{2(-1)^nN_v^n}{N_v^{2n}-1}
(1-N_v^{-ns})\right)\log N_v,&s\text{ even},\\[2mm]
\displaystyle\left(s-1-\frac{2(1-N_v^{-n(s-1)})}{N_v^{2n}-1}
\right)\log N_v,&s\text{ odd}.
\end{cases}
\end{equation}
\end{proposition}
\begin{proof}
We first treat even $s$: in dimension two we compute on a regular surface, and for $n\ge3$ we determine the difference of the Hodge bundles along the special fiber. We then sum the intersections. Finally, Lemma~\ref{inert odd adjacent correspondence} gives the odd case by successive pullback to type $1$. We briefly explain why it is necessary to treat the cases \(n=2\) and \(n\geq 3\) separately. When $n\ge3$, a vertical prime divisor has dimension $n-1$, while the basic locus has dimension at most $\lfloor n/2\rfloor<n-1$. Hence its generic point is nonbasic, and the calculation can be reduced to the $\mu$-ordinary locus, where the ordinary filtration gives the Hodge divisor directly. For $n=2$, this dimension inequality fails, and a vertical prime may itself be a basic curve. 

\textbf{Step 1.} The first step is to reduce the even case in dimension two to a calculation on the regular relative surface. For $s=0$, the orthogonal splitting along a vector of unit norm identifies the two metrics. Assume $s=2\ell>0$, and recall that
\begin{equation}\label{inert old cycle definition}
T_{b,s}=Z(j/p_v^b)|_{\mathcal D_s},\qquad 1\le b\le\ell.
\end{equation}
For $n=2$, resolve the closure of the graph of $f_y$. Write $H$ and $H_0$ for the pullbacks of the ambient and lower Hodge bundles to this regular proper surface. Their fixed generic identification defines the divisor $H-H_0$.

At every geometric point, $A$ is isogenous to $Y\times B$: the generic isogeny and an integral multiple of its inverse extend over a discrete valuation ring through the point whose generic point has characteristic zero, as do their scalar compositions. Thus $A$ and $B$ are either both basic or both ordinary. On the ordinary locus, the two relative Verschiebung maps identify the Hodge line with its $N_v^2$-Frobenius pullback. The fixed CM factor has degree zero, so the Hodge bundle has degree zero on every complete curve contained in this locus.

Note that both Hodge bundles have zero degree on every curve contracted by the resolution. Indeed, for $H$ this follows from its definition; for $H_0$, the image in the lower curve is either basic, hence a point, or ordinary, where the preceding calculation applies. Similarly, both bundles have degree zero on any complete curve contained in the ordinary locus. Thus the required pairings are determined by the components that meet the basic locus. At the basic completions, \cite[Theorem~5.12 and Corollary~5.14]{LMZ} identifies the maps and Hodge bundles with those on the regular relative surface of \cite[Theorem~1.2.1]{Zhu}. Hence the projection formula reduces the required intersections to this surface.

\textbf{Step 2.} We compute the fiber multiplicities and the coefficients of $H-H_0$ on this surface. Let $Z_0$ be the special fiber of $Z(j/p_v^\ell)$. At an ordinary point, the positive CM factor has rank one and unit norm, so the norm of $j$ makes $j/p_v^\ell$ integral there. Hence the ordinary part lies on $Z_0$; the remaining curves are basic. Consider the incidence graph whose vertices are these curves together with $Z_0$, with two vertices joined when the corresponding curves meet. For $i\ge1$, we call the basic curves at distance $i$ from $Z_0$ the $i$-th layer, and let $R$ denote a curve in this layer.

The incidence relations and intersection numbers in \cite[Sections~2.6--2.7, Lemma~6.2.1]{LZ} give the following relation for the multiplicity $\mu_i$ of $Z(j)$ in layer $i$:
\begin{equation*}
   \mu_{i-1}-2N_v^2\mu_i+N_v^4\mu_{i+1}=-(N_v^2-1)\quad(2\le i\le\ell),\qquad \mu_\ell=1,\quad\mu_{\ell+1}=0.
\end{equation*}
Thus $\mu_i=\sum_{a=0}^{\ell-i}N_v^{2a}$, while the ordinary Lubin--Tate equation gives multiplicity $N_v^{2\ell}$ at $Z_0$. Subtracting the multiplicities of $Z(j/p_v)$ gives
\begin{equation*}
   \operatorname{div}(p_v)=(N_v^2-1)N_v^{2\ell-2}Z_0+\sum_{i=1}^{\ell}N_v^{2(\ell-i)}
   \sum_{R\text{ in layer }i}R.
\end{equation*}
Here $\operatorname{div}(p_v)$ denotes the divisor defined by $p_v=0$ on this surface. Each basic curve has
\begin{equation*}
H\cdot R=N_v^2-1,\qquad H_0\cdot R=0.
\end{equation*}
The incidence relations and $\operatorname{div}(p_v)\cdot R=0$ give
\begin{equation*}
R^2=
\begin{cases}
  -2N_v^2-N_v+1,&i<\ell,\\
  -N_v^2-N_v+1,&i=\ell.
\end{cases}
\end{equation*}

We also need the coefficient of $H-H_0$ in the $\ell$-th layer. Here $j$ is primitive, and $v(j^\dagger j)=2\ell\ne1$. The saturation argument in the proof of Lemma~\ref{inert odd adjacent correspondence} therefore shows that its image is saturated in both components. Write $U$ for the orthogonal lattice. Over the coefficient ring $\mathcal O$, choose lifts $w_{\mathrm{CM}}$ and $w_{\mathrm{Hdg}}$ of the two radical lines in the positive CM component and the component defining the Hodge bundle, respectively. The two components of the lower lattice $U+p_v^\ell U^\vee$ are
\begin{equation*}
U_{\mathrm{CM}}+\mathcal O w_{\mathrm{CM}}/p_v^\ell,\qquad
U_{\mathrm{Hdg}}+\mathcal O w_{\mathrm{Hdg}}/p_v^\ell.
\end{equation*}
The coefficient of $w_{\mathrm{Hdg}}$ in $Vw_{\mathrm{CM}}$ is a unit. Stability under $V$ therefore gives
\begin{equation*}
U_{\mathrm{Hdg}}+\mathcal O w_{\mathrm{Hdg}}/p_v^\ell
=
U_{\mathrm{Hdg}}+
V\bigl(U_{\mathrm{CM}}+\mathcal O w_{\mathrm{CM}}/p_v^\ell\bigr).
\end{equation*}
Thus the inclusion of $U$ in $U+p_v^\ell U^\vee$ gives a surjection between the two Lie line bundles, hence it is an isomorphism. The coefficient of $H-H_0$ in the $\ell$-th layer is therefore zero.

Write $b_0,b_i$ for the coefficients of $H-H_0$ at $Z_0$ and in layer $i$. Since $(H-H_0)\cdot R=N_v^2-1$, we obtain
\begin{equation*}
\begin{gathered}
   b_\ell=0,\qquad b_{\ell-1}=N_v^2-1\quad(\ell>1),\\
   b_{i-1}-2N_v^2b_i+N_v^4b_{i+1}=N_v^2-1\quad(2\le i<\ell),\\
   b_0+(1-2N_v^2)b_1+N_v^4b_2=N_v^2-1.
\end{gathered}
\end{equation*}
Setting the coefficients of the missing layers equal to zero, we solve the relation and have
\begin{equation*}
  b_0=\ell(N_v^2-1)N_v^{2\ell-2},\qquad b_i=(N_v^2-1)\sum_{a=0}^{\ell-i-1}(a+1)N_v^{2a}.
\end{equation*}

Each basic point of $Z_0$ contributes $(N_v+1)N_v^{4(i-1)}$ curves to layer $i$, counted with the original component and stabilizer weights. Moreover,
\begin{equation*}
T_{1,s}\cdot Z_0=-H\cdot Z_0,
\end{equation*}
while $T_{1,s}\cdot R$ is $-(N_v^2-1)$ for $i<\ell$ and $1$ for $i=\ell$. Substituting these values into
\begin{equation*}
\operatorname{div}(p_v)\cdot T_{1,s}=0
\end{equation*}
shows that the weighted number of basic points of $Z_0$ is $(N_v-1)(H\cdot Z_0)$. Hence layer $i$ contains $(N_v^2-1)(H\cdot Z_0)N_v^{4(i-1)}$ curves with these weights. Consequently,
\begin{equation*}
\begin{aligned}
\operatorname{div}(p_v)\cdot H
    &=(N_v^2-1)(H\cdot Z_0)\left(N_v^{2\ell-2}+(N_v^2-1)\sum_{i=1}^{\ell}N_v^{2\ell+2i-4}\right)\\
    &=(N_v^2-1)N_v^{4\ell-2}(H\cdot Z_0),\\
    (H-H_0)\cdot H&=(H\cdot Z_0)\left(b_0+(N_v^2-1)^2 \sum_{i=1}^{\ell}b_iN_v^{4(i-1)}\right)\\
    &=(N_v^2-1)(H\cdot Z_0)\sum_{a=0}^{\ell-1}N_v^{4a}.
\end{aligned}
\end{equation*}
Note that $Z_0$ has positive $H_0$-degree since $H_0\cdot R=0$, and
\begin{equation*}
  \operatorname{div}(p_v)\cdot H_0
  =
  \operatorname{div}(p_v)\cdot H
\end{equation*}
because $H$ and $H_0$ agree on the generic fiber. The coefficients of $Z_0$ in $H-H_0$ and $\operatorname{div}(p_v)$ have ratio $\ell$, so
\begin{equation*}
    (H-H_0)\cdot H_0=\ell\bigl(\operatorname{div}(p_v)\cdot H\bigr).
\end{equation*}
Using
\begin{equation*}
     H^2-H_0^2=(H-H_0)\cdot H+(H-H_0)\cdot H_0
\end{equation*}
and multiplying by $2\log N_v$, we get
\begin{equation*}
   \frac{(H^2-H_0^2)_v}{\operatorname{div}(p_v)\cdot H}=\left(2\ell+\frac{2N_v^2}{N_v^4-1}(1-N_v^{-4\ell})\right)\log N_v.
\end{equation*}
This proves the even formula for $n=2$.

\textbf{Step 3.} For even $s$ and $n\ge3$, we determine the Hodge divisor along the vertical prime divisors. Since its dimension estimate and deformation to the $\mu$-ordinary locus do not use the parity of $s$, the reduction in Step~3 of Lemma~\ref{inert odd adjacent correspondence} applies here as well. Write the projection of $j$ to the positive CM factor as $p_v^i u$, where $u$ is primitive and $0\le i\le\ell$. Thus, at a vertical prime, either $j/p_v$ is not integral or its generic point is $\mu$-ordinary. The first case means $i=0$, and it is simple. Note that $j$ is primitive and saturated, since $s=2\ell\ne1$. The computation of Lie bundles of Step~2 above is independent of $n$, so the Hodge comparison is a unit and all $T_{b,s}$ vanish.

We now work at a $\mu$-ordinary generic point, or equivalently $1\le i\le \ell$. Let $u:j(Y)^\perp\longrightarrow B$ be the lower isogeny of Lemma~\ref{inert lower morphism lemma}. For convenience, write $U=\ker(j^\dagger:\Lambda\to\mathcal O_{E_v})$ on the generic Tate lattice. If $w\in\Lambda$ satisfies $j^\dagger w=1$, then $p_v^{-\ell}j(1)-p_v^\ell w\in p_v^\ell U^\vee,$ and its class generates $(U+p_v^\ell U^\vee)/U$. Hence $\ker u=j(Y[p_v^\ell]).$ By the ordinary filtration described above,
\[
\ker u\cap M_A=j(Y[p_v^i])=M_A[p_v^i].
\]
Thus $u$ acts on the factor defining the Hodge bundle by $p_v^i$ times a unit. Since the $B$-component of the lower map is $p_v^\ell u$, while the fixed generic identification divides by $p_v^\ell$, the Hodge comparison has order $i\operatorname{ord}(p_v).$

The Lubin--Tate calculation in Step~2 of Lemma~\ref{inert odd adjacent correspondence} depends only on this exact divisibility. For $i>0$, it gives
\[
\operatorname{ord}(p_v)=(N_v^2-1)N_v^{2i-2},
\qquad
\operatorname{ord}(T_{b,s})=
\begin{cases}
    N_v^{2(i-b)},&b\le i,\\
    0,&b>i.
\end{cases}
\]
For $i=0$, both sides vanish. Therefore
\[
    (N_v^2-1)\sum_{b=1}^{i}N_v^{2(b-1)}N_v^{2(i-b)}=i(N_v^2-1)N_v^{2i-2}
\]
shows that the two divisors have the same coefficient at every vertical prime. Since both vanish on the generic fiber, we have
\begin{equation}\label{inert Hodge divisor formula}
\LL|_{\mathcal Z(y)}-\LL_{\mathcal Z(y)}
=(N_v^2-1)\sum_{b=1}^{\ell}N_v^{2(b-1)}T_{b,s}.
\end{equation}

\textbf{Step 4.} Now we use this equality to compute the even self-intersection difference. Expand the difference of the $n$-th powers in \eqref{D adelic relation} and group products of the $T_{b,s}$ by their largest index $b$. Equation~\eqref{inert nested conormal} replaces each further factor by $-\LL|_{\mathcal Z(y)}$. Since the sum of the coefficients up to $b$ is $N_v^{2b}-1$, the coefficient of $(\LL|_{\mathcal Z(y)})^{n-1}T_{b,s}$ is $N_v^{2nb}-N_v^{2n(b-1)}$. Thus it is not hard to check from \eqref{inert Hodge divisor formula} that
\begin{equation}\label{inert complete power difference}
   \deg\left((\LL|_{\mathcal Z(y)})^n-\LL_{\mathcal Z(y)}^n\right)=(N_v^{2n}-1)\sum_{b=1}^{\ell}N_v^{2n(b-1)}\deg\left(T_{b,s}(\LL|_{\mathcal Z(y)})^{n-1}\right).
\end{equation}
The divisor $T_{b,s}$ is defined by $p_v=0$ on the flat cycle $Z(j/p_v^b)$. Its degree is therefore the sum of the generic degrees of types $0,2,\ldots,2(\ell-b)$. Dividing by the common factors away from $v$, a simple computation gives
\begin{equation}\label{inert generic degree ratios}
   1\quad(s=0),\qquad N_v^{n(s-1)}\bigl(N_v^n-(-1)^n\bigr)\quad(s>0).
\end{equation}
Then the degree sum for $T_{b,s}$ becomes
\begin{equation*}
   1+\bigl(N_v^n-(-1)^n\bigr)
\sum_{i=1}^{\ell-b}N_v^{n(2i-1)}.
\end{equation*}
According to \eqref{inert complete power difference} we have
\begin{equation*}
   \frac{\deg\left((\LL|_{\mathcal Z(y)})^n-\LL_{\mathcal Z(y)}^n\right)}{\deg_L Z(y)}=\ell+\frac{(-1)^nN_v^n}{N_v^{2n}-1}
(1-N_v^{-2n\ell}).
\end{equation*}
The residue field of $E_v$ has cardinality $N_v^2$. Multiplying by $2\log N_v$ and using \eqref{local D0 adelic definition} proves the even case.

\textbf{Step 5.} It remains to compute the odd case by successive comparisons with type $1$. For odd $s\ge1$, apply Lemma~\ref{inert odd adjacent correspondence} to the map from type $s+2$ to type $s$. The Hodge difference is $(N_v^2-1)T_{1,s+2}$, so the same power expansion gives coefficient $N_v^{2n}-1$. The metric obtained from type $1$ pulls back under this map, so its contribution cancels in $D_{0,v}(s+2)-D_{0,v}(s)$ by the projection formula.

The degree of $T_{1,s+2}$ is the sum of the degrees of types $1,3,\ldots,s$. Using \eqref{inert generic degree ratios}, its ratio to the degree of type $s+2$ is
\begin{equation*}
   \frac{\sum_{i=0}^{(s-1)/2}(N_v^n-(-1)^n)N_v^{2ni}}{(N_v^n-(-1)^n)N_v^{n(s+1)}}=\frac{1-N_v^{-n(s+1)}}{N_v^{2n}-1}.
\end{equation*}
Multiplying by $2(N_v^{2n}-1)\log N_v$ therefore gives
\begin{equation*}\label{inert odd depth recurrence}
    \deg_L(X)\bigl(D_{0,v}(s+2)-D_{0,v}(s)\bigr)=2(1-N_v^{-n(s+1)})\log N_v,\qquad s\ge1\text{ odd}.
\end{equation*}
At type $1$ the two metrics agree, so $D_{0,v}(1)=0$. For $s=2\ell+1$, summing this equation gives
\begin{equation*}
   \deg_L(X)D_{0,v}(s)=2\left(\ell-\frac{1-N_v^{-2n\ell}}{N_v^{2n}-1}\right)\log N_v,
\end{equation*}
which finishes the proof.
\end{proof}

We now apply these two formulas for $D_{0,v}$ into the sum defining $f_{\Phi_v,a}$, with the different and Weil index as in \eqref{raw local Siegel Weil factor}.

\begin{proposition}\label{Explicit f term inert}
Suppose that $v$ is inert in $E/F$, $a\ne0$, and $r=v(a)$. Then $f_{\Phi_v,a}(1)=0$ for $r<0$. For $r\ge0$,
\begin{equation}\label{inert f equals S}
\begin{aligned}
   &2W'_{a,v}(0,1,\Phi_v)-\log|a|_v\,W_{a,v}(0,1,\Phi_v)-f_{\Phi_v,a}(1)\\
   ={}&2\left(-\frac{L'_v(n+1,\eta_v^{n+1})}{L_v(n+1,\eta_v^{n+1})}+\log|d_v|\right)W_{a,v}(0,1,\Phi_v)\\
   &+2\gamma(\BV_v,q)|d_v|^{n+\frac12}\frac{1-|d_v|}{N_v-1}\log N_v\\
   &+\begin{cases}
   0,&r\text{ even},\\[1mm]
   \displaystyle\frac{N_v^n+(-1)^n}{N_v^n-(-1)^n}\log N_v\,W_{a,v}(0,1,\Phi_v),&r\text{ odd}.
\end{cases}
\end{aligned}
\end{equation}
Especially, for even $r$, we have explicitly
\begin{equation}\label{inert explicit f formula}
\begin{aligned}
   f_{\Phi_v,a}(1)={}&\gamma(\BV_v,q)|d_v|^{n+\frac12}\bigl(1+(-1)^nN_v^{-(n+1)}\bigr)\\
   &\cdot\frac{r(1-N_v^{-n(r+2)})+(-1)^n(r+2)(N_v^{-n}-N_v^{-n(r+1)})}{(1+(-1)^nN_v^{-n})^2}\log N_v.
\end{aligned}
\end{equation}
\end{proposition}
\begin{proof}
We evaluate the orbit sum in \eqref{f series} and compare it with the Whittaker formula in \cite[Lemma~4.5]{Guo1}. The first assertion follows from $q(\Lambda_v)\subset\mathcal O_{F_v}$. Lemma~\ref{algebraic lemma inert case} gives $0\le s\le r$ and $s\equiv r\pmod2$ for the types that occur. The stabilizer factor at type $0$ is
\[
  |d_v|^{n+\frac12}\bigl(1+(-1)^nN_v^{-(n+1)}\bigr).
\]
For positive type $s$, it is multiplied by
\[
   N_v^{n(s-1)}(N_v^n-(-1)^n).
\]
When $r$ is even, we substitute the even formula for $D_{0,v}$ into \eqref{f series}. This gives \eqref{inert explicit f formula}.

For the comparison with the derivative, according to \cite[Lemma~4.5]{Guo1}, for every $r\ge0$,
\[
\begin{aligned}
   &2W'_{a,v}(0,1,\Phi_v)-\log|a|_v\,W_{a,v}(0,1,\Phi_v)\\
   &\quad+2\left(\frac{L'_v(n+1,\eta_v^{n+1})}{L_v(n+1,\eta_v^{n+1})}-\log|d_v|\right)W_{a,v}(0,1,\Phi_v)\\
   &\quad-2\gamma(\BV_v,q)|d_v|^{n+\frac12}\frac{1-|d_v|}{N_v-1}\log N_v\\
   ={}&\gamma(\BV_v,q)|d_v|^{n+\frac12}\bigl(1+(-1)^nN_v^{-(n+1)}\bigr)\log N_v\\
   &\quad\cdot\frac{r\bigl(1-(-1)^{(n+1)r}N_v^{-n(r+2)}\bigr)+(-1)^n(r+2)\bigl(N_v^{-n}-(-1)^{(n+1)r}N_v^{-n(r+1)}\bigr)}{(1+(-1)^nN_v^{-n})^2}.
\end{aligned}
\]
For even $r$, its right hand side is \eqref{inert explicit f formula}. This proves \eqref{inert f equals S} in that case.

Next, we consider the sum for odd $r$. The odd formula for $D_{0,v}$ gives
\[
\begin{aligned}
&\frac{f_{\Phi_v,a}(1)}
{\gamma(\BV_v,q)|d_v|^{n+\frac12}
(1+(-1)^nN_v^{-(n+1)})\log N_v}\\
={}&N_v^{-nr}\!\sum_{\substack{1\le s\le r\\s\text{ odd}}}
(N_v^n-(-1)^n)N_v^{n(s-1)}
\left(s-1-\frac{2(1-N_v^{-n(s-1)})}{N_v^{2n}-1}\right).
\end{aligned}
\]
After combination of the terms, this is
\[
\begin{aligned}
&\frac{
r\bigl(1-(-1)^{n+1}N_v^{-n(r+2)}\bigr)
+(-1)^n(r+2)\bigl(N_v^{-n}-(-1)^{n+1}N_v^{-n(r+1)}\bigr)}
{(1+(-1)^nN_v^{-n})^2}\\
&\qquad-\frac{N_v^n+(-1)^n}{N_v^n-(-1)^n}
\frac{1-N_v^{-n(r+1)}}{1+(-1)^nN_v^{-n}}.
\end{aligned}
\]
The first term is the quotient in the Whittaker derivative formula above for odd $r$. Moreover, the last quotient is the sum of the same orbit weights with constant summand $1$:
\[
   N_v^{-nr}\!\sum_{\substack{1\le s\le r\\s\text{ odd}}}(N_v^n-(-1)^n)N_v^{n(s-1)}=\frac{1-N_v^{-n(r+1)}}{1+(-1)^nN_v^{-n}}.
\]
By \eqref{raw local Siegel Weil factor}, multiplication by
$\gamma(\BV_v,q)|d_v|^{n+\frac12}(1+(-1)^nN_v^{-(n+1)})$
gives $W_{a,v}(0,1,\Phi_v)$. The remaining term is therefore precisely the odd contribution in \eqref{inert f equals S}.
\end{proof}

\subsection{Explicit local computations at ramified places}\label{Explicit local computations at ramified places}

Suppose that $v$ is ramified in $E/F$.  As in \cite[Assumption~2.4]{Guo2}, we assume odd residue characteristic and $F_v/\QQ_p$ unramified, and take $\varpi_{E_v}$ with $\overline{\varpi}_{E_v}=-\varpi_{E_v}$. For convenience, We often use ``modular” and``almost modular” to mean \(\varpi_{E_v}\)-modular and almost \(\varpi_{E_v}\)-modular, respectively. We first compare the Hodge bundles on the modular and almost modular models in Lemma~\ref{ramified marked Hodge comparison}. Then we construct the comparison from type $s$ to type $s-1$, for odd $n$ in Lemma~\ref{ramified odd adjacent correspondence} and for even $n$ in Lemma~\ref{ramified even adjacent correspondence}. Taking intersection numbers as in \eqref{local D0 adelic definition} gives the formulas for $D_{0,v}$ in Propositions~\ref{ramified odd local D0 formula} and~\ref{ramified even local D0 formula}. Finally, Proposition~\ref{Explicit f term ramified} evaluates the orbit sum \eqref{f series} for $\Phi_v^\vee=1_{\Lambda_v^\vee}$, and Proposition~\ref{ramified standard support comparison} returns to $\Phi_v=1_{\Lambda_v}$ for the global comparison.

We keep $A,Y$ and $j:Y\to A$ as in the inert case. All crystals, Lie algebras and filtrations in the local arguments below are taken in the relative $F_v$-factor. For such an object $A$, put
\begin{equation*}
 \mathcal I_A=\operatorname{im}
 \bigl(\iota_A(\varpi_{E_v})-\varpi_{E_v}:\Lie A\longrightarrow\Lie A\bigr),
\end{equation*}
where the second $\varpi_{E_v}$ acts through the structure sheaf. By \cite[Section~2.3]{Guo2}, the inverse Hodge bundle is
\begin{equation*}
 \operatorname{Hom}(\Fil Y,\mathcal I_A^{\vee\vee}).
\end{equation*}
For a $\varpi_{E_v}$-modular lattice, $\mathcal I_A$ is already an invertible direct summand. For an almost $\varpi_{E_v}$-modular lattice, the double dual gives its extension on the regular model. On general levels and proper refinements, we pull back this line bundle.

The construction in \cite[Section~9]{RSZ1} gives two embeddings $\delta^\pm$ from the almost $\varpi_{E_v}$-modular model to the $\varpi_{E_v}$-modular model of one higher rank. On the generic fiber, both embeddings are obtained by adding a fixed positive CM factor. We write $\LL_{\mathrm{almost}}$ and $\LL_{\mathrm{mod}}$ for the respective Hodge bundles. The following comparison will identify the bundle at type $0$ when $n$ is odd and allow us to restrict the calculation to even $n$.

\begin{lemma}\label{ramified marked Hodge comparison}
The embeddings $\delta^\pm$ give canonical isomorphisms
\begin{equation*}
 (\delta^+)^*\LL_{\mathrm{mod}}
 \simeq\LL_{\mathrm{almost}}
 \simeq(\delta^-)^*\LL_{\mathrm{mod}}.
\end{equation*}
They agree with the fixed tautological identification on the generic fiber and keep the same fixed CM abelian variety $Y$.
\end{lemma}

\begin{proof}
We first obtain a map from the isogeny defining $\delta^\pm$, then compute its divisor on the projective local model.

For either embedding, \cite[Theorem~5.3]{Yao} gives an isogeny from the modular object to the product of the almost modular object and the added positive CM factor. Its differential commutes with $\iota(\varpi_{E_v})-\varpi_{E_v}$, which vanishes on that factor. Therefore, it induces a map between the image lines. Taking double duals and applying $\operatorname{Hom}(\Fil Y,-)$ gives
\begin{equation*}
 (\delta^\pm)^*\LL_{\mathrm{mod}}^{-1}
 \longrightarrow\LL_{\mathrm{almost}}^{-1}.
\end{equation*}
This construction also works on the projective local model: the lattice inclusion in \cite[(12.2), proof of Proposition~12.1]{RSZ1} gives the map on Lie quotients, which is compatibly with the local model diagram. On the generic fibre, the polarization compatibility in \cite[(9.4)--(9.5)]{RSZ1} shows that this map agrees with the fixed tautological identification.

After passing to the strict henselization, the almost modular projective local model is proper and smooth with connected generic fiber; see \cite[Theorem~7.5 and its proof]{RSZ1}. Its special fiber is consequently connected and irreducible. Thus the divisor of the map is a nonnegative multiple of this fiber. To determine the multiple, take the point obtained from a modular object by adding a positive CM factor. In \cite[(9.4)]{RSZ1}, the two choices of isogeny act on the resulting plane of positive CM factors and are the identity on the remaining modular summand. The image line on this summand has, on its rank-two block, the description
\begin{equation*}
 \begin{pmatrix}-\varpi_{E_v}&\varpi_{E_v}^2\\1&-\varpi_{E_v}\end{pmatrix}
 =\binom{-\varpi_{E_v}}{1}(1,-\varpi_{E_v});
\end{equation*}
see \cite[Proposition~3.2]{Guo2}. Both isogenies fix its generator $(-\varpi_{E_v},1)^t$. Therefore, this map is nonzero at this point of the special fiber, so its divisor is zero. Dualizing and descending through the local model diagram proves this assertion. We also remind the reader that further pullback preserves both the isomorphism and its generic identification.
\end{proof}

\subsubsection*{The case when \texorpdfstring{$n$}{n} is odd}

Assume that $n\ge3$ is odd, so the ambient lattice is $\varpi_{E_v}$-modular. For each $s\ge0$, choose a primitive vector $y_{v,s}\in\Lambda_v^\vee$ of type $s$, meaning that its norm has valuation $s$. The corresponding homomorphism is denoted by $j$. As in the inert case, we define the primitive divisor and its intersection with the preceding special cycle by
\begin{equation*}
 Z(j)=Z(j/\varpi_{E_v})+\mathcal D_s,\qquad
 T_{1,s}=Z(j/\varpi_{E_v})|_{\mathcal D_s}\quad(s>0).
\end{equation*}
These are the Cartier cycles defined by the lifting condition, with their scheme multiplicities. The next lemma is the ramified analogue of Lemma~\ref{inert odd adjacent correspondence}. More precisely, it compares the bundles at type $s$ and type $s-1$, and their difference gives the corresponding difference in $D_{0,v}$. 

\begin{lemma}\label{ramified odd adjacent correspondence}
Let $s>0$. On a proper common refinement of the primitive models of type $s$ and type $s-1$, there is a correspondence $\rho_s$, of degree $N_v^n$ on the generic fiber, such that
\begin{equation}\label{ramified odd adjacent Hodge}
 \LL|_{\mathcal D_s}-\rho_s^*(\LL|_{\mathcal D_{s-1}})
 =(N_v-1)T_{1,s}.
\end{equation}
The adelic lower Hodge bundles in Definition~\ref{adelic lower Hodge line definition} agree under pullback by $\rho_s$.
\end{lemma}

\begin{proof}
We briefly explain the strategy of the proof. We first construct $\rho_s$ using two isogenies with a common target, then compute the lengths of the cokernels of the two induced maps on Lie algebras, and hence the Hodge difference. Finally, we compare the lattices defining the lower Hodge bundles.

\noindent\textbf{Step 1.} To begin with, we construct the correspondence and identify $T_{1,s}$. By symmetrizing the polarization as in \cite[Remark~2.2]{LL2}, the dual special vector gives an actual homomorphism $j:Y\to A$. The flatness argument preceding Lemma~\ref{inert odd adjacent correspondence} applies to these special cycles as follows. Their morphisms to the ambient model are finite, and their local branches are Cartier. After a finite extension of $E$, each geometric connected component of the generic fiber contains a CM point whose reduction at $v$ is basic. Proper smooth specialization then shows that every component of the special fibre contains a basic point. A vertical component of a special cycle would map onto an ambient component and meet such a point. On the basic completion, the relative comparison identifies the cycle with the relative Cartier divisor of \cite[Lemma~2.40]{LL2}. Thus the special cycles have no vertical irreducible components. Since a Cartier divisor on the regular ambient model has no embedded components, this implies that the special cycles are flat.

Let $f_0,f_1$ be local equations of $Z(j)$ and $Z(j/\varpi_{E_v})$. The obstruction in \cite[Lemma~2.39]{LL2} takes values in $\operatorname{Hom}(\Fil Y,\mathcal I_A)$ and is linear in $j$. Therefore, the conormal calculation used before Lemma~\ref{inert odd adjacent correspondence} with $\varpi_{E_v}$ in place of $p_v$ in this case gives
\begin{equation*}
    (f_0/f_1)|_{Z(j/\varpi_{E_v})} =\varpi_{E_v}\cdot\text{a unit}.
\end{equation*}
Thus $T_{1,s}$ is the full special fiber of $Z(j/\varpi_{E_v})$. The same calculation identifies its conormal line bundle and gives
\begin{equation}\label{ramified odd conormal powers}
   T_{1,s}^{\,k}=(-\LL|_{\mathcal D_s})^{k-1}\cdot T_{1,s}\qquad(k\ge1).
\end{equation}

Now we take quotient of $A$ by $j(Y[\varpi_{E_v}])$ on the generic fiber. In the auxiliary lattice chain of \cite[Section~9]{RSZ1}, this quotient has two modular objects above it; one is $A$, and we denote the other by $A'$. Denote their isogenies to the common quotient by $\alpha$ and $\alpha'$. The kernel is an isotropic line in the kernel of the modular polarization, so the descended polarization has kernel of order $N_v^{n-1}$, killed by $\varpi_{E_v}$. These are the lattice and polarization conditions in the auxiliary problem of \cite[Theorem~5.3]{Yao}. We take the closures of the prescribed generic choices; the determinant, wedge and strengthened spin conditions are closed conditions and hence hold on these flat closures. The isometry $(\alpha')^{-1}\alpha$ identifies the rational Hermitian spaces, and
\begin{equation*}
     j'=(\alpha')^{-1}\alpha j/\varpi_{E_v}
\end{equation*}
is primitive of type $s-1$. In an adapted hyperbolic basis, this is the change $X_1\mapsto X_1/\varpi_{E_v}$, $Y_1\mapsto-\varpi_{E_v}Y_1$. The calculation in Step~4 below also checks the norm and primitivity. For a fixed target vector, the inverse choices are lines outside its orthogonal hyperplane in the residue symplectic space. Hence
\begin{equation*}
    \deg\rho_s=\frac{N_v^{n+1}-1}{N_v-1}-\frac{N_v^n-1}{N_v-1}=N_v^n.
\end{equation*}
We take the closure of this generic correspondence and a normal proper common refinement. Since the two Hodge bundles are already identified on the generic fiber, their difference is vertical. On the normal common refinement it is therefore enough to compute its order at each vertical prime divisor. Any further finite pullback multiplies both the relevant intersections and degrees by the degree of the map, so the normalized calculation is unchanged by the projection formula.

\noindent\textbf{Step 2.} Now we the lengths of the cokernels of the two induced maps on Lie algebras. Let $\mathcal O$ be the discrete valuation ring at such a prime, and let $f$ be the pullback of a local equation for $Z(j/\varpi_{E_v})$. We may take the completion of $\mathcal O$ for this calculation. The generic isogenies extend over $\mathcal O$; the common target $A/\overline{j(Y[\varpi_{E_v}])}$ is the quotient by the finite flat closure of $j(Y[\varpi_{E_v}])$.

Choose a Lubin--Tate coordinate $z_0$ on $Y$ in which its $\varpi_{E_v}$-torsion has equation $z_0^{N_v}+\varepsilon\varpi_{E_v}z_0=0$, with $\varepsilon\in\mathcal O^\times$. In coordinates compatible with the action of the units in the residue field, we have $j|_{Y[\varpi_{E_v}]}$ as $(a_1z_0,\ldots,a_{n+1}z_0)$. Note that the coefficients $a_i$ generate $(f)$. Indeed, vanishing of this restriction is exactly the condition that $j$ factor through $[\varpi_{E_v}]:Y\to Y$, which defines $Z(j/\varpi_{E_v})$. The scheme-theoretic image consequently has the description by algebra
\begin{equation*}
 \mathcal O[fz_0]
 \simeq\mathcal O[z]/
 \bigl(z^{N_v}+\varepsilon\varpi_{E_v}f^{N_v-1}z\bigr).
\end{equation*}
Since $f\ne0$, this algebra is free of rank $N_v$ and is the closure of the generic kernel. Its group law is induced by $j$. Taking the cotangent space at the identity, or equivalently the determinant of the relative Lie map, gives
\begin{equation*}
 \operatorname{length}_{\mathcal O}\operatorname{coker}(\Lie\alpha)
 =v_{\mathcal O}(\varpi_{E_v})+(N_v-1)v_{\mathcal O}(f).
\end{equation*}

For the other length, we use the lattice chain of \cite[Section~9]{RSZ1}. On its nontrivial  block of rank two, choose generators $e_1,e_2$ such that the two modular lattices are generated by
\[
 e_1,\ \iota(\varpi_{E_v})e_2,\qquad\text{and}\qquad \iota(\varpi_{E_v})e_1,\ e_2,
\]
inside the lattice generated by $e_1,e_2$. The dual of the latter is generated by
\[
 \iota(\varpi_{E_v})e_1,\qquad \iota(\varpi_{E_v})e_2.
\]
Its discriminant module is
\begin{equation*}
 (\mathcal O/\varpi_{E_v}^2)^2,\qquad[(x_1,x_2),(y_1,y_2)]=\frac{x_1y_2+x_2y_1}{\varpi_{E_v}^2}\pmod{\mathcal O},
\end{equation*}
up to a unit. Modulo the dual lattice, the two modular lattices are generated respectively by the classes of $e_1$ and $e_2$. In fact, the image of the Hodge filtration in the discriminant module above is its own annihilator. Indeed, an integral functional on the filtration extends to the common lattice because the filtration is a direct summand; the dual lattice represents that extension. Subtracting it leaves an element of the generic orthogonal complement of the filtration, which is the filtration itself.

Let $m$ be the valuation of the inverse image in $\mathcal O$ of the first coordinate ideal of the image from Hodge filtration, so $0\le m\le2v_{\mathcal O}(\varpi_{E_v})$. Suppose $(x_1,x_2)$ is a vector in the image of the Hodge filtration whose first coordinate has valuation \(m\). Since it is isotropic, $x_2$ has valuation at least $2v_{\mathcal O}(\varpi_{E_v})-m$. For any $(y_1,y_2)$ in the image of the Hodge filtration, pairing it with $(x_1,x_2)$ gives the same bound for $y_2$. Since the image is its own annihilator, it has length $2v_{\mathcal O}(\varpi_{E_v})$. The bound is therefore an equality, and the image is the sum of these coordinate ideals. The two Lie cokernel lengths are therefore complementary:
\begin{equation*}
 \operatorname{length}_{\mathcal O}\operatorname{coker}(\Lie\alpha)
 +\operatorname{length}_{\mathcal O}\operatorname{coker}(\Lie\alpha')
 =2v_{\mathcal O}(\varpi_{E_v}).
\end{equation*}
It follows that
\begin{equation*}
 \operatorname{length}_{\mathcal O}\operatorname{coker}(\Lie\alpha')
 =v_{\mathcal O}(\varpi_{E_v})-(N_v-1)v_{\mathcal O}(f).
\end{equation*}

\noindent\textbf{Step 3.} Now we pass from the relative Lie determinants to the Hodge bundle. Write $\mathbb D(A)$ for the relative de Rham crystal and put
\begin{equation*}
 \mathcal F_A=\ker(\iota_A(\varpi_{E_v})-\varpi_{E_v}\mid\Lie A).
\end{equation*}
The symmetrized polarization identifies $\Fil A$ with $(\Lie A)^\vee$. Denote by $\mathcal F_A^\perp\subset\Fil A$ the annihilator of $\mathcal F_A$. By \cite[Lemmas~2.36--2.37 and~2.41]{LL2}, it is dual to $\mathcal I_A$, and there is an exact sequence
\begin{equation*}
 0\longrightarrow\mathcal I_A\longrightarrow\Lie A
 \longrightarrow
 \frac{(\iota_A(\varpi_{E_v})+\varpi_{E_v})\mathbb D(A)}
 {\mathcal F_A^\perp}
 \longrightarrow0.
\end{equation*}
Taking determinants gives
\begin{equation*}
 \det\Lie A\simeq\mathcal I_A^{\otimes2}\otimes
 \det\bigl((\iota_A(\varpi_{E_v})+\varpi_{E_v})\mathbb D(A)\bigr).
\end{equation*}
The same identity holds for $A'$.

On $(\iota_A(\varpi_{E_v})+\varpi_{E_v})\mathbb D(A),$ the operator $\iota_A(\varpi_{E_v})$ acts as multiplication by $\varpi_{E_v}$. Hence the two lattice inclusions in Step~2 induce determinant maps of the same order $v_{\mathcal O}(\varpi_{E_v})$. Their contributions to the comparison of $\det\Lie A$ and $\det\Lie A'$ cancel. The difference of the two Lie cokernel lengths from Step~2 is therefore twice the order of $(\alpha')^{-1}\alpha$ on $\mathcal I_A$. That order is $(N_v-1)v_{\mathcal O}(f)$. Since $Y$ is unchanged, applying $\operatorname{Hom}(\Fil Y,-)$ and dualizing proves \eqref{ramified odd adjacent Hodge} at every vertical prime. The fixed generic identifications then give the divisor identity on the common refinement.

\noindent\textbf{Step 4.} Finally, we show that $\rho_s$ identifies the lower lattices used to define the $v$-adic metric in Definition~\ref{adelic lower Hodge line definition}, hence prove the last statement. Use the hyperbolic basis $X_1,Y_1$ of Lemma~\ref{algebraic lemma ramified}, with $\langle X_1,Y_1\rangle=\varpi_{E_v}$. In the dual lattice the vector and its orthogonal vector are
\begin{equation*}
 y_{v,s}=\varpi_{E_v}^{-1}X_1+\frac{q(y_{v,s})}{2}Y_1,
 \qquad
 e_s=X_1-\frac{\varpi_{E_v}q(y_{v,s})}{2}Y_1.
\end{equation*}
The vector $w_s=\varpi_{E_v}^{-s-1}e_s$ has unit norm and defines the lower lattice in Subsection~\ref{Adelic Hodge line section}. Under $(\alpha')^{-1}\alpha$, take $X_1'=X_1/\varpi_{E_v}$ and $Y_1'=-\varpi_{E_v}Y_1$. Then
\begin{equation*}
 y_{v,s}/\varpi_{E_v}
 =\varpi_{E_v}^{-1}X_1'
  +\frac{q(y_{v,s}/\varpi_{E_v})}{2}Y_1',
 \qquad e'_{s-1}=e_s/\varpi_{E_v},
\end{equation*}
and hence
\begin{equation}\label{ramified lower vector compatibility}
 w'_{s-1}=\varpi_{E_v}^{-s}e'_{s-1}=w_s.
\end{equation}
The remaining modular summands are unchanged. Thus the lower lattices, levels and generic maps $f_y$ agree after composition with $\rho_s$. After identifying the two lower lattices, the maps to the corresponding lower-dimensional integral model obtained from the two sides agree on the generic fibre. Since the target is separated, they agree on the common refinement, and hence give the same pullback of the Hodge bundle. At type $0$, Lemma~\ref{ramified marked Hodge comparison} identifies this pullback with the almost modular Hodge bundle.
\end{proof}

We now compute the intersection numbers in \eqref{ramified odd adjacent Hodge}. We will use the degree factorization \eqref{degree local factorization}. With its Haar measures, the local stabilizer of a primitive vector of type $s$ satisfies
\begin{equation}\label{ramified odd volume}
 \frac{\vol(U_v)}{\vol(U_{y,v})}
 =(1-N_v^{-(n+1)})N_v^{ns}.
\end{equation}
Here $y_v$ is primitive of type $s$, and $U_{y,v}$ fixes the vector. Indeed, at norm valuation $0$, only type $0$ occurs; \eqref{raw local Siegel Weil factor} and \cite[Corollary~4.7(1)]{Guo1} give the factor $1-N_v^{-n-1}$. Iterating the degree $N_v^n$ in Lemma~\ref{ramified odd adjacent correspondence} then gives \eqref{ramified odd volume}.

\begin{proposition}\label{ramified odd local D0 formula}
Suppose that $n$ is odd and $y_v\in\Lambda_v^\vee$ has primitive type $s$. Then
\begin{equation}\label{ramified odd D0 formula}
 \deg_L(X)D_{0,v}(y_v)
 =\left(s-\frac{1-N_v^{-ns}}{N_v^n-1}\right)\log N_v.
\end{equation}
\end{proposition}

\begin{proof}
We first determine the value at type $0$. At type $0$, the complementary lattice in \cite[Theorem~5.5 and Proposition~5.16]{Yao} is generated by $w_0=e_0/\varpi_{E_v}$ and the unchanged modular blocks. It is the lattice defining the lower-dimensional bundle. Lemma~\ref{ramified marked Hodge comparison} therefore gives $D_{0,v}(y_{v,0})=0$.

For $s>0$, substitute \eqref{ramified odd adjacent Hodge} into
$(\LL|_{\mathcal D_s})^n-(\rho_s^*(\LL|_{\mathcal D_{s-1}}))^n$.
By \eqref{ramified odd conormal powers}, the binomial coefficient sum is
$\sum_{k=1}^n\binom nk(N_v-1)^k=N_v^n-1$. Thus
\begin{equation}\label{ramified odd adjacent power}
 (\LL|_{\mathcal D_s})^n
 -(\rho_s^*(\LL|_{\mathcal D_{s-1}}))^n
 =(N_v^n-1)(\LL|_{\mathcal D_s})^{n-1}\cdot T_{1,s}.
\end{equation}
Flatness identifies the degree on $T_{1,s}$ with the generic Hodge degree of $Z(j/\varpi_{E_v})$. Its generic types are $0,\ldots,s-1$. Using \eqref{degree local factorization} and \eqref{ramified odd volume}, with the same component weights, gives
\begin{equation*}
 \frac{\deg\bigl(T_{1,s}\cdot(\LL|_{\mathcal D_s})^{n-1}\bigr)}
 {\deg_L Z(y)}
 =\frac{\sum_{k=0}^{s-1}N_v^{nk}}{N_v^{ns}}
 =\frac{1-N_v^{-ns}}{N_v^n-1}.
\end{equation*}
Here $y$ has primitive type $s$ at $v$; the common degree factors away from $v$ cancel. The lower Hodge bundle used in Definition~\ref{adelic lower Hodge line definition} is preserved under $\rho_s$, so the projection formula and \eqref{local D0 adelic definition} turn \eqref{ramified odd adjacent power} into
\begin{equation*}
 \deg_L(X)\bigl(D_{0,v}(y_{v,s})-D_{0,v}(y_{v,s-1})\bigr)
 =(1-N_v^{-ns})\log N_v.
\end{equation*}
Summing for types $1,\ldots,s$ proves \eqref{ramified odd D0 formula}.
\end{proof}

\subsubsection*{The case when \texorpdfstring{$n$}{n} is even}

Assume now that $n$ is even. Recall that we write $B_s^\pm$ for the two type-$s$ orbits in $\Lambda_v^\vee$ described in Corollary~\ref{algebraic corollary ramified}, where the sign $+$ means $q(y_v)\in\epsilon(\Lambda_v)\Nm(E_v^\times)$, and $-$ means the other norm class. This sign is independent of the choice of $\delta^\pm$. The primitive representative of $B_s^\pm$ has norm valuation $s\ge0$ in the dual hyperbolic summand. We also write $D_{0,v}(B_s^\pm)$ for $D_{0,v}(y_v)$ such that $y_v\in B_s^+$. These two classes will have the same Hodge comparison but different degrees for the preceding special cycle.

In the $+$ class there is also an orbit in the self-dual line $E_ve$. For its primitive vector $e$, the associated homomorphism has unit norm. The integral self-adjoint idempotent $j(j^\dagger j)^{-1}j^\dagger$ gives an orthogonal decomposition with one factor $Y$ and the other the $\varpi_{E_v}$-modular orthogonal complement. Lemma~\ref{ramified marked Hodge comparison} identifies its Hodge metric with the restriction of $\LL$. Scaling the vector leaves this complement and the prescribed metric unchanged. Hence
\begin{equation}\label{ramified even line reference}
 D_{0,v}(y_v)=0\qquad(y_v\in E_ve\setminus\{0\}).
\end{equation}

For $y_v\in B_s^\pm$, let $Z(j)$ denote the lifting locus of the map $Y\to A^\vee$ obtained from the corresponding quasi-homomorphism via the polarization. This is the $\mathcal Y$-cycle of \cite[Section~2.4]{Yao}. We define $\mathcal D_s$ and $T_{1,s}$ from $Z(j)$ and $Z(j/\varpi_{E_v})$ as above, separately in each norm class. By \cite[Proposition~5.17]{Yao}, these full lifting cycles are the scheme-theoretic restrictions of the modular lifting cycles under $\delta^\pm$. The following lemma is obtained by restricting the Hodge comparison in Lemma~\ref{ramified odd adjacent correspondence} from the $\varpi_{E_v}$-modular model to the present almost $\varpi_{E_v}$-modular model. We can see that the results are essentially parallel.

\begin{lemma}\label{ramified even adjacent correspondence}
For $s>0$, there is a correspondence $\rho_s$ from the primitive model for $B_s^\pm$ to that for $B_{s-1}^\pm$, of generic degree $N_v^n$. On a proper common refinement,
\begin{equation}\label{ramified even adjacent Hodge}
 \LL|_{\mathcal D_s}-\rho_s^*(\LL|_{\mathcal D_{s-1}})
 =(N_v-1)T_{1,s}.
\end{equation}
The lower Hodge bundles of Definition~\ref{adelic lower Hodge line definition} agree under pullback, including when $s=1$.
\end{lemma}

\begin{proof}
The main idea is to restrict the construction of Lemma~\ref{ramified odd adjacent correspondence} through $\delta^\pm$, and use Lemma~\ref{ramified marked Hodge comparison} to identify the resulting bundles. We check that the special cycles and the lower-dimensional Hodge metric are preserved under this restriction.

The embedding $\delta^\pm$ adds a fixed orthogonal unit line and passes from rank $n+1$ to the $\varpi_{E_v}$-modular rank $n+2$. In the larger model, we use $j$ for the actual homomorphism corresponding to the $\mathcal Y$-cycle under \cite[Lemma~3.11]{Yao}, as in Step~1 of Lemma~\ref{ramified odd adjacent correspondence}. By \cite[Proposition~5.17]{Yao}, the $\mathcal Y$-cycle of a nonzero vector is the scheme-theoretic pullback of its cycle in the larger space. The projection to the added unit factor vanishes on $j(Y[\varpi_{E_v}])$, since this factor is orthogonal to $j$. Therefore, it descends to the common quotient in the odd-dimensional construction. The modification constructed in Step~1 of Lemma~\ref{ramified odd adjacent correspondence} changes only the hyperbolic plane containing $j$ and leaves the added unit factor unchanged. Hence the two modular objects $A$ and $A'$ constructed there remain in the image of $\delta^\pm$, so the correspondence restricts to the smaller moduli space.

The change $X_1\mapsto X_1/\varpi_{E_v}$, $Y_1\mapsto-\varpi_{E_v}Y_1$ lowers the type by one. To compute the degree of the restricted correspondence, fix a target of type $B_{s-1}^\pm$. Its preimages correspond to lines in the residue symplectic space that are orthogonal to the added unit vector but not orthogonal to the reduction of the special vector. The first condition defines a hyperplane of dimension $n+1$. Since the reductions of the added unit vector and the special vector are independent for $B_{s-1}^\pm$, the vectors also orthogonal to the special vector form a proper hyperplane inside it. Consequently
\begin{equation*}
 \deg\rho_s
 =\frac{N_v^{n+1}-1}{N_v-1}-\frac{N_v^n-1}{N_v-1}=N_v^n.
\end{equation*}
For the orbit in $E_ve$, these conditions coincide, so this construction gives no inverse choice. Thus, that orbit was already handled in \eqref{ramified even line reference}.

Now we restrict the Hodge identity as an identity of Cartier divisors. The special cycles are flat by the argument in Step~1 of Lemma~\ref{ramified odd adjacent correspondence}, where the $\mathcal Y$-cycle is a Cartier divisor by \cite[Remark~1.10]{Yao}. In detail, their finite morphisms to the ambient model and the basic CM point on each fiber component exclude vertical components. The scheme-theoretic pullback in \cite[Proposition~5.17]{Yao} identifies their lifting equations and obstruction sections. It follows that the quotient of the equations of $Z(j)$ and $Z(j/\varpi_{E_v})$, restricted to the latter, is $\varpi_{E_v}$ times a unit. Thus $T_{1,s}$ is its full special fiber, and \eqref{ramified odd conormal powers} also holds for the present $T_{1,s}$.

Therefore, we can pull back \eqref{ramified odd adjacent Hodge} as a Cartier identity: the primitive cycle is flat and its generic fiber avoids the vertical divisor. Lemma~\ref{ramified marked Hodge comparison} identifies both Hodge pullbacks, including their generic identifications. This gives \eqref{ramified even adjacent Hodge}.

It remains to check that the lower-dimensional Hodge metric is preserved under $\rho_s$. By \eqref{ramified lower vector compatibility}, the unit vector $e$ and the orthogonal vector $w_s$ used in Subsection~\ref{Adelic Hodge line section} are preserved by the modification. In the $-$ class, the lower-dimensional Hodge metric is pulled back from the fixed type-$0$ metric, so it is preserved under $\rho_s$. In the $+$ class, the plane spanned by $e,w_s$ is split. Choose the same hyperbolic basis of this fixed plane for both types, and take the modular lattice chosen in Subsection~\ref{Adelic Hodge line section}, then the other modular blocks are unchanged. The lower lattice and level are therefore identical at the two types. The generic maps $f_y$ agree after composition with $\rho_s$, and separatedness identifies their extensions on a common proper graph. This finishes the proof.
\end{proof}

We next determine the degrees from \eqref{ramified even adjacent Hodge}. For the vector stabilizers in \eqref{degree local factorization}, we have
\begin{equation}\label{ramified even orbit weights}
 \frac{\vol(U_v)}{\vol(U_{y,v})}
 =\begin{cases}
 N_v^{-\frac12}(1-N_v^{-n})N_v^{ns},&y_v\in B_s^\pm,\\[1mm]
 2N_v^{-n-\frac12},&y_v\in E_ve\setminus\{0\}.
 \end{cases}
\end{equation}
At type $0$, the normal forms in Lemma~\ref{algebraic lemma ramified} give the same stabilizer index for the two hyperbolic orbits. Indeed, their reductions give the same counting problem in the residue symplectic space. At norm valuation $0$, the $-$ class has only this orbit. Thus \eqref{raw local Siegel Weil factor} and \cite[Corollary~4.7(3)]{Guo1} give its weight $N_v^{-1/2}(1-N_v^{-n})$. The total weight in the $+$ class is $N_v^{-1/2}(1+N_v^{-n})$, so the difference is the contribution of the orbit in the self-dual line. The degree $N_v^n$ of $\rho_s$ yields the factor at general $s$.

Now, Lemma~\ref{ramified even adjacent correspondence} reduces $D_{0,v}$ to its value at type $0$. In the $-$ class, the value at type $0$ is zero, since the metric there is induced by $\LL|{\mathcal D_0}$. The corresponding normalized height is included in $c\sigma$ in \eqref{reference height constant}. For the $+$ class, we leave the value at type $0$ as a constant, and for convenience we make the following definition.

\begin{definition}\label{ramified even constants}
We put
\begin{equation*}
 c_{+,v}=\frac{\deg_L(X)D_{0,v}(B_0^+)}{\log N_v}.
\end{equation*}
\end{definition}
In fact, we will not compute this constant by another method until the next section. 

\begin{proposition}\label{ramified even local D0 formula}
Suppose that $n$ is even. For $s\ge0$, with the metrics of Definition~\ref{adelic lower Hodge line definition},
\begin{equation}\label{ramified even D0 formula}
\begin{aligned}
 \deg_L(X)D_{0,v}(B_s^+)
 &=\left(c_{+,v}+s+\frac{1-N_v^{-ns}}{N_v^n-1}\right)\log N_v,\\
 \deg_L(X)D_{0,v}(B_s^-)
 &=\left(s-\frac{1-N_v^{-ns}}{N_v^n-1}\right)\log N_v.
\end{aligned}
\end{equation}
\end{proposition}

\begin{proof}
By Step~2 of Lemma~\ref{ramified even adjacent correspondence}, the powers of $T_{1,s}$ satisfy the same relation as in the odd case. Hence the calculation giving \eqref{ramified odd adjacent power} applies to \eqref{ramified even adjacent Hodge} as well. The only difference is the degree of $T_{1,s}$.

For $s>0$, the old special cycle has hyperbolic types $0,\ldots,s-1$, together with the extra orbit in the $+$ class. By flatness, \eqref{degree local factorization} and \eqref{ramified even orbit weights} we have
\begin{equation*}
\begin{aligned}
 \frac{\deg\bigl(T_{1,s}\cdot(\LL|_{\mathcal D_s})^{n-1}\bigr)}
 {\deg_L Z(y)}
 &=\begin{cases}
 \displaystyle\frac{(1-N_v^{-n})\sum_{k=0}^{s-1}N_v^{nk}+2N_v^{-n}}
 {(1-N_v^{-n})N_v^{ns}},&y_v\in B_s^+,\\[3mm]
 \displaystyle\frac{(1-N_v^{-n})\sum_{k=0}^{s-1}N_v^{nk}}
 {(1-N_v^{-n})N_v^{ns}},&y_v\in B_s^-,
 \end{cases}\\[2mm]
 &=\frac{1\pm N_v^{-ns}}{N_v^n-1}.
\end{aligned}
\end{equation*}
The same component weights and degree factors away from $v$ occur in numerator and denominator. Since the lower-dimensional Hodge metric is preserved under $\rho_s$, \eqref{local D0 adelic definition} and \eqref{ramified odd adjacent power} now yield
\begin{equation*}
 \deg_L(X)\bigl(D_{0,v}(B_s^\pm)-D_{0,v}(B_{s-1}^\pm)\bigr)
 =(1\pm N_v^{-ns})\log N_v.
\end{equation*}
Summing these expressions and using Definition~\ref{ramified even constants} and $D_{0,v}(B_0^-)=0$ proves \eqref{ramified even D0 formula}.
\end{proof}

The computation from arithmetic intersection is now complete up to $c_{+,v}$. We substitute \eqref{ramified odd D0 formula} and \eqref{ramified even D0 formula} into the orbit sum \eqref{f series}. The resulting formulas use the dual lattice. The constant $c_{+,v}$ and the heights retained in $c_\sigma$ will be determined in Section~\ref{ramified matching endpoint subsection}.

\begin{proposition}\label{Explicit f term ramified}
Let $a\ne0$ and use $\Phi_v^\vee$. Then $f_{\Phi_v^\vee,a}(1)=0$ for $v(a)<0$. For $r=v(a)\ge0$, the following formulas hold.
\begin{enumerate}
\item If $n$ is odd, then
\begin{equation*}
\begin{aligned}
 f_{\Phi_v^\vee,a}(1)
 ={}&2W'_{a,v}(0,1,\Phi_v^\vee)\\
 &+\left(2\frac{\zeta'_v(n+1)}{\zeta_v(n+1)}-\log|a|_v\right)
 W_{a,v}(0,1,\Phi_v^\vee).
\end{aligned}
\end{equation*}
\item If $n$ is even, then
\begin{equation}\label{ramified even dual comparison}
\begin{aligned}
 &f_{\Phi_v^\vee,a}(1)-2W'_{a,v}(0,1,\Phi_v^\vee)
 +\log|a|_vW_{a,v}(0,1,\Phi_v^\vee)\\
 &\quad=\gamma(\BV_v,q)N_v^{-\frac12}\log N_v
 \begin{cases}
 c_{+,v}+(2-c_{+,v})N_v^{-n(r+1)},&+\text{ class},\\[1mm]
 \displaystyle-\frac{2(1-N_v^{-n(r+1)})}{N_v^n-1},&-\text{ class}.
 \end{cases}
\end{aligned}
\end{equation}
In the $-$ class the right hand side is
$-2\log N_v\,W_{a,v}(0,1,\Phi_v^\vee)/(N_v^n-1)$.
\end{enumerate}
\end{proposition}

\begin{proof}
This result is parallel to Proposition \ref{Explicit f term} and \ref{Explicit f term inert}. The dual lattice has integral norms, which gives the first assertion for $r<0$.

For odd $n$, \eqref{f series}, \eqref{ramified odd volume} and \eqref{ramified odd D0 formula} give
\begin{equation*}
\begin{aligned}
   f_{\Phi_v^\vee,a}(1)={}&(1-N_v^{-n-1})\sum_{s=0}^{r}N_v^{n(s-r)}\left(s-\frac{1-N_v^{-ns}}{N_v^n-1}\right)\log N_v\\
   ={}&\frac{1-N_v^{-n-1}}{(1-N_v^{-n})^2}\Bigl(r(1-N_v^{-n(r+2)})\\
   &\hspace{32mm}-(r+2)(N_v^{-n}-N_v^{-n(r+1)})\Bigr)\log N_v.
\end{aligned}
\end{equation*}
Here $\gamma(\BV_v,q)=1$. Differentiating $W_{a,v}(t,1,\Phi_v^\vee)$ which is \cite[Corollary~4.7(1)]{Guo1}, proves part~(1).

For even $n$, the extra orbit contributes zero by \eqref{ramified even line reference}. Equations~\eqref{f series} and \eqref{ramified even orbit weights} give
\begin{equation*}
   f_{\Phi_v^\vee,a}(1)=\gamma(\BV_v,q)\deg_L(X)N_v^{-\frac12}(1-N_v^{-n})\sum_{s=0}^{r}N_v^{n(s-r)}D_{0,v}(B_s^\pm).
\end{equation*}
The analytic term from \cite[Corollary~4.7(3)]{Guo1} is
\begin{equation*}
\begin{aligned}
   &2W'_{a,v}(0,1,\Phi_v^\vee)-\log|a|_vW_{a,v}(0,1,\Phi_v^\vee)\\
   &\qquad=\gamma(\BV_v,q)N_v^{-\frac12}\bigl(r\mp(r+2)N_v^{-n(r+1)}\bigr)\log N_v,
\end{aligned}
\end{equation*}
where the upper sign is for the $+$ class. Substituting \eqref{ramified even D0 formula} in the finite sum and subtracting this expression gives \eqref{ramified even dual comparison}. Finally, in the $-$ class,
$W_{a,v}(0,1,\Phi_v^\vee)=\gamma(\BV_v,q)N_v^{-1/2}(1-N_v^{-n(r+1)})$,
which gives the last assertion.
\end{proof}

We finish by changing the support from $\Lambda_v^\vee$ to $\Lambda_v$. This changes the orbit sum, but it does not change the stabilizer or the lower-dimensional Hodge metric associated with each vector. To state the coefficients needed in Section~\ref{Comparison of two series}, for every finite place we set
\begin{equation}\label{local comparison coefficient}
 \widetilde f_{\Phi_v,a}(g_v)
 =2W'_{a,v}(0,g_v,\Phi_v)
 -\log|a|_vW_{a,v}(0,g_v,\Phi_v)-f_{\Phi_v,a}(g_v).
\end{equation}

\begin{proposition}\label{ramified standard support comparison}
Suppose that $v$ is ramified and $|d_v|=1$, as assumed above. For $r=v(a)\ge0$,
\begin{equation}\label{ramified standard support difference}
 f_{\Phi_v^\vee,a}(1)-f_{\Phi_v,a}(1)
 =\begin{cases}
 (1-N_v^{-n-1})\deg_L(X)D_{0,v}(y_{v,r}),&n\text{ odd},\\
 \gamma(\BV_v,q)N_v^{-1/2}(1-N_v^{-n})\deg_L(X)D_{0,v}(B_r^\pm),&n\text{ even}.
 \end{cases}
\end{equation}
For odd $n$,
\begin{equation}\label{ramified standard odd coefficient}
\begin{aligned}
 \widetilde f_{\Phi_v,a}(1)
 ={}&\left(-2\frac{L_v'(n+1,\eta_v^{n+1})}{L_v(n+1,\eta_v^{n+1})}-\log N_v\right)W_{a,v}(0,1,\Phi_v)\\
 &+2N_v^{-n-1}\log N_v\,1_{\mathcal O_{F_v}}(a).
\end{aligned}
\end{equation}
For even $n$,
\begin{equation}\label{ramified standard even coefficient}
 \widetilde f_{\Phi_v,a}(1)
 =\gamma(\BV_v,q)N_v^{-n-1/2}\log N_v
 \begin{cases}
 (1-c_{+,v})+(c_{+,v}-3)N_v^{-nr},&+\text{ class},\\[1mm]
 \displaystyle\frac{N_v^n+1}{N_v^n-1}
 +\frac{N_v^n-3}{N_v^n-1}N_v^{-nr},&-\text{ class}.
 \end{cases}
\end{equation}
All these coefficients vanish for $v(a)<0$.
\end{proposition}

\begin{proof}
First, we identify the removed orbit and its contribution to \eqref{f series}. Since $(\Lambda_v^\vee)^\vee=\Lambda_v$, the two lattices have the same stabilizer. At norm $a$, Lemma~\ref{algebraic lemma ramified} gives primitive types $0,\ldots,r$ in the dual lattice. A vector of type $s<r$ is divisible by $\varpi_{E_v}^{r-s}$ in the dual hyperbolic summand, and hence lies in the standard lattice. Type $r$ lies outside it. For even $n$, scalar multiples of $e$ occur in both lattices. Thus only the type-$r$ hyperbolic orbit is removed. Its contribution, with the factor $|a|_v^n$ in \eqref{f series}, is exactly \eqref{ramified standard support difference}, by \eqref{ramified odd volume} and \eqref{ramified even orbit weights}.

We then use the equality of the standard and dual Whittaker derivatives to compute \eqref{local comparison coefficient}.For the spectral parameter $t$, \cite[Lemma~4.5 and Corollary~4.7]{Guo1} give
\begin{equation*}
   W_{a,v}(t,1,\Phi_v^\vee)-W_{a,v}(t,1,\Phi_v)=1-N_v^{-n-1}\qquad(n\text{ odd}),
\end{equation*}
and, for even $n$,
\begin{equation*}
\begin{aligned}
 W_{a,v}(t,1,\Phi_v^\vee)
 &=\gamma(\BV_v,q)N_v^{-1/2}(1\pm N_v^{-(r+1)(t+n)}),\\
 W_{a,v}(t,1,\Phi_v)
 &=\gamma(\BV_v,q)N_v^{-1/2}(N_v^{-n}\pm N_v^{-(r+1)(t+n)}).
\end{aligned}
\end{equation*}
The differences are constant in $t$, so the derivatives agree. Therefore \eqref{local comparison coefficient} gives
\begin{equation*}
\begin{aligned}
 \widetilde f_{\Phi_v,a}(1)={}&\widetilde f_{\Phi_v^\vee,a}(1)+f_{\Phi_v^\vee,a}(1)-f_{\Phi_v,a}(1)\\
 &+\log|a|_v\bigl(W_{a,v}(0,1,\Phi_v^\vee)-W_{a,v}(0,1,\Phi_v)\bigr).
\end{aligned}
\end{equation*}
For odd $n$, insert part~(1) of Proposition~\ref{Explicit f term ramified}; for even $n$, use \eqref{ramified even dual comparison}. Together with \eqref{ramified standard support difference} and the values in \eqref{ramified odd D0 formula} and \eqref{ramified even D0 formula}, these prove \eqref{ramified standard odd coefficient} and \eqref{ramified standard even coefficient}, since $\log|a|_v=-r\log N_v$. Here the local logarithmic derivative is $\zeta_v'(n+1)/\zeta_v(n+1)$ for odd $n$ and zero for even $n$.

Note that these computations include $r=0$: the odd coefficient is $2N_v^{-n-1}\log N_v$, and the even coefficients are $-2\gamma(\BV_v,q)N_v^{-n-1/2}\log N_v$ and $2\gamma(\BV_v,q)N_v^{-n-1/2}\log N_v$ in the $+$ and $-$ classes, respectively. The last assertion for $r<0$ follows from \cite[Lemma~4.5]{Guo1} and the absence of lattice vectors of norm $a$.
\end{proof}

\section{Comparison of two series}\label{Comparison of two series}

We combine the derivative series and the height series to prove Theorem~\ref{Modular height}. Subsection~\ref{ramified matching endpoint subsection} determines the constants that we need but did not compute completely in \eqref{reference height constant} and \eqref{ramified even D0 formula}. The specific techniques used in the proof come from Appendix~\ref{appendix auxiliary comparisons}. We then add the three lines of the difference series in \eqref{Difference series} and take its constant term, as in \cite[Section~5]{Yuan1}. The resulting relation gives the height formula once the additional primes in the induction have been treated in Appendix~\ref{appendix auxiliary heights}.

We keep \cite[Assumption~2.4]{Guo2}. In particular, $2$ splits in $E/F$, and $F_v/\QQ_p$ is unramified when $v$ is ramified in $E/F$. The conditions in \cite[(2.3.5)]{Guo2} make $\mathcal P-\hat\xi$ have degree zero on each connected component. 

As in the introduction, put
\begin{equation*}
 \mathcal D(g,\Phi)=\Pr I'(0,g,\Phi) +\left(\wh{\mathcal Z}_*(g,\Phi)-\frac{\deg_L(Z_*(g,\Phi))}{\deg_L(X)}\LL\right)\cdot(\mathcal P-\hat\xi).
\end{equation*}
The decomposition of the holomorphic projection gives
\begin{equation}\label{Difference series}
\begin{aligned}
 \mathcal D(g,\Phi)={}&\Pr'I'(0,g,\Phi)+\wh{\mathcal Z}_*(g,\Phi)\cdot\mathcal P\\
 &-\Pr'\mathcal J'(0,g,\Phi)-\wh{\mathcal Z}_*(g,\Phi)\cdot\hat\xi\\
 &-\frac{\deg_L(Z_*(g,\Phi))}{\deg_L(X)}\LL\cdot\mathcal P+\frac{\deg_L(Z_*(g,\Phi))}{\deg_L(X)}\LL\cdot\hat\xi.
\end{aligned}
\end{equation}
We call this the difference series. The arithmetic intersections retain the normalization over $\mathcal O_E$ used above.

\subsection{The lower-dimensional heights}\label{ramified matching endpoint subsection}

The calculations in Section~\ref{Explicit local terms} leave the constant $c_{+,v}$ and the normalized heights in \eqref{reference height constant} undetermined. We record these two explicit expressions needed below. Their proofs are in Appendix~\ref{appendix constant comparisons}. Roughly speaking, there we will use some local formulas already been proved in \cite{Guo1}, together with a comparison of Fourier coefficients.

\begin{proposition}\label{ramified matching c value}
For even $n$ and a ramified place $v$,
\begin{equation}\label{ramified matching c equals one}
   c_{+,v}=1.
\end{equation}
\end{proposition}

Substituting \eqref{ramified matching c equals one} into
\eqref{ramified even D0 formula} determines $D_{0,v}(B_s^+)$ for every
$s\ge0$. In the $+$ class, \eqref{ramified standard even coefficient} becomes
\begin{equation}\label{ramified matching ordinary correction}
 \widetilde f_{\Phi_v,a}(1)
 =-2\log N_v\,W_{a,v}(0,1,\Phi_v)
  +2\gamma(\BV_v,q)N_v^{-n-1/2}\log N_v\,1_{\mathcal O_{F_v}}(a).
\end{equation}

The second result combines the lower-dimensional height with the additional local multiples of the Whittaker functions. As in \eqref{reference height constant}, $\sigma$ records the choices of the lower-dimensional Hodge metrics in Definition~\ref{adelic lower Hodge line definition}. Locally, $\sigma_v$ records the parity of the inert type and, when $n$ is even, the ramified norm class.

\begin{proposition}\label{reference height comparison}
The expression
\begin{equation}\label{common reference height}
\begin{aligned}
 h_{n-1}={}&h_{\LL}(Z(y_0))+c_\sigma
 -\sum_{\substack{v\ \mathrm{inert}\\\sigma_v\ \mathrm{odd}}}
       \frac{N_v^n+(-1)^n}{N_v^n-(-1)^n}\log N_v\\
 &-\sum_{\substack{v\ \mathrm{ramified},\;2\mid n\\
                         \sigma_v\ \mathrm{opposite}}}
       \frac{N_v^n+1}{N_v^n-1}\log N_v
\end{aligned}
\end{equation}
is independent of $\sigma$.
\end{proposition}

If all the lower lattices and metrics are standard, the two sums in
\eqref{common reference height} are empty and $h_{n-1}$ is their modular
height. In general, Subsection~\ref{final formula section} determines this
number. In the next subsection, \eqref{common reference height} allows us
to combine all local types before taking the constant term.

\subsection{Induction formula from comparison}\label{Comparison subsection}

We return to the standard Schwartz function and add the three lines of
\eqref{Difference series}. Proposition~\ref{reference height comparison}
allows us to combine the lower-dimensional heights with the local
corrections. We then apply the key lemma of \cite{Guo1} to the complete
sum and take its constant term.

\subsubsection*{The first line of the difference series}

The comparison with the CM cycle in \cite[Theorem~1.3 and Section~5.4]{Guo2} gives, on the finite compact neighborhoods specified there,
\begin{equation}\label{comparison first line}
\begin{aligned}
    &\Pr'I'(0,g,\Phi)+\wh{\mathcal Z}_*(g,\Phi)\cdot\mathcal P\\
    ={}&-[F:\QQ]\left(\gamma+\log(4\pi)-2\sum_{i=1}^{n-1}\frac1i\right)E_*(0,g,\Phi)\\
    &-\sum_{v\nmid\infty\ \mathrm{nonsplit}} \avint_{E^1\backslash E^1(\BA)}\theta_*(g,\tau,\bar k_{\Phi_v}\otimes\Phi^v)\,d\tau+A(g,\Phi).
\end{aligned}
\end{equation}
Here $\bar k_{\Phi_v}$ is the function for proper intersections from \cite[Section~5.3]{Guo2}. The term $A(g,\Phi)$ is a finite sum of nonsingular pseudo-theta series on $W^\perp$, which is proved in \cite[Section~5.4]{Guo2}.

\subsubsection*{The second line of the difference series}

We use \cite[Theorem~3.4]{Guo1} for the analytic projection and \eqref{height series with local F} for the arithmetic intersection. They give
\begin{equation}\label{comparison second line decomposition}
\begin{aligned}
 &-\Pr'\mathcal J'(0,g,\Phi)-\wh{\mathcal Z}_*(g,\Phi)\cdot\hat\xi\\
 ={}&(c_0-2c_3[F:\QQ])E_*(0,g,\Phi)+C_*(0,g,\Phi)\\
 &-\sum_{v\nmid\infty}\bigl(2E'(0,g,\Phi)(v)-\mathcal F_\Phi^{(v)}(g)\bigr)\\
 &-\sum_{\substack{y\in U\backslash V_f\\q(y)\in F_+}}r(g)\Phi(y)\frac{\deg_LZ(y)}{\deg_LX}
\bigl(h_{\LL}(Z(y_0))+c_\sigma\bigr).
\end{aligned}
\end{equation}

We focus in particular on the term involving $\widetilde f_{\Phi_v,a}$ in the middle of the sum, which is defined in \eqref{local comparison coefficient}. At an unramified place, their part common to all types is
\begin{equation*}
 2\left(-\frac{L'_v(n+1,\eta_v^{n+1})}{L_v(n+1,\eta_v^{n+1})}
                  +\log|d_v|\right)W_{a,v}(0,1,\Phi_v).
\end{equation*}
The additional function contributed by the different $d_v$ is
\begin{equation*}
 2\gamma(\BV_v,q)|d_v|^{n+3/2}\log N_v
 \sum_{j=0}^{v(d_v)-1}N_v^j
             1_{p_v^{j-v(d_v)}\mathcal O_{F_v}}(a),\qquad a\ne0.
\end{equation*}
This follows by differentiating the terms with $j<v(d_v)$ in \cite[Lemma~4.5(1)]{Guo1}, and this sum is empty when $v(d_v)=0$. For inert odd types we also use \eqref{inert f equals S}. At ramified places we use Proposition~\ref{ramified standard support comparison}, with $c_{+,v}=1$ from Proposition~\ref{ramified matching c value}.

We now combine the terms in the middle sum of \eqref{comparison second line decomposition} that depend on the local type with the height term in its last line. By \eqref{common reference height}, their sum is expressed in terms of $h_{n-1}$. The ramified multiples left in the second line are $\log N_v\,E_*$ for odd $n$ and $2\log N_v\,E_*$ for even $n$. In particular, we use \eqref{ramified matching ordinary correction} only in the $+$ class, since the difference in the $-$ class is already included in \eqref{common reference height}.

The remaining local functions are Schwartz. By \cite[Lemma~2.5]{Guo1}, each is a sum of Whittaker functions associated with $\Phi_v^+$ and $\Phi_v^-$ such that
\begin{equation}\label{compact correction zero identity}
   \Phi_v^+(0)+\Phi_v^-(0)=0.
\end{equation}
They vanish unless $v$ divides the different or is ramified in $E/F$. Their Eisenstein coefficients occur with a minus sign in \eqref{comparison second line decomposition}.

\subsubsection*{The third line of the difference series}

Geometric Siegel--Weil and $\LL\cdot\hat\xi=h_{\LL}(X)$ give
\begin{equation}\label{comparison third line}
  -\frac{\deg_LZ_*}{\deg_LX}\LL\cdot\mathcal P+\frac{\deg_LZ_*}{\deg_LX}\LL\cdot\hat\xi=E_*\bigl(\LL\cdot\mathcal P-h_{\LL}(X)\bigr).
\end{equation}

\subsubsection*{The difference series and its constant term}

We add \eqref{comparison first line}, \eqref{comparison second line decomposition} and \eqref{comparison third line}, with the local terms combined as above. The full cuspidal form $\mathcal D(g,\Phi)$ is now a finite sum of nonsingular pseudo-series. The key lemma \cite[Lemma~2.4]{Guo1} replaces the terms on the full Hermitian space by their theta and Eisenstein series; the terms $A(g,\Phi)$ on $W^\perp$ contribute zero to this identity. Write
\begin{equation}\label{known completed terms}
\begin{aligned}
 K(g,\Phi)={}&-\sum_{v\nmid\infty\ \mathrm{nonsplit}}\avint_{E^1\backslash E^1(\BA)}\theta(g,\tau,\bar k_{\Phi_v}\otimes\Phi^v)\,d\tau\\
 &+\sum_{v\nmid\infty}C(0,g,\Phi)(v)\\
 &-\sum_{v\nmid\infty}\bigl(E(0,g,\Phi_v^+\otimes\Phi^v)+E(0,g,\Phi_v^-\otimes\Phi^v)\bigr)
\end{aligned}
\end{equation}
for these known terms. Recall the value
\begin{equation*}
 c_3=\frac12\sum_{i=1}^{n-1}\frac1i-\frac1{2n}-\frac12\log4
\end{equation*}
from \cite[Lemma~4.4]{Guo1}. Also we have $\sum_{v\ \mathrm{ramified}}\log N_v=\log d_{E/F}$. Thus, we obtain
\begin{equation}\label{Explicit difference series}
\begin{aligned}
 \mathcal D(g,\Phi)={}&K(g,\Phi)+\Bigg([F:\QQ]\left(\sum_{i=1}^{n}\frac1i-\gamma-\log\pi\right)+c_0\\
 &\quad+2\sum_{v\nmid\infty}\left(\frac{L'_v(n+1,\eta_v^{n+1})}{L_v(n+1,\eta_v^{n+1})}-\log|d_v|\right)\\
 &\quad+h_{n-1}-h_{\LL}(X)+\LL\cdot\mathcal P+\begin{cases}2\log d_{E/F},&2\mid n,\\
                \log d_{E/F},&2\nmid n\end{cases}\Bigg)E(0,g,\Phi).
\end{aligned}
\end{equation}

For convenience, we make the following notation. Recall that $N\subset\UU$ is the standard unipotent subgroup. For an automorphic form $\Psi$ on $\UU(\BA)$, we define its constant term along $N$ by
\begin{equation*}
 \operatorname{CT}(\Psi)(g)=\int_{N(F)\backslash N(\BA)}\Psi(ng)\,dn=\int_{F\backslash\BA}\Psi(n(b)g)\,db,
\end{equation*}
where $db$ is the quotient measure induced by the additive Haar measure fixed above. When $g=1$, we write
\begin{equation*}
   \operatorname{CT}(\Psi)=\operatorname{CT}(\Psi)(1).
\end{equation*}

Now we take the constant term of \eqref{Explicit difference series} at $g=1$. The holomorphic section at infinity has zero intertwining integral, as in \eqref{holomorphic intertwining vanishing}. The finite product converges for $n\ge2$, so $\Phi(0)=1$ gives $\operatorname{CT}(E)=1$. The two Eisenstein series for each pair $\Phi_v^+,\Phi_v^-$ have total constant term zero by \eqref{compact correction zero identity}.

It remains to evaluate the functions in the theta and $C$-series at zero. By \cite[Section~5.3]{Guo2} and \cite[Lemma~4.8]{Guo1},
\begin{equation*}
 -\bar k_{\Phi_v}(1,0)+c_{\Phi_v}(1,0)=\begin{cases}\log|d_v|,&v\text{ unramified},\\
                -\log N_v,&v\text{ ramified}.
   \end{cases}
\end{equation*}
Here $\bar k_{\Phi_v}=0$ at split places. At a ramified place the two values are $N_v^{-1}\log N_v$ and $-(1-N_v^{-1})\log N_v$. Thus
\begin{equation}\label{Equation to 0}
  \operatorname{CT}(K)=-\log|d_F|-\log d_{E/F}.
\end{equation}

Finally, \cite[Theorem~1.1]{Guo2} gives
\begin{equation*}
  \LL\cdot\mathcal P=-2\frac{L_f'(0,\eta)}{L_f(0,\eta)}.
\end{equation*}
According to \cite[(3.1.8)]{Guo1}, we have
\begin{equation*}
   c_0+\LL\cdot\mathcal P=\log|d_F|+\log d_{E/F}-[F:\QQ](\gamma+\log(4\pi)).
\end{equation*}
Substituting these values into the constant term of \eqref{Explicit difference series}, and using $\operatorname{CT}(\mathcal D)=0$ and $\sum_{v\nmid\infty}\log|d_v|=-\log|d_F|$, we have
\begin{equation}\label{induction formula}
\begin{aligned}
 h_{\LL}(X)={}&h_{n-1}+2\frac{L_f'(n+1,\eta^{n+1})}{L_f(n+1,\eta^{n+1})}+2\log|d_F|\\
 &+[F:\QQ]\left(\sum_{i=1}^{n}\frac1i-2\gamma-2\log(2\pi)\right)\\
 &+\begin{cases}2\log d_{E/F},&2\mid n,\\
                 \log d_{E/F},&2\nmid n.
   \end{cases}
\end{aligned}
\end{equation}
The next subsection determines $h_{n-1}$ and proves the height formula.

\subsection{The final formula}\label{final formula section}

We now identify $h_{n-1}$ and finish the proof of Theorem~\ref{Modular height}. Let $\mathfrak h_m$ denote the explicit expression in that theorem with $n=m$. The numerical relation \eqref{numerical height recurrence} in Appendix~\ref{appendix auxiliary heights} is obtained by subtracting its values in consecutive dimensions. The curve formula \cite[Theorem~1.2]{Guo2}, together with the functional equation of $L(s,\eta)$, gives the initial value $\mathfrak h_1$.

A lower-dimensional variety used in the induction may have nonstandard lattices at additional primes. Lemma~\ref{auxiliary height lemma} in Appendix~\ref{appendix auxiliary heights} treats this issue. For an $m$-dimensional model that is standard outside a finite set $S$ of rational primes not dividing $2d_Fd_{E/F}$, it proves
\begin{equation*}
 h_{\LL}(X)-\mathfrak h_m\in\sum_{p\in S}\overline{\QQ}\log p.
\end{equation*}
For the integral model of Theorem~\ref{Modular height}, take $S=\varnothing$. Lemma~\ref{auxiliary height lemma} gives $h_{\LL}(X)=\mathfrak h_n$. Comparing \eqref{induction formula} with \eqref{numerical height recurrence} now yields
\begin{equation*}
   h_{n-1}=\mathfrak h_{n-1}.
\end{equation*}
Consequently \eqref{induction formula} is the relation for the explicit heights in Theorem~\ref{Modular height}, with the curve formula as its initial value. This completes the proof.

\clearpage
\appendix
\section{Auxiliary comparisons}\label{appendix auxiliary comparisons}

This appendix contains the arguments used to determine the constants in Subsection~\ref{ramified matching endpoint subsection} and to justify the induction in Subsection~\ref{final formula section}. Since this material is not closely related to the main line of our proof and has independent mathematical interest, we collect it in this separate appendix. The arrangement is as follows. We first compare Fourier coefficients of holomorphic forms on $\UU$. Then we apply this comparison to the local formulas of Section~\ref{Explicit local terms}. The last subsection treats the additional primes that may occur in a lower-dimensional variety.

\subsection{Fourier coefficients and constant terms}

We begin with the choice of a Schwartz function at one split place. This choice makes the terms indexed by $W^\perp$ vanish without removing the Eisenstein coefficients used to compare the heights.

\begin{lemma}\label{split test for reference heights}
Let $w$ be a split finite place with $|d_w|=1$, where the standard lattice has the orthogonal decomposition $\Lambda_w=\Lambda_{W_w^\perp}\perp\Lambda_{W_w}$. Choose $e\in\Lambda_{W_w}$ with unit norm and an integer $k\ge1$, and put
\begin{equation*}
 \Phi_w=1_{p_w^k\Lambda_{W_w^\perp}}\otimes1_{e+p_w^k\Lambda_{W_w}}.
\end{equation*}
Then the following assertions hold at $g_w=1$.
\begin{enumerate}
\item The contributions indexed by $y\in W^\perp$ in the first line of \eqref{Difference series} vanish, and $c_{\Phi_w}(1,y)=0$ for $y\in W_w^\perp$. The function $D_{0,w}$ is zero on the support of $\Phi_w$. Thus $f_{\Phi_w,a}(1)=0$ and $\log|a|_wW_{a,w}(0,1,\Phi_w)=0$ for $a\ne0$, so \eqref{local comparison coefficient} reduces to
\begin{equation*}
 \widetilde f_{\Phi_w,a}(1)=2W'_{a,w}(0,1,\Phi_w).
\end{equation*}
\item The function $a\mapsto 2W'_{a,w}(0,1,\Phi_w)$ extends to a Schwartz function on $F_w$.
\end{enumerate}
\end{lemma}
\begin{proof}
Since $0\notin e+p_w^k\Lambda_{W_w}$, the restriction of $\Phi_w$ to $W_w^\perp$ is zero. This removes the terms weighted by that restriction. We compute the norm on its support to prove the remaining assertions.

Choose coordinates on $W_w$ for which $\Lambda_{W_w}=\mathcal O_{F_w}^2$ and $q(u,v)=uv$. Both coordinates of $e=(e_1,e_2)$ are units. The map $(u,v)\mapsto(u,uv)$ sends $e+p_w^k\Lambda_{W_w}$ onto
\begin{equation*}
 (e_1+p_w^k\mathcal O_{F_w})\times(q(e)+p_w^k\mathcal O_{F_w})
\end{equation*}
and has Jacobian determinant of absolute value one. The measure obtained by applying $q$ to $1_{e+p_w^k\Lambda_{W_w}}$ is therefore a constant multiple of $1_{q(e)+p_w^k\mathcal O_{F_w}}$. For $x\in p_w^k\Lambda_{W_w^\perp}$, the value $q(x)$ belongs to $p_w^{2k}\mathcal O_{F_w}$. Since $q(x+(u,v))=q(x)+uv$, integration over $x$ changes only the constant. Hence the measure obtained by applying $q$ to $\Phi_w$ has the same density, up to a nonzero constant.

The Weil action in \cite[Section~2.1]{Guo1} identifies $b\mapsto r(wn(b))\Phi_w(0)$ with the Fourier transform of this density, up to the fixed Weil factor. This function is Schwartz. The definition of the Whittaker integral in \cite[Section~2.2]{Guo1} gives
\begin{equation*}
    W'_{a,w}(0,1,\Phi_w)=\int_{F_w}\log\delta_w(wn(b))\,
    r(wn(b))\Phi_w(0)\psi_w(-ab)\,db.
\end{equation*}
The function $b\mapsto\log\delta_w(wn(b))$ is locally constant, so the integrand before multiplication by $\psi_w(-ab)$ is Schwartz in $b$. Differentiation under the integral is valid on its fixed compact support. Fourier transformation proves part~(2).

For $c_{\Phi_w}$, the Whittaker integral in \cite[(3.1.5)]{Guo1} uses only the factor $1_{e+p_w^k\Lambda_{W_w}}$ on $W_w$. Its norm density is supported on units. Fourier inversion therefore gives
\begin{equation*}
   W_{0,w}(0,1,1_{e+p_w^k\Lambda_{W_w}})=0.
\end{equation*}
This integral is holomorphic in the complex parameter, since its integrand has compact support in $b$. The normalization of $W_{0,w}^{\circ}(t,1,1_{e+p_w^k\Lambda_{W_w}})$ multiplies it by $(1-N_w^{-t})/(1-N_w^{-t-1})$, which also vanishes at $t=0$. Its derivative at zero is consequently zero. The other term in \cite[(3.1.5)]{Guo1} vanishes because $\log\delta_w(1)=0$. Thus $c_{\Phi_w}(1,y)=0$. At places other than $w$, the terms indexed by $W^\perp$ have the factor $\Phi_w|_{W_w^\perp}=0$.

All vectors in the support of $\Phi_w$ have unit norm. Proposition~\ref{split local D0 formula} gives $D_{0,w}=0$ there, and \eqref{f series} gives $f_{\Phi_w,a}(1)=0$. Fourier inversion for the full norm density also shows that $W_{a,w}(0,1,\Phi_w)$ is supported on units. This proves $\log|a|_wW_{a,w}(0,1,\Phi_w)=0$ and completes part~(1).
\end{proof}

We next compare holomorphic forms on $\UU(F)\backslash\UU(\BA_F)$. At a finite place unramified in $E/F$, a local section is called \emph{spherical} if it is right invariant under the standard maximal compact subgroup of $\UU(F_v)$. Its Whittaker function is called spherical as well. The compact subgroup here is in $\UU(F_v)$, not in the group $\U(V_v)$ defining the Shimura variety.

\begin{lemma}\label{Whittaker germ detection}
Consider the finite-dimensional space of holomorphic automorphic forms on $\UU(F)\backslash\UU(\BA_F)$ with the archimedean weight of $E(0,g,\Phi)$, where $n+1\ge3$, a fixed central character, and right invariance under a fixed compact open subgroup of $\UU(\BA_{F,f})$. Let $S$ contain the finite places where $E/F$ is ramified, where $\psi_v$ does not have conductor $\mathcal O_{F_v}$, or where the fixed compact open subgroup is not the standard maximal compact subgroup, or the section defining $E(0,g,\Phi)$ is not invariant under that subgroup.

Choose finitely many classes of totally positive elements of $F$ modulo $\Nm(E^\times)$. For each of these classes, we choose Fourier indices $a$ so that, for every $v\in S$, the valuation $v(a)$ can be chosen arbitrarily large subject to a fixed congruence condition, independently of the choices at the other places in $S$. The residues of $p_v^{-v(a)}a$ modulo fixed sufficiently high powers of $p_v$ are fixed. We write $R_a(1,g^S)$ for the Fourier coefficient of a form $R$ evaluated at $g_v=1$ for $v\in S$ and at $g^S$ outside $S$. In the identities below, we omit the common exponential factor at infinity from both the Fourier coefficients and the Whittaker products.
\begin{enumerate}
\item Suppose that for every chosen $a$, the function $g^S\mapsto R_a(1,g^S)$ is a scalar multiple of $W_a^S(0,g^S,\Phi^S)$ for all $g^S\in\UU(\BA_{F,f}^S)$, with the scalar independent of $g^S$. An Eisenstein series in the induction defining $E(0,g,\Phi)$ has these same coefficients. Consequently there is one polynomial $P$, of degree at most one in each variable, such that
\begin{equation*}
 R_a(1,g^S)=P\bigl((\eta_v(a)^{n+1}|a|_v^n)_{v\in S}\bigr)W_a^S(0,g^S,\Phi^S).
\end{equation*}
The polynomial is the same for all the chosen norm classes.
\item The functional taking the constant term at the standard cusp is a $\overline{\QQ}$-linear combination of finitely many functionals $R\mapsto R_a(1,g^S)$. The index $a$ may be restricted to any one of the chosen norm classes, and all right translates outside $S$ are allowed.
\end{enumerate}
In part~(1), one may also fix $g_w=1$ at a split place $w\notin S$ where the level need not be spherical, and allow right translates only outside $S\cup\{w\}$. In that case $a_w$ lies in a ball of units on which the Whittaker functions of all sections fixed by the level at $w$ are constant. The ball is chosen before the Fourier indices.
\end{lemma}
\begin{proof}
The proof has two inputs: the local Whittaker coefficient is affine in $\eta_v(a)^{n+1}|a|_v^n$ near zero, and Hecke operators separate the Eisenstein subspaces from the cuspidal subspace. We prove these facts first and use them for both assertions.

For a section $h_v$ in the local induction defining $E(0,g,\Phi)$, put
\begin{equation*}
    W_{a,v}(h_v)=\int_{F_v}h_v(wn(b))\psi_v(-ab)\,db,
\end{equation*}
with the measure of \cite[Section~2.2]{Guo1}. The transformation law and local constancy of $h_v$ give
\begin{equation*}
    h_v(wn(b))=\eta_v(-b^{-1})^{n+1}|b|_v^{-n-1}h_v(1)
\end{equation*}
for sufficiently large $|b|_v$. For small nonzero $|a|_v$, the character $\psi_v(-ab)-1$ vanishes on the remaining compact set. Subtracting the coefficient at zero and substituting $t=-ab$ therefore gives
\begin{equation}\label{local Whittaker germ}
   W_{a,v}(h_v)=W_{0,v}(h_v)+\beta_vh_v(1)\eta_v(a)^{n+1}|a|_v^n,\qquad 0<|a|_v\ll1,
\end{equation}
where
\begin{equation*}
  \beta_v=\int_{F_v}\eta_v(t)^{-n-1}|t|_v^{-n-1}(\psi_v(t)-1)\,dt\ne0.
\end{equation*}
The integral vanishes near zero and converges at infinity. When $\psi_v$ has conductor $\mathcal O_{F_v}$ and $\eta_v^{n+1}$ is unramified, summation by valuation gives
\begin{equation*}
  \beta_v=\frac{1-\eta_v(p_v)^{n+1}N_v^{-n-1}}{1-\eta_v(p_v)^{-n-1}N_v^n}.
\end{equation*}
For a ramified character the nonzero contribution is a primitive Gauss sum, multiplied by a nonzero scalar. Changing the additive conductor also multiplies $\beta_v$ by a nonzero scalar. Replacing $\eta_v^{n+1}$ by the inducing character of another Eisenstein component gives \eqref{local Whittaker germ} for that component, again with a nonzero coefficient $\beta_v$. A basis of each finite-dimensional space of local sections gives a common range for \eqref{local Whittaker germ}.

The holomorphic space is the Hecke-stable sum of its cuspidal subspace and the Eisenstein subspaces indexed by their inducing characters. For the Hilbert congruence components of $\mathrm{SU}(1,1)\simeq\mathrm{SL}_2$, this is the decomposition in \cite[Section~4, (27)--(28)]{DK}; fixing the central character gives the stated decomposition on $\UU$. The common kernel of the Fourier functionals in the lemma is stable under Hecke operators outside $S$. Indeed, each such operator is a finite sum of right translations, and all those translations are included among the functionals.

At a sufficiently large split place $v\notin S$, take the Hecke operator $T_v$ on $\GL_2(F_v)$ given by $N_v^{-1/2}$ times the double coset of $\operatorname{diag}(p_v,1)$, with $\GL_2(\mathcal O_{F_v})$ of volume one. Its norm on the cuspidal subspace is at most $N_v^{1/2}+N_v^{-1/2}$, since the double coset has $N_v+1$ unitary right translations. Every Eisenstein eigenvalue at this weight has absolute value at least $N_v^{n/2}-N_v^{-n/2}$, which is larger. Thus polynomials in the Hecke operators project onto the Eisenstein subspaces and preserve the common kernel. The operators outside $S$ also distinguish the finitely many inducing characters. The distinct absolute values of the two inducing parameters determine their order; a quotient character trivial outside $S$, at infinity, and on $E^\times$ is trivial at $S$ by weak approximation.

In part~(1), the class of $R$ modulo the common kernel has the Hecke eigenvalues of $E(0,g,\Phi)$. Projection onto that Eisenstein subspace therefore gives an Eisenstein series with exactly the specified coefficients of $R$. Express its section as a finite sum of tensor products. Taking out the spherical factors outside $S$ and applying \eqref{local Whittaker germ} at every $v\in S$ gives the polynomial $P$. The projection used the functionals from all the chosen classes at once, so one section, and hence one polynomial, works for all of them. In the variant with $w$, the Hecke operators are taken outside $S\cup\{w\}$; the fixed Whittaker values on the chosen ball at $w$ enter the coefficients of $P$.

For part~(2), use the Hecke projections just constructed to treat one Eisenstein subspace at a time. Formula~\eqref{local Whittaker germ}, with its inducing character, writes the coefficients as a polynomial of degree at most one in each local variable. Choose two distinct values of each variable. Taking successive difference quotients in the variables isolates the coefficient of their product. This coefficient is the value at the identity of the global section multiplied by $\prod_{v\in S}\beta_v$. The factors outside $S$ can be chosen nonzero, since all their right translates are allowed. Therefore, we have expressed the value of the section at the identity using finitely many of its Fourier coefficients.

That value is the constant term of the Eisenstein series: the other contribution contains the holomorphic archimedean integral
\begin{equation}\label{holomorphic intertwining vanishing}
 \int_{\RR}(1-ib)^{-n-1}\,db=0,
\end{equation}
and the finite intertwining product converges for $n\ge2$. The cuspidal subspace has constant term zero. Finally, the Fourier coefficients, constant terms, and finite-level maps are defined over $\overline{\QQ}$, as in \cite[Proposition~4.21, Theorems~6.8--6.9 and~7.8]{Sil} on the Hilbert congruence components. The Hecke projections and finite differences may therefore be taken over $\overline{\QQ}$. This proves part~(2).
\end{proof}

\subsection{The two height comparisons}\label{appendix constant comparisons}

We prove Propositions~\ref{ramified matching c value} and~\ref{reference height comparison} by applying Lemma~\ref{Whittaker germ detection}(1) to one holomorphic form. We first define this form using the three lines of \eqref{Difference series}, keeping $c_{+,v}$ and the lower-dimensional heights as unknowns.

We fix finitely many choices of $\sigma$ to compare. We choose a split place $w$ where the chosen lower metrics are standard and $|d_w|=1$, and use $\Phi_w$ from Lemma~\ref{split test for reference heights}. At every other finite place we keep the standard Schwartz function. For each $v\ne w$, we choose Schwartz functions $\Phi_v^+$ and $\Phi_v^-$ as in \cite[Lemma~2.5]{Guo1} so that, for $a\ne0$,
\begin{equation*}
 \begin{aligned}
 &W_{a,v}(0,1,\Phi_v^+)+W_{a,v}(0,1,\Phi_v^-)\\
 ={}&\begin{cases}
 \displaystyle 2\gamma(\BV_v,q)|d_v|^{n+3/2}\log N_v
 \sum_{j=0}^{v(d_v)-1}N_v^j
 1_{p_v^{j-v(d_v)}\mathcal O_{F_v}}(a),&v\text{ unramified},\\[2mm]
 2N_v^{-n-1}\log N_v\,1_{\mathcal O_{F_v}}(a),
       &v\text{ ramified},\ n\text{ odd},\\[1mm]
 2\gamma(\BV_v,q)N_v^{-n-1/2}\log N_v\,1_{\mathcal O_{F_v}}(a),
       &v\text{ ramified},\ n\text{ even}.
 \end{cases}
 \end{aligned}
\end{equation*}
The first function comes from the terms with $j<v(d_v)$ in \cite[Lemma~4.5(1)]{Guo1}; the last two occur in \eqref{ramified standard odd coefficient} and \eqref{ramified standard even coefficient}. At $w$, we choose $\Phi_w^+$ and $\Phi_w^-$ to represent $2W'_{a,w}(0,1,\Phi_w)$, using Lemma~\ref{split test for reference heights} and \cite[Lemma~2.5]{Guo1}. These choices satisfy $\Phi_v^+(0)+\Phi_v^-(0)=0$, and we take both functions to be zero whenever the function they represent is zero.

With the theta kernels of \cite[Section~5.3]{Guo2} and the series $C$ of \cite[(3.1.6)]{Guo1}, define
\begin{equation}\label{height comparison holomorphic form}
\begin{aligned}
 R(g)={}&\mathcal D(g,\Phi)
 +\sum_{v\nmid\infty\ \mathrm{nonsplit}}
 \avint_{E^1\backslash E^1(\BA)}
       \theta(g,\tau,\bar k_{\Phi_v}\otimes\Phi^v)\,d\tau\\
 &-C(0,g,\Phi)
 +\sum_{v\nmid\infty}\bigl(E(0,g,\Phi_v^+\otimes\Phi^v)
                         +E(0,g,\Phi_v^-\otimes\Phi^v)\bigr).
\end{aligned}
\end{equation}
Every series in \eqref{height comparison holomorphic form} is holomorphic, and only finitely many local corrections are nonzero. We work in each central-character component. We choose a common compact open level for these forms. At $w$, we also require this level to fix $\Phi_w$ and $\Phi_w^\pm$. We then choose the ball of units allowed in Lemma~\ref{Whittaker germ detection}(1) inside $q(e)+p_w^k\mathcal O_{F_w}$.

For each $\sigma$, choose an initial positive index $a_\sigma$ in its norm class. Since $w$ is split, weak approximation allows us to multiply $a_\sigma$ by a global norm so that its $w$-component lies in the chosen ball. Let $S$, disjoint from $w$, contain the remaining nonspherical places, the ramified places, the places dividing the different, and all finite places where one of the $a_\sigma$ is not a unit. For each $v\in S$, choose a prime of $E$ above $v$. Finiteness of a ray class group gives an element $\alpha_v\in E^\times$ generating a positive power of this prime and satisfying sufficiently strong unit congruences at the other prescribed places, including $w$. Taking further powers also keeps the unit parts at $S$ in the required residue classes. In particular,
\begin{equation}\label{ramified depth principal ideal}
 v\bigl(\Nm(\alpha_v)\bigr)>0,\qquad
 v'\bigl(\Nm(\alpha_v)\bigr)=0\quad(v'\ne v).
\end{equation}
The indices
\begin{equation*}
 a=a_\sigma\prod_{v\in S}\Nm(\alpha_v)^{m_v},\qquad m_v\in\ZZ_{\ge0},
\end{equation*}
therefore have independent valuations at $S$ and stay in the chosen ball at $w$. We take the integers $m_v$ sufficiently large. At an inert place, $m_v$ varies in steps of $2$.

Multiplication by $\Nm(\alpha)$ is realized by replacing a vector $y$ by $\alpha y$, which has the same orthogonal complement $V_y^\perp$. We keep the lower metrics of Definition~\ref{adelic lower Hodge line definition} under this identification. For standard lower metrics we use \eqref{adelic lower projection formula}; for the other metrics we use the pullback constructions in Lemmas~\ref{inert odd adjacent correspondence} and~\ref{ramified even adjacent correspondence}. These give the same lower bundle after passage to a common finite cover. The projection formula multiplies its self-intersection and its geometric degree by the same covering degree. Thus the normalized height in \eqref{reference height constant} is fixed within each family. The geometric Siegel--Weil formula \eqref{geometric Siegel--Weil formula} then gives
\begin{equation*}
 \sum_{\substack{y\in U\backslash V_{f,a}}}
 r(g)\Phi(y)\frac{\deg_LZ(y)}{\deg_LX}
       \bigl(h_{\LL}(Z(y_0))+c_\sigma\bigr)
 =-\bigl(h_{\LL}(Z(y_0))+c_\sigma\bigr)E_a(0,g,\Phi).
\end{equation*}
The sum has the orbit weights of \eqref{height series with local F}.

We must also check the right translates required in Lemma~\ref{Whittaker germ detection}(1). For $v\notin S\cup\{w\}$, scalar multiplication preserves the vector stabilizer. It also preserves $r-s-t$ at a split place and the primitive type at an inert place. The formulas in Propositions~\ref{split local D0 formula} and~\ref{inert local D0 formula} therefore show that it preserves $D_{0,v}$. Substitution in \eqref{f series}, using the Weil action of \cite[Section~2.1]{Guo1}, gives
\begin{equation}\label{local comparison covariance}
 f_{\Phi_v,a}(m(t))
 =\chi_{\BV,v}(t)|t|_{E_v}^{(1-n)/2}
                         f_{\Phi_v,\Nm(t)a}(1),
 \qquad t\in E_v^\times.
\end{equation}
The transformation formula for $W_{a,v}$ in \cite[Section~2.2]{Guo1} gives the identical factor for $2W'_{a,v}(0,g_v,\Phi_v)-\log|a|_vW_{a,v}(0,g_v,\Phi_v)$. Thus \eqref{split f equals S} and \eqref{inert f equals S} for even valuation hold on the diagonal subgroup. If the transformed local index has negative valuation, both sides are zero by \cite[Lemma~4.5]{Guo1} and the support in \eqref{f series}. The unipotent action multiplies both sides by $\psi_v(ab)$, and both sides are invariant under the standard maximal compact subgroup. Iwasawa decomposition proves the comparison for every $g_v\in\UU(F_v)$. The proper-intersection formula of \cite[Section~5.4]{Guo2} has this compatibility with the Weil action as well.

We can now compute the specified coefficients of $R$. The first-line comparison \eqref{comparison first line} and Lemma~\ref{split test for reference heights} remove the terms indexed by $W^\perp$. Equations~\eqref{comparison second line decomposition} and \eqref{comparison third line}, which follow from \cite[Theorem~3.4]{Guo1}, \eqref{height series with local F} and geometric Siegel--Weil, give the other two lines. In the second line the coefficient at $v$ is $-\widetilde f_{\Phi_v,a}$, by \eqref{local comparison coefficient} and \eqref{F series}. The Eisenstein series added in \eqref{height comparison holomorphic form} cancel the contributions of the Schwartz functions represented by $\Phi_v^+$ and $\Phi_v^-$. At $w$, the whole local coefficient is removed, since Lemma~\ref{split test for reference heights} gives $f_{\Phi_w,a}(1)=0$ and $\log|a|_wW_{a,w}(0,1,\Phi_w)=0$.

For a fixed chosen index $a$, each remaining term contains $W_a^{S\cup\{w\}}(0,g^{S\cup\{w\}},\Phi^{S\cup\{w\}})$. The coefficients multiplying this product do not depend on $g^{S\cup\{w\}}$: the comparison at each place outside $S\cup\{w\}$ is a fixed multiple of $W_{a,v}(0,g_v,\Phi_v)$, and the sum of these multiples converges because $n+1>1$. All remaining local terms are at $S$, where $g_v=1$. Lemma~\ref{Whittaker germ detection}(1) therefore gives
\begin{equation}\label{common reference ordinary factor}
 R_a(1,g^{S\cup\{w\}})
 =P\bigl((\eta_v(a)^{n+1}|a|_v^n)_{v\in S}\bigr)
       W_a^{S\cup\{w\}}(0,g^{S\cup\{w\}},\Phi^{S\cup\{w\}}).
\end{equation}
The construction of $R$ in \eqref{height comparison holomorphic form} is independent of $\sigma$. Hence one polynomial $P$ applies to all the chosen families. We now use this polynomial to prove the two comparisons stated in Subsection~\ref{ramified matching endpoint subsection}.

\begin{proof}[Proof of Proposition~\ref{ramified matching c value}]
We use \eqref{common reference ordinary factor} for families in both norm classes at $v$. Put
\begin{equation*}
 z_v=\begin{cases}
 N_v^{-nv(a)},&a\text{ in the }+\text{ class},\\
 -N_v^{-nv(a)},&a\text{ in the }-\text{ class}.
 \end{cases}
\end{equation*}
By Corollary~\ref{algebraic corollary ramified}, this variable is $\eta_v(\epsilon(\Lambda_v))\eta_v(a)|a|_v^n$. The fixed sign only changes the variable of $P$. The standard Whittaker formula in the proof of Proposition~\ref{ramified standard support comparison} gives
\begin{equation*}
 W_{a,v}(0,1,\Phi_v)
 =\gamma(\BV_v,q)N_v^{-n-1/2}(1+z_v).
\end{equation*}

In \eqref{height comparison holomorphic form}, the local contribution from \eqref{ramified standard even coefficient} is $-\widetilde f_{\Phi_v,a}(1)+2\gamma(\BV_v,q)N_v^{-n-1/2}\log N_v$. After division by $\gamma(\BV_v,q)N_v^{-n-1/2}\log N_v$, it is
\begin{equation*}
 \begin{cases}
 (c_{+,v}+1)+(3-c_{+,v})z_v,&+\text{ class},\\[1mm]
 \displaystyle\frac{N_v^n-3}{N_v^n-1}(1+z_v),&-\text{ class}.
 \end{cases}
\end{equation*}
Every contribution from another place, and the entire height term, contains $W_{a,v}(0,1,\Phi_v)$ and hence $1+z_v$.

On the family in the $-$ class, all contributions are therefore divisible by $1+z_v$. The valuations vary independently through infinite sets, so polynomial uniqueness proves that $P$ itself is divisible by $1+z_v$. We now use the family in the $+$ class, whose coefficients are given by this same $P$. For each fixed choice of the other valuations, both expressions are polynomials in $z_v$ and agree for infinitely many values. Their equality thus holds at $z_v=-1$. The contribution from $v$ becomes $2(c_{+,v}-1)$ times a nonzero scalar and the other Whittaker factors; all other contributions are zero. We choose the other valuations sufficiently large that these factors are nonzero. It follows that $c_{+,v}=1$.
\end{proof}

\begin{proof}[Proof of Proposition~\ref{reference height comparison}]
We compare two choices of $\sigma$ by including both in the construction of \eqref{common reference ordinary factor}. At an inert place, \eqref{inert f equals S} has an additional multiple of $W_{a,v}$ when the valuation is odd. Its negative sign in the second line of \eqref{Difference series} gives the first subtraction in \eqref{common reference height}.

At a ramified place with even $n$, remove the compact function from \eqref{ramified standard even coefficient} and use Proposition~\ref{ramified matching c value}. The coefficients of $W_{a,v}$ in $\widetilde f_{\Phi_v,a}$ are $-2\log N_v$ in the $+$ class and $-(N_v^n-3)\log N_v/(N_v^n-1)$ in the $-$ class. Their difference is
\begin{equation*}
 \left(2-\frac{N_v^n-3}{N_v^n-1}\right)\log N_v
 =\frac{N_v^n+1}{N_v^n-1}\log N_v.
\end{equation*}
The negative sign of $\widetilde f_{\Phi_v,a}$ gives the second subtraction in \eqref{common reference height}.

After the cancellations in \eqref{height comparison holomorphic form}, equations \eqref{comparison first line}, \eqref{comparison second line decomposition} and \eqref{comparison third line} express $R_a$ as $E_a(0,g,\Phi)$ times the height expression in \eqref{common reference height}, plus a multiple of $E_a$ independent of $\sigma$. The independence follows from the common terms of \eqref{split f equals S}, \eqref{inert f equals S}, \eqref{ramified standard odd coefficient}, and the two ramified multiples just computed.

The two families use the same polynomial $P$ in \eqref{common reference ordinary factor}. After taking out the factors outside $S\cup\{w\}$, the coefficient $E_a$ is also given by one fixed, nonzero polynomial, by \eqref{local Whittaker germ} and the choice of the ball at $w$. On each family, polynomial uniqueness makes the coefficient equality an identity between polynomials. Comparing the two identities cancels the common multiple of the Eisenstein polynomial. Since that polynomial is nonzero, the two height expressions agree.
\end{proof}

\subsection{Heights with finitely many nonstandard places}\label{appendix auxiliary heights}

The argument below is a standard method for removing ambiguities supported at finitely many residue characteristics. One first proves the desired identity modulo the span of the logarithms of a finite set of primes, and then repeats the construction with an auxiliary set having disjoint residue characteristics. The linear independence of the logarithms of distinct primes then removes the ambiguity. The same argument is used in \cite[Proof of Theorem~5.2.5]{Qiu}; a closely related separation of local contributions by the linear independence of $\log p$ appears in \cite[Section~14]{Zh2}.

We prove the estimate used in Subsection~\ref{final formula section}. For the numbers $\mathfrak h_m$ defined there, direct subtraction of the explicit expression in Theorem~\ref{Modular height} gives
\begin{equation}\label{numerical height recurrence}
\begin{aligned}
 \mathfrak h_m-\mathfrak h_{m-1}={}&2\frac{L_f'(m+1,\eta^{m+1})}{L_f(m+1,\eta^{m+1})}+2\log|d_F|\\
 &+[F:\QQ]\left(\sum_{i=1}^{m}\frac1i-2\gamma-2\log(2\pi)\right)\\
 &+\begin{cases}
 2\log d_{E/F},&m\text{ even},\\
 \log d_{E/F},&m\text{ odd}.
 \end{cases}
\end{aligned}
\end{equation}

In this subsection $S$ and $T$ are sets of rational primes; a sum over places above them includes every place of $F$ over those primes. The following estimate concerns the actual model metrics and does not require explicit intersection formulas at the additional primes.

\begin{lemma}\label{auxiliary height lemma}
Let $X$ have dimension $m\ge1$, and let $S$ be a finite set of rational primes not dividing $2d_Fd_{E/F}$. Suppose that its model is standard outside $S$. At $S$, allow split maximal models and their specified level pullbacks, or inert self-dual and almost-self-dual maximal models, with rational extensions of the generic Hodge bundle and the fixed archimedean metric. Then
\begin{equation}\label{auxiliary S error}
 h_{\LL}(X)-\mathfrak h_m\in\overline{\QQ}\log S,\qquad
 \overline{\QQ}\log S=\sum_{p\in S}\overline{\QQ}\log p.
\end{equation}
\end{lemma}
\begin{proof}
We use induction on $m$, with the numbers $\mathfrak h_m$ defined in Subsection~\ref{final formula section}. The proof uses the three lines of \eqref{Difference series} and the numerical identity \eqref{numerical height recurrence}. The proof is divided into 4 steps. We first fix a lower-dimensional variety, then compare Fourier coefficients modulo the logarithms of a fixed finite set of primes. Lemma~\ref{Whittaker germ detection}(2) passes this comparison to the constant term.

\noindent\textbf{Step 1.} We choose an ample extension and recall the curve and CM formulas in \cite{Guo2}. By \cite[Lemma~4.4.3(2) and Remark~4.4.7]{Qiu}, choose an ample rational extension $\mathcal M$ of $L$, standard outside $S$, and give it the fixed archimedean metric. Two rational model metrics with these properties differ on a common proper model by rational vertical divisors over $S$. Expanding the difference of their powers shows that their normalized heights differ by an element of $\QQ\log S$. The admissible vertical corrections also have rational coefficients: their defining intersection matrix and degree normalization are rational. It therefore suffices to prove \eqref{auxiliary S error} for $\overline{\mathcal M}$.

For algebraic finite Schwartz data with $\Phi_f(0)=1$, consider
\begin{equation*}
   \left(\wh{\mathcal Z}^{\mathcal M}(g,\Phi)-\frac{\deg_LZ(g,\Phi)}{\deg_LX}\overline{\mathcal M}\right)\cdot\left(\mathcal P-\frac{\overline{\mathcal M}^{m}}{\deg_LX}\right).
\end{equation*}
The superscript denotes the admissible extension with respect to $\mathcal M$. This pairing is holomorphic by \cite[Lemma~4.3.6]{Qiu} and the exact sequence for admissible arithmetic classes. At each cusp, geometric Siegel--Weil makes the two constant terms inside the first parenthesis cancel. Thus the pairing is cuspidal. Adding $\Pr I'$ gives the cuspidal difference series and its decomposition \eqref{Difference series} for these data.

For $m=1$, \cite[Theorem~1.2 and Section~3.3]{Guo2} gives \eqref{auxiliary S error}. For the CM cycle, the argument of \cite[Proposition~3.1]{Guo2} gives
\begin{equation*}
   \overline{\mathcal M}\cdot\mathcal P+2\frac{L_f'(0,\eta)}{L_f(0,\eta)}\in\QQ\log S.
\end{equation*}
For the CM congruence, the CM cycle, generic Hodge bundle, and weights remain fixed while the auxiliary lower-dimensional varieties are chosen twice, as in the proof of \cite[Proposition~3.1]{Guo2}. Prescribing standard data at the first set of new primes makes the second set disjoint from it. Independence of the logarithms removes the new primes, leaving only $S$. We now assume $m\ge2$.

\noindent\textbf{Step 2.} We fix the lower-dimensional variety before varying the Fourier index. Choose a positive vector $y_0\in W^\perp$ by weak approximation so that its lower lattices and metrics are standard at the original bad places outside $S$, including the primes dividing $2d_Fd_{E/F}$. Let $T$ contain the new nonstandard rational primes and the denominators of the comparison maps, with $T$ disjoint from $S$ and from the primes dividing $2d_Fd_{E/F}$. At $S\cup T$, choose maximal lower lattices of the appropriate type. Induction and the metric comparison in Step~1 give their normalized height as $\mathfrak h_{m-1}$ modulo $\overline{\QQ}\log(S\cup T)$. The same congruence holds for the lower bundles used in Definition~\ref{adelic lower Hodge line definition}: after common finite pullback their metrics are standard outside $S\cup T$, and their differences at $S\cup T$ are rational vertical divisors. Thus induction controls the actual lower-dimensional height in this fixed norm class.

Use only Fourier indices in $q(y_0)\Nm(E^\times)$. The generators in \eqref{ramified depth principal ideal} let the valuations vary independently through arbitrarily large values satisfying the prescribed congruences. The corresponding scalar multiples of $y_0$ have the same orthogonal complement, so $S\cup T$ remains fixed. Let $S_0$ be a finite set of places of $F$ containing the ramified places, the places dividing the different, all nonmaximal levels, the places above $S\cup T$, and those where $q(y_0)$ is not a unit. We allow all right translates outside $S_0$. Equation~\eqref{local comparison covariance} and the unipotent and compact-subgroup actions proved with it give the required comparison at those places. We keep $\Phi_f(0)=1$ throughout this constant-term calculation.

\noindent\textbf{Step 3.} We compare the three lines modulo $\overline{\QQ}\log(S\cup T)$. In \eqref{known completed terms}, keep only the local correction terms outside $S\cup T$, and retain the chosen Schwartz functions at $S\cup T$ in every tensor product. Denote the resulting holomorphic form by $K(g,\Phi)$ in this proof. We claim that the coefficients specified in Step~2 of
\begin{equation}\label{auxiliary three line comparison}
\begin{aligned}
 R(g)={}&\mathcal D(g,\Phi)-K(g,\Phi)\\
 &-\left(\mathfrak h_m-h_{\overline{\mathcal M}}(X)+\log|d_F|+\log d_{E/F}
 -2\sum_{v\mid S\cup T}
     \frac{L'_v}{L_v}(m+1,\eta_v^{m+1})\right)E(0,g,\Phi)
\end{aligned}
\end{equation}
belong to $\overline{\QQ}\log(S\cup T)$.

For the first line, use the comparison in \cite[Section~5]{Guo2} with all CM components and their original weights. At $S\cup T$, proper intersection lengths, orders of comparison maps for containing divisors, and admissible corrections are rational multiples of $\log N_v$. The local Whittaker derivatives are algebraic multiples of $\log N_v$, as derivatives of rational functions of $N_v^{-s}$. The remaining factors in the holomorphic Fourier coefficients are algebraic, by \cite[Proposition~4.21 and Theorem~7.8]{Sil}. When a local Whittaker value is zero, compare with a section whose $a$-Whittaker value is nonzero to determine the remaining factor.

The first line also contains $A(g,\Phi)$ from \eqref{comparison first line}. After the containing terms are combined as in \cite[Section~5.4]{Guo2}, their coefficient at the standard places is
\begin{equation*}
 -2\frac{L_f'(0,\eta)}{L_f(0,\eta)}
                       -\overline{\mathcal M}\cdot\mathcal P
 \quad\bmod\ \overline{\QQ}\log(S\cup T).
\end{equation*}
Step~1 puts this in $\QQ\log S$. The remaining local coefficients at $S\cup T$ are algebraic multiples of the local logarithms. Thus the specified Fourier coefficients of $A(g,\Phi)$ lie in the required space as well.

For the second line, use \eqref{comparison second line decomposition} and the congruence for the lower-dimensional height from Step~2. The differences of model metrics at $S\cup T$ are vertical, while the standard formulas apply elsewhere. The admissible vertical correction has zero Hodge intersection, and the Green function has zero mean. The third line follows from geometric Siegel--Weil and the CM congruence in Step~1. Adding the three lines with the constants of \eqref{Explicit difference series} and \eqref{numerical height recurrence} proves the claim. The sum over $S\cup T$ in \eqref{auxiliary three line comparison} subtracts precisely the local logarithmic derivatives omitted from the standard comparison.

\noindent\textbf{Step 4.} We take the constant term and remove the primes in $T$. All primes dividing the different or ramified in $E/F$ lie outside $S\cup T$. Hence the zero values used in \eqref{Equation to 0} give
\begin{equation*}
  \operatorname{CT}(K)=-\log|d_F|-\log d_{E/F}.
\end{equation*}
Also $\operatorname{CT}(\mathcal D)=0$, while \eqref{holomorphic intertwining vanishing} and $\Phi_f(0)=1$ give $\operatorname{CT}(E)=1$. Therefore
\begin{equation*}
 \operatorname{CT}(R)=h_{\overline{\mathcal M}}(X)-\mathfrak h_m+2\sum_{v\mid S\cup T}\frac{L'_v}{L_v}(m+1,\eta_v^{m+1}).
\end{equation*}
Each local summand is a rational multiple of $\log N_v$. By Lemma~\ref{Whittaker germ detection}(2), this constant term is an algebraic linear combination of the coefficients in Step~3. Thus
\begin{equation*}
   h_{\overline{\mathcal M}}(X)-\mathfrak h_m\in\overline{\QQ}\log(S\cup T),
\end{equation*}
and Step~1 gives the same conclusion for $\LL$.

We keep the ambient model and $\mathcal M$ fixed and choose a second $y_0$ whose lower lattices are standard also at $T$. Its new nonstandard primes form a set $T'$ disjoint from $S\cup T$. The same argument puts $h_{\LL}(X)-\mathfrak h_m$ in $\overline{\QQ}\log(S\cup T')$. By Baker's theorem on logarithms of multiplicatively independent algebraic numbers \cite{Wal},
\begin{equation*}
 \overline{\QQ}\log(S\cup T)\cap\overline{\QQ}\log(S\cup T')=\overline{\QQ}\log S.
\end{equation*}
This proves \eqref{auxiliary S error} and completes the induction.
\end{proof}

\

\noindent \small{ School of Mathematical Sciences, Peking University, Beijing 100871, China}

\noindent \small{\it Email: ziqiguo0603@pku.edu.cn}

\end{document}